\RequirePackage{fix-cm}
\documentclass[twocolumn]{svjour3}          
\smartqed  

\usepackage{amsmath}
\usepackage{amssymb}
\usepackage{amsthm}
\usepackage{graphicx}
\usepackage{mathrsfs}
\usepackage{hyperref}
\usepackage{color}
\usepackage{multirow}
\usepackage{array}
\usepackage{booktabs}
\usepackage{enumitem}   
\usepackage[misc,geometry]{ifsym}

\DeclareMathOperator\Span{span}
\DeclareMathOperator\spect{Spect}

\theoremstyle{definition}

\newtheorem{proposition}[theorem]{Proposition}

\begin{document}

\title{Nonlinear analysis of forced mechanical systems with internal resonance using spectral submanifolds, Part II: Bifurcation and quasi-periodic response}

\titlerunning{Nonlinear analysis of systems with internal resonance using spectral submanifolds -- Part II}

\author{Mingwu Li         \and
        George Haller
}


\institute{M. Li (\Letter)\and G. Haller \at
              Institute for Mechanical Systems, ETH Z\"{u}rich \\
              Leonhardstrasse 21, 8092 Zurich, Switzerland\\
              \email{mingwli@ethz.ch}           
}

\date{Received: date / Accepted: date}

\maketitle

\begin{sloppypar}
\begin{abstract}
In Part I of this paper, we have used spectral submanifold (SSM) theory to construct reduced-order models for harmonically excited mechanical systems with internal resonances. In that setting, extracting forced response curves formed by periodic orbits of the full system was reduced to locating the solution branches of equilibria of the corresponding reduced-order model. Here we use bifurcations of the equilibria of the reduced-order model to predict bifurcations of the periodic response of the full system. Specifically, we identify Hopf bifurcations of equilibria and limit cycles in reduced models on SSMs to predict the existence of two-dimensional and three-dimensional quasi-periodic attractors and repellers in periodically forced mechanical systems of arbitrary dimension. We illustrate the accuracy and efficiency of these computations on finite-element models of beams and plates.
\end{abstract}
\end{sloppypar}

\keywords{Invariant manifolds \and Reduced-order models \and  Spectral submanifolds \and Internal resonances \and Bifurcation}

\section{Introduction}
\begin{sloppypar}
Periodic response of harmonically excited nonlinear mechanical systems is
commonly observed in experiments and numerics. As the amplitude or frequency of the harmonic excitation varies, however, a stable periodic response may become unstable via bifurcations. Three common bifurcations of periodic orbits are saddle-node, period-doubling and torus bifurcations. The last one is also referred to as Neimark-Sacker bifurcation when discussed via an appropriate Poinc\'are map~\cite{kuznetsov2013elements}. In this bifurcation, a unique two-dimensional invariant torus is born out of a periodic orbit for each fixed excitation frequency and amplitude. The quasi-periodic response associated with such an invariant torus has been widely observed in mechanical systems under harmonic excitation, ranging from simple van der Pol oscillator~\cite{holmes1978bifurcations} to more complicated systems~\cite{kim1990bifurcation,detroux2015performance,chang1993non,thomas2005non,shaw2016periodic,huang2019quasi,fontanela2019computation}. 

Torus bifurcations of periodic orbits are also commonly observed in harmonically excited mechanical systems with internal resonance. Such quasi-periodic responses have been reported in beams with a 3:1 internal resonance~\cite{shaw2016periodic}, plates with a 1:1 internal resonance~\cite{chang1993non}, shells with a 1:1:2 internal resonance~\cite{thomas2005non} and structures with a cyclic symmetry that creates a 1:1 internal resonance~\cite{fontanela2019computation}.

The quasi-periodic orbits created in these bifurcations can be calculated with various numerical methods. The simplest is direct numerical integration, which is applicable if the quasi-periodic torus is asymptotically stable. However, it is difficult to determine when the steady state is reached and the method is inapplicable if the quasi-periodic torus is of saddle type. In the shooting method~\cite{kaas1985computation,kim1996quasi}, a quasi-periodic orbit is formulated as a fixed point of a second-order Poincar\'e map, which is constructed with Poincar\'e points that satisfy some conditions. The stability and bifurcation characteristics for the calculated quasi-periodic orbit can be inferred from the spectrum of that Poincar\'e map. However, this method requires repeated numerical integrations. 

As an alternative, harmonic balance methods have also been used to calculate quasi-periodic orbits~\cite{lau1983incremental,kim1996quasi-HB,guskov2012harmonic,ju2017modified,liao2020continuation}. In the \emph{classic} harmonic balance method, an unknown quasi-periodic orbit is approximated with a truncated Fourier series of multiple time scales. The equation of motion is approximately satisfied with a Galerkin projection, yielding a set of nonlinear algebraic equations for the coefficients of the series. Then a linearization of the resulting nonlinear algebraic equations is performed to locate their roots iteratively~\cite{kim1996quasi-HB}. In the \emph{incremental} harmonic balance method, the linearization is first applied to the nonlinear equation of motion (in the form of differential equations), and the Galerkin projection is then applied to the linearized equation of motion~\cite{lau1983incremental,ju2017modified}. The alternating frequency/time technique has been used to speed up the computation of the projection~\cite{kim1996quasi-HB}. In addition, Broyden's method~\cite{ju2017modified} and optimization techniques~\cite{liao2020continuation} have been used to reduce the computational cost of the evaluation of Jacobians in the iterations. Recently, an integral equation approach has also been proposed for the fast computation of quasi-periodic orbits of systems with quasi-periodic forcing~\cite{jain2019fast}.

A quasi-periodic orbit is contained in a torus densely filled by infinitely many quasi-periodic orbits~\cite{broer2009quasi}. Such a torus is generically referred to as a \emph{quasi-periodic invariant torus}~\cite{schilder2005continuation,schilder2006fourier}, or simply, an invariant torus. This invariant torus is the solution to a boundary-value problem composed of a partial differential equation (PDE) and appropriate phase conditions~\cite{schilder2005continuation,schilder2006fourier}. Both \emph{semi-discretization} and \emph{full-discretization} schemes have been developed to solve the boundary-value problem. In the full-discretization method, the PDE is discretized via a finite difference method~\cite{schilder2005continuation}, a Galerkin projection approach~\cite{schilder2006fourier}, or a spectral collocation scheme~\cite{roose2007continuation,Kunt}. In the semi-discretization method, the unknown torus is expressed as a truncated Fourier series with unknown periodic coefficients, and the PDE is approximated as a set of ordinary differential equations (ODEs) for those unknown coefficients. The coefficients are then solved for as periodic orbits of the ODEs with some phase conditions~\cite{schilder2006fourier}. In particular, such periodic solutions can be obtained with various numerical methods, as detailed in the introduction of Part I of this paper~\cite{li2021nonlinear}. 

An alternative semi-discretization method was discussed in~\cite{dankowicz2013recipes}, where the unknown periodic coefficients are expressed in terms of multiple trajectories coupled through boundary conditions. Remarkably, all these trajectories share the same form of ODEs, enabling object-oriented problem constructions. Based on this construction paradigm, a general-purpose toolbox for the parameter continuation of two-dimensional invariant tori,~\texttt{Tor}, has been developed recently~\cite{li2020tor}.

The above numerical methods are effective for low-dimensional systems but impractical for high-dimensional systems, e.g., finite element models. For the latter, one needs to construct reduced-order models in order to calculate the quasi-periodic response efficiently. Invariant manifolds provide a powerful tool for constructing such reduced-order models~\cite{pesheck2002new,jiang2005construction,touze2006nonlinear,haller2016nonlinear}. The invariant manifolds envisioned as nonlinear continuations of linear modal subspaces are often called \emph{nonlinear normal modes} (NNMs)~\cite{shaw1993normal}. Galerkin approaches have been used to calculate the NNMs and derive reduced-order models to extract the free and forced response~\cite{pesheck2002new,jiang2005construction,jiang2005nonlinear}. The method of normal form has also been used to derive reduced-order models on NNMs~\cite{touze2006nonlinear,amabili2007reduced,liu2019simultaneous,vizzaccaro2021direct,touze2021model}. Touz{\'e} and coworkers derived explicit third-order reduced-order models for systems with quadratic and cubic nonlinearities~\cite{touze2006nonlinear,vizzaccaro2021direct,touze2021model}. They used the reduced-order models to calculate the backbone and forced response curves of full systems~\cite{touze2006nonlinear,amabili2007reduced,vizzaccaro2021direct,touze2021model}.

An alternative reduction method based on invariant manifolds was recently proposed by Haller and Ponsioen~\cite{haller2016nonlinear}. This method uses spectral submanifold (SSM) theory to reduce the full dynamics to exact invariant SSM surfaces in the phase space. An SSM is the unique smoothest nonlinear continuation of a spectral subspace of the linear part of the dynamical system, which exists under low-order non-resonance conditions. For systems without internal resonance, the SSM tangent to a linear mode of interest is two-dimensional and the reduced-order model on the resonant SSM enables explicit calculation of the backbone and forced response curves around the mode~\cite{szalai2017nonlinear,breunung2018explicit,ponsioen2019analytic,ponsioen2020model}. A numerical procedure for computing arbitrary dimensional SSM up to any polynomial order of approximation has been developed very recently by Jain and Haller~\cite{SHOBHIT}. This procedure has been implemented in the open-source MATLAB package, SSMTool-2.0~\cite{ssmtool2}. Based on the work~\cite{SHOBHIT}, Part I of this paper~\cite{li2021nonlinear} investigated SSM reduction for systems with internal resonances and demonstrated that the forced response curve of the full system can be efficiently obtained from the solution branch of fixed points in the corresponding reduced-order model on SSMs.

The studies mentioned so far have mostly focused on the computation of backbone and forced response curves composed of periodic orbits. As an exception, Amabili and Touz{\'e}~\cite{amabili2007reduced} also discuss quasi-periodic orbits. Specifically, they derive reduced-order models for fluid-filled circular cylindrical shells subject to harmonic excitation. They observe torus bifurcations of periodic orbits in the continuation of periodic orbits of the reduced-order models. They then perform numerical integration to obtain quasi-periodic response. 

Here we focus on the computation and bifurcation analysis of quasi-periodic orbits using SSM reduction. As we will see, the two-dimensional invariant tori containing the quasi-periodic responses of the full system can be obtained from the limit cycles of the corresponding reduced-order model, and the stability types of these quasi-periodic orbits can be inferred from that of the corresponding limit cycles. In addition, we perform bifurcation analysis of the quasi-periodic orbits based on the bifurcation of the corresponding limit cycles in the reduced models on the SSMs. Finally, we calculate the bifurcating three-dimensional invariant tori of the full system from the corresponding two-dimensional invariant tori of the reduced-order model.

The rest of this paper is organized as follows. In the next section, we review the definition of SSMs and their fundamental properties. In section~\ref{sec:theos}, we summarize from Part I the reduced-order models obtained on resonant SSMs. We then show how the bifurcations of periodic orbits in the full system are related to the bifurcations of the equilibria of the corresponding reduced-order models on resonant SSMs. We move on to the computation and bifurcation analysis of quasi-periodic orbits of the full system, including two- and three-dimensional invariant tori, in section~\ref{sec:tori}. Section~\ref{sec:implementation} describes two numerical toolboxes, namely, \texttt{SSM-po} and \texttt{SSM-tor}, for the computation of two- and three-dimensional invariant tori, respectively. In section~\ref{sec:example}, we illustrate the power of our resonant-SSM-based approaches on several examples, including finite element models of a Bernoulli beam with a nonlinear support spring, a von K\'arm\'an beam and a von K\'arm\'an plate with distributed nonlinearities. A summary and several directions for future research are presented in the concluding section~\ref{sec:conclusion}.

\end{sloppypar}

\section{Spectral submanifold}
\subsection{System setup}
\begin{sloppypar}
Consider a first-order dynamical system subject to periodic forcing of the form
\begin{equation}
\label{eq:full-first}
\boldsymbol{B}\dot{\boldsymbol{z}}	=\boldsymbol{A}\boldsymbol{z}+\boldsymbol{F}(\boldsymbol{z})+\epsilon\boldsymbol{F}^{\mathrm{ext}}({\Omega t}),\quad 0<\epsilon\ll1,
\end{equation}
where $\boldsymbol{z}\in\mathbb{R}^N$, $\boldsymbol{A}, \boldsymbol{B}\in\mathbb{R}^{N\times N}$, $\boldsymbol{F}(\boldsymbol{z})\sim\mathcal{O}(|\boldsymbol{z}|^2)$, $\Omega\in\mathbb{R}$ and $t\in\mathbb{R}_+$.
As a special case, the equation of motion of a harmonically excited mechanical system in the second-order form,
\begin{equation}
\label{eq:eom-second-full}
\boldsymbol{M}\ddot{\boldsymbol{x}}+\boldsymbol{C}\dot{\boldsymbol{x}}+\boldsymbol{K}\boldsymbol{x}+\boldsymbol{f}(\boldsymbol{x},\dot{\boldsymbol{x}})=\epsilon \boldsymbol{f}^{\mathrm{ext}}(\Omega t),
\end{equation}
can be recast into the first-order form~\eqref{eq:full-first} by letting
\begin{gather}
\boldsymbol{z}=\begin{pmatrix}\boldsymbol{x}\\\dot{\boldsymbol{x}}\end{pmatrix},\,\,
\boldsymbol{A}=\begin{pmatrix}-\boldsymbol{K} 
& \boldsymbol{0}\\\boldsymbol{0} & \boldsymbol{M}\end{pmatrix},\,\,
\boldsymbol{B}=\begin{pmatrix}\boldsymbol{C} 
& \boldsymbol{M}\\\boldsymbol{M} & \boldsymbol{0}\end{pmatrix},\nonumber\\
\boldsymbol{F}(\boldsymbol{z})=\begin{pmatrix}{-\boldsymbol{f}(\boldsymbol{x},\dot{\boldsymbol{x}})}\\\boldsymbol{0}\end{pmatrix},\,\,
\boldsymbol{F}^{\mathrm{ext}}(\Omega t) = \begin{pmatrix}\boldsymbol{f}^{\mathrm{ext}}(\Omega t)\\\boldsymbol{0}\end{pmatrix},
\end{gather}
where $\boldsymbol{x}\in\mathbb{R}^n$ is generalized displacement vector; $\boldsymbol{M}, \boldsymbol{C},\boldsymbol{K}\in\mathbb{R}^{n\times n}$ are mass, damping and stiffness matrices, respectively; and $\boldsymbol{f}(\boldsymbol{x},\dot{\boldsymbol{x}})$ denotes the nonlinear internal force vector which is of class $C^r$ and satisfies $\boldsymbol{f}(\boldsymbol{x},\dot{\boldsymbol{x}})\sim \mathcal{O}(|\boldsymbol{x}|^2,|\boldsymbol{x}||\dot{\boldsymbol{x}}|,|\dot{\boldsymbol{x}}|^2)$.
\end{sloppypar}

The analysis of the linear part of~\eqref{eq:full-first} leads to generalized eigenvalue problems for the matrix pair $(\boldsymbol{A},\boldsymbol{B})$ in the form
\begin{equation}
\boldsymbol{A}\boldsymbol{v}_j=\lambda_j\boldsymbol{B}\boldsymbol{v}_j,\quad \boldsymbol{u}_j^\ast \boldsymbol{A}=\lambda_j \boldsymbol{u}_j^\ast \boldsymbol{B},
\end{equation}
where $\lambda_j$ is a generalized eigenvalue and $\boldsymbol{v}_j$ and $\boldsymbol{u}_j$ are the corresponding \emph{right} and \emph{left} eigenvectors, respectively. We assume that the real parts of all eigenvalues are strictly less than zero
\begin{equation}
    \mathrm{Re}(\lambda_j)<0,\quad \forall j\in\{1,\cdots,N\}
\end{equation}
and hence the trivial equilibrium of the linearized system, $\boldsymbol{B}\dot{\boldsymbol{z}}=\boldsymbol{A}\boldsymbol{z}$, is asymptotically stable. In addition, we assume that for a $2m$-dimensional \emph{master} underdamped modal subspace
\begin{equation}
\mathcal{E}=\Span\{\boldsymbol{v}^\mathcal{E}_1,\bar{\boldsymbol{v}}^\mathcal{E}_1,\cdots,\boldsymbol{v}^\mathcal{E}_m,\bar{\boldsymbol{v}}^\mathcal{E}_m\},
\end{equation}
the complex eigenvalues in the spectrum of $\mathcal{E}$ satisfy an approximate \emph{inner} or \emph{internal} resonance relationship of the form
\begin{equation}
\label{eq:res-inner}
    \lambda_i^\mathcal{E}\approx\boldsymbol{l}\cdot\boldsymbol{\lambda}^\mathcal{E}+\boldsymbol{j}\cdot\bar{\boldsymbol{\lambda}}^\mathcal{E},\quad \bar{\lambda}_i^\mathcal{E}\approx\boldsymbol{j}\cdot\boldsymbol{\lambda}^\mathcal{E}+\boldsymbol{l}\cdot\bar{\boldsymbol{\lambda}}^\mathcal{E}
\end{equation}
for some $i\in\{1,\cdots,m\}$, where $\boldsymbol{l},\boldsymbol{j}\in\mathbb{N}_0^m$ (the subscript 0 here emphasizes that zero is included), $|\boldsymbol{l}+\boldsymbol{j}|:=\sum_{k=1}^m (l_k+j_k)\geq2$, and $\boldsymbol{\lambda}_\mathcal{E}=(\lambda^\mathcal{E}_1,\cdots,\lambda^\mathcal{E}_m)$. As an example, consider a weakly damped system with an approximate 1:3 internal resonance between the first two modes, namely, $\omega_2\approx3\omega_1$, where $\omega_{1}$ and $\omega_2$ are the undamped {natural} frequencies of the first and second modes. Under the assumed weak damping, the first two modes have complex conjugate pairs of eigenvalues. If $\lambda_{i}$ and $\bar{\lambda}_i$ denote the eigenvalues of the $i$-th complex conjugate pair of modes, then the assumed internal resonance can be written in the form~\eqref{eq:res-inner} with
\begin{gather}
\lambda_2\approx(3+j)\lambda_1+j\bar{\lambda}_1+l\lambda_2+l\bar{\lambda}_2,\nonumber\\
\lambda_2\approx(1+j)\lambda_2+j\bar{\lambda}_2+l\lambda_1+l\bar{\lambda}_1,\nonumber\\
\lambda_1\approx(1+j)\lambda_1+j\bar{\lambda}_1+l\lambda_2+l\bar{\lambda}_2
\end{gather}
for all $j,l\in\mathbb{N}_0$. 

The check of internal resonances in~\eqref{eq:res-inner} is based on the eigenvalues of the linear part of the nonlinear system~\eqref{eq:full-first}. This is different from the internal resonances between nonlinear frequencies in conservative systems~\cite{kerschen2009nonlinearI}. Modal interactions resulted from the internal resonances between nonlinear frequencies can be observed in conservative backbone curves composed of nonlinear normal modes~\cite{kerschen2009nonlinearI}. However, the persistence of these interacted backbone curves under the addition of damping and excitation is not clear. In contrast, this paper focuses on forced responses of damped systems. Provided that the backbone curves with nonlinear modal interactions are persistent, they are generally observed in systems operate at large response amplitudes~\cite{kerschen2009nonlinearI}, which result from large excitation levels. We have assumed small forcing amplitudes (see~\eqref{eq:full-first}) to adapt the theory of spectral submanifolds. Therefore, the scenario of internal resonances between nonlinear frequencies is outside the scope of this paper.

\subsection{SSM definition and reduction}
As discussed in Part I, under assumption~\eqref{eq:res-inner}, the full system~\eqref{eq:full-first} can be reduced to a low-dimensional system on an SSM that captures the dynamics of the full system. Specifically, following~\cite{haller2016nonlinear} we look for
a non-autonomous \emph{spectral submanifold} (SSM) with period ${2\pi}/{\Omega}$, $\mathcal{W}(\mathcal{E},\Omega t)$, corresponding to the master spectral subspace $\mathcal{E}$, as a $2m$-dimensional \emph{invariant} manifold to the nonlinear system~\eqref{eq:full-first} such that $\mathcal{W}(\mathcal{E},\Omega t)$
\begin{enumerate}[label=(\roman*)]
\item perturbs smoothly from $\mathcal{E}$ at the trivial equilibrium $\boldsymbol{z}=0$ under the addition of nonlinear terms and external forcing in~\eqref{eq:full-first},
\item is strictly smoother than any other periodic invariant manifolds with period ${2\pi}/{\Omega}$ that satisfy (i).
\end{enumerate}

\begin{sloppypar}
The existence and uniqueness of such SSMs have been proved in~\cite{haller2016nonlinear} in a more general setting, based on the results of Cabr\'e et al.~\cite{cabre2003parameterization-i,cabre2003parameterization-ii,cabre2005parameterization-iii}. Here we summarize the main results in the following theorem, which has already been stated in Part I but will be included here for completeness. 
\begin{theorem}
\label{th:SSM-existence-uniqueness}
Let $\spect(\mathcal{E}) = \{\lambda^\mathcal{E}_1,\bar{\lambda}^\mathcal{E}_1,\cdots,\lambda^\mathcal{E}_m,\bar{\lambda}^\mathcal{E}_m\}$ and define  $\spect(\boldsymbol{\Lambda})=\{\lambda_1,\cdots,\lambda_{2n}\}$.
Under the non-resonance condition
\begin{gather}
\boldsymbol{a}\cdot\mathrm{Re}(\boldsymbol{\lambda}^\mathcal{E})+\boldsymbol{b}\cdot\mathrm{Re}(\bar{\boldsymbol{\lambda}}^\mathcal{E})\neq \mathrm{Re}(\lambda_k),\nonumber\\
\forall\,\,\lambda_k\in\spect(\boldsymbol{\Lambda})\setminus\spect(\mathcal{E}),\nonumber\\
\forall\,\,\boldsymbol{a},\boldsymbol{b}\in\mathbb{N}_0^m,\,\,2\leq |\boldsymbol{a}+\boldsymbol{b}|\leq\Sigma(\mathcal{E}),
\end{gather}
where the \emph{absolute spectral quotient} $\Sigma(\mathcal{E})$ of $\mathcal{E}$ is defined as
\begin{equation}
\Sigma(\mathcal{E}) = \mathrm{Int}\left(\frac{\min_{\lambda\in\spect(\boldsymbol{\Lambda})}\mathrm{Re}\lambda}{\max_{\lambda\in\spect(\mathcal{E})}\mathrm{Re}\lambda}\right),
\end{equation}
the following hold for system~\eqref{eq:full-first}:
\begin{enumerate}[label=(\roman*)]
\item There exists a unique $2m$-dimensional, time-periodic SSM of class $C^{\Sigma(\mathcal{E})+1}$, $\mathcal{W}(\mathcal{E},\Omega t)\subset\mathbb{R}^{2m}$ that depends smoothly on the parameter $\epsilon$.
\item $\mathcal{W}(\mathcal{E},\Omega t)$ can be viewed as an embedding of an open set $\mathcal{U}$ into the phase space of system~\eqref{eq:full-first} via a map
\begin{equation}
\boldsymbol{W}_{\epsilon}(\boldsymbol{p},\phi):\mathcal{U}=U\times{S}^1\to\mathbb{R}^{2n},\quad U\subset \mathbb{C}^{2m}.
\end{equation}
\item There exists a polynomial function $\boldsymbol{R}_\epsilon(\boldsymbol{p},\phi):\mathcal{U}\to{U}$ satisfying the invariance equation
\begin{align}
\label{eq:invariance}
& \boldsymbol{B}\left({D}_{\boldsymbol{p}}\boldsymbol{W}_{\epsilon}(\boldsymbol{p},\phi) \boldsymbol{R}_{\epsilon}(\boldsymbol{p},\phi)+{D}_{\phi}\boldsymbol{W}_{\epsilon}(\boldsymbol{p},\phi) \Omega\right)\nonumber\\
&=\boldsymbol{A}\boldsymbol{W}_{\epsilon}(\boldsymbol{p},\phi)+\boldsymbol{F}( \boldsymbol{W}_{\epsilon}(\boldsymbol{p},\phi))+\epsilon\boldsymbol{F}^{\mathrm{ext}}({\phi}),
\end{align}
where $D_{\boldsymbol{p}}$ and $D_\phi$ denote the partial derivatives with respect to $\boldsymbol{p}$ and $\phi$, such that the reduced dynamics on the SSM $\mathcal{W}(\mathcal{E},\Omega t)$ can be expressed as
\begin{equation}
\label{eq:red-dyn}
\dot{\boldsymbol{p}} = \boldsymbol{R}_\epsilon(\boldsymbol{p},\phi),\quad \dot{\phi}=\Omega.
\end{equation}
\end{enumerate}
\end{theorem}

Although there are infinitely many time-periodic invariant manifolds that are tangent to $\mathcal{E}$ at the origin, the above theorem indicates that there is a unique SSM, which is the smoothest one among all these time-periodic invariant manifolds~\cite{haller2016nonlinear}. The smoothness of the SSM is determined by the absolute spectral quotient $\Sigma(\mathcal{E})$, which can be computed a priori from the spectrum of the linear part of the system. Importantly, $\Sigma(\mathcal{E})$ only depends on the real parts of the eigenvalues. This spectral quotient is, however, generally very large in practical applications~\cite{touze2021model}. Therefore, most computations generally lead only to approximations of the SSM.

The are two essential ingredients for the SSM: the map $\boldsymbol{W}_\epsilon$ and the vector field $\boldsymbol{R}_\epsilon$. The latter determines the reduced dynamics on the time-periodic $2m$-dimensional SSM, and the former maps any trajectories in reduced coordinates $\boldsymbol{p}$ to the corresponding trajectories of the full system whose phase space is of dimension $2n$. Since we generally have $m\ll n$, the SSM and its associated reduced dynamics are naturally used for model reduction. Next we move to the computations of these two ingredients.

\subsection{Computation of the resonant SSM}
We have discussed the computation of $\boldsymbol{W}_\epsilon(\boldsymbol{p},\phi)$ and $\boldsymbol{R}_\epsilon(\boldsymbol{p},\phi)$ in detail in Part I~\cite{li2021nonlinear}, but briefly recall them in this section. Let
\begin{equation}
\label{eq:p-to-qs}
\boldsymbol{p}=(q_1,\bar{q}_1,\cdots,q_m,\bar{q}_m)
\end{equation}
where $q_i$ and $\bar{q}_i$ denote the \emph{parameterization} (also referred to as \emph{normal}) coordinates corresponding to $\boldsymbol{v}_i^{\mathcal{E}}$ and $\bar{\boldsymbol{v}}_i^{\mathcal{E}}$, respectively, we have
\begin{equation}
\label{eq:red-auto-block}
    \boldsymbol{R}_\epsilon(\boldsymbol{p},\phi)=\begin{pmatrix}\boldsymbol{R}_{\epsilon,1}(\boldsymbol{p},\phi)\\\vdots\\\boldsymbol{R}_{\epsilon,m}(\boldsymbol{p},\phi)\end{pmatrix}
\end{equation}
where $\boldsymbol{R}_{\epsilon,i}(\boldsymbol{p},\phi)\in\mathbb{C}^2$ contains a complex conjugate pair components of the vector field associated with the $i$-th pair of master modes $(\boldsymbol{v}_i^{\mathcal{E}},\bar{\boldsymbol{v}}_i^{\mathcal{E}})$.

With a Taylor expansion in $\boldsymbol{p}$, the leading-order approximation to the non-autonomous part of SSM is of the form
\begin{equation}
\label{eq:red-truncation}
\boldsymbol{R}_{\epsilon,i}(\boldsymbol{p},\phi)=\boldsymbol{R}_{i}(\boldsymbol{p})+\epsilon \boldsymbol{S}_{{0},i}(\phi)+\mathcal{O}(\epsilon|\boldsymbol{p}|)
\end{equation}
for $i=1,\cdots,m$, where the autonomous part $\boldsymbol{R}_{i}(\boldsymbol{p})$ is given by
\begin{equation}
\label{eq:auto-ssm-red}
    \boldsymbol{R}_{i}(\boldsymbol{p})=\begin{pmatrix}\lambda_i^{\mathcal{E}}q_i\\\bar{\lambda}_i^{\mathcal{E}}\bar{q}_i\end{pmatrix}+\sum_{(\boldsymbol{l},\boldsymbol{j})\in\mathcal{R}_i}\begin{pmatrix}\gamma(\boldsymbol{l},\boldsymbol{j})\boldsymbol{q}^{\boldsymbol{l}}\bar{\boldsymbol{q}}^{\boldsymbol{j}}\\\bar{\gamma}(\boldsymbol{l},\boldsymbol{j})\boldsymbol{q}^{\boldsymbol{j}}\bar{\boldsymbol{q}}^{\boldsymbol{l}}\end{pmatrix},
\end{equation}
with 
\begin{equation}
\label{eq:def-Ri}
\mathcal{R}_i=\{(\boldsymbol{l},\boldsymbol{j}): \lambda_i^\mathcal{E}\approx\boldsymbol{l}\cdot\boldsymbol{\lambda}^\mathcal{E}+\boldsymbol{j}\cdot\bar{\boldsymbol{\lambda}}^\mathcal{E}\}.
\end{equation}
The non-autonomous part, $\boldsymbol{S}_{{0},i}(\phi)$, in~\eqref{eq:red-truncation} is given by
\begin{equation}
\label{eq:auto-ssm-vec}
\boldsymbol{S}_{{0},i}(\phi)=\begin{pmatrix}{{S}}_{{0},i}e^{\mathrm{i}\phi}\\\bar{{S}}_{{0},i}e^{-\mathrm{i}\phi}\end{pmatrix}
\end{equation}
with
\begin{equation}
{{S}}_{{0},i}=\left\lbrace \begin{array}{cl}
      (\boldsymbol{u}_i^{\mathcal{E}})^\ast\boldsymbol{F}^{\mathrm{a}}   & \text{if} \hspace{2mm} \lambda_i^{\mathcal{E}}\approx\mathrm{i}\Omega \\
      {0}   & \text{otherwise}
    \end{array} \right..
\end{equation}
The subscript ${0}$ in $\boldsymbol{S}_{{0},i}$ and ${S}_{{0},i}$ denotes the zeroth-order (or leading-order) approximation of the non-autonomous part of the reduced dynamics on SSM.
Here ${\boldsymbol{F}}^{\mathrm{a}}\in\mathbb{R}^{2n}$, related to the external forcing $\boldsymbol{F}^{\mathrm{ext}}(\phi)$ via
\begin{equation}
\label{eq:forcing-conj}
\boldsymbol{F}^{\mathrm{ext}}(\phi)={\boldsymbol{F}}^{\mathrm{a}}e^{\mathrm{i}\phi}+{\boldsymbol{F}^{\mathrm{a}}}e^{-\mathrm{i}\phi}.
\end{equation}
The superscript `a' represents the amplitude of the external harmonic forcing.
Accordingly, the embedding map $\boldsymbol{W}_\epsilon(\boldsymbol{p},\phi)$ is decomposed as the sum of autonomous and non-autonomous parts
\begin{equation}
\label{eq:ssm-decomp}
\boldsymbol{W}_\epsilon(\boldsymbol{p},\phi)=\boldsymbol{W}(\boldsymbol{p})+\epsilon\boldsymbol{X}_{{0}}(\phi)+\mathcal{O}(\epsilon|\boldsymbol{p}|).
\end{equation}
Consistently, the subscript $0$ denotes the zeroth-order (or leading-order) approximation of the non-autonomous part of SSM.
The coefficients $\gamma(\boldsymbol{l},\boldsymbol{j})$ in~\eqref{eq:auto-ssm-red} are calculated along with the expansion coefficients for $\boldsymbol{W}_\epsilon(\boldsymbol{p},\phi)$. We refer the reader to Part I and~\cite{SHOBHIT} for more detail on these calculations.
\end{sloppypar}

\section{Bifurcation of periodic orbits}
\label{sec:theos}
As discussed in Part I, when the excitation frequency $\Omega$ is well separated from all natural frequencies of system~\eqref{eq:full-first} for $\epsilon=0$, then the periodic response of the full system at leading order in $\epsilon$ is given by
\begin{equation}
 \boldsymbol{z}(t)=-2\epsilon\mathrm{Re}\left((\boldsymbol{A}-\mathrm{i}\Omega\boldsymbol{B})^{-1}\boldsymbol{F}^{\mathrm{a}}e^{\mathrm{i}\Omega t}\right),
\end{equation}
which is always stable, for $\epsilon>0$ small enough.

When the system is subject to external resonance with the forcing, the response amplitude becomes large. Therefore the nonlinear terms in the equation of motion play an essential role, leading possibly to a bifurcation of periodic orbits. Indeed, as observed in the examples of Part I, both stable and unstable periodic orbits may arise when $\Omega$ is near the natural frequencies. Here, we are interested in the responses of systems subject to an external resonance in addition to the internal resonance assumed in~\eqref{eq:res-inner}. In particular, we assume that the forcing frequency $\Omega$ is resonant with the master eigenvalues in the following way:
\begin{equation}
\label{eq:res-forcing}
    \boldsymbol{\lambda}^{\mathcal{E}}-\mathrm{i}\boldsymbol{r}\Omega\approx0, \,\,\boldsymbol{r}\in\mathbb{Q}^m.
\end{equation}

\begin{sloppypar}
\begin{theorem}
\label{theo-po}
Under the internal resonance condition~\eqref{eq:res-inner} and the external resonance condition~\eqref{eq:res-forcing}, the following statements hold for the resonant SSM $\mathcal{W}(\mathcal{E},\Omega t)$ for $\epsilon>0$ small enough:
\begin{enumerate}[label=(\roman*)]
\item \emph{[Polar reduced dynamics]} Rewriting the parameterization~\eqref{eq:p-to-qs} in the time-periodic, polar form
\begin{equation}
\label{eq:polar-form}
    q_i=\rho_ie^{\mathrm{i}(\theta_i+r_i\Omega t)},\,\,\bar{q}_i=\rho_ie^{-\mathrm{i}(\theta_i+r_i\Omega t)}
\end{equation}
for $i=1,\cdots,m$, we obtain the reduced dynamics~\eqref{eq:red-dyn} on the $2m$-dimensional SSM in the autonomous form
\begin{gather}
\begin{pmatrix}\dot{\rho}_i\\\dot{\theta}_i\end{pmatrix}=\boldsymbol{r}^{\mathrm{p}}_i(\boldsymbol{\rho},\boldsymbol{\theta},\Omega,\epsilon)+\mathcal{O}(\epsilon|\boldsymbol{\rho}|)g_i^\mathrm{p}(\phi),i=1,\cdots,m\nonumber\\
\dot{\phi}=\Omega, \quad (\boldsymbol{\rho},\boldsymbol{\theta})\in\mathbb{R}^m\times\mathbb{T}^m.\label{eq:ode-reduced-slow-polar}
\end{gather}
Here the superscript \emph{p} stands for `polar', $g_i^\mathrm{p}$ is a $2\pi$-periodic function, and
\begin{align}
\boldsymbol{r}^{\mathrm{p}}_i&=\begin{pmatrix}\rho_i\mathrm{Re}(\lambda_i^{\mathcal{E}})\\\mathrm{Im}(\lambda_i^{\mathcal{E}})-r_i\Omega\end{pmatrix}\nonumber\\
& +\sum_{(\boldsymbol{l},\boldsymbol{j})\in\mathcal{R}_i}\boldsymbol{\rho}^{\boldsymbol{l}+\boldsymbol{j}}\boldsymbol{Q}(\rho_i,\varphi_i(\boldsymbol{l},\boldsymbol{j}))\begin{pmatrix}\mathrm{Re}(\gamma(\boldsymbol{l},\boldsymbol{j}))\\\mathrm{Im}(\gamma(\boldsymbol{l},\boldsymbol{j}))\end{pmatrix}\nonumber\\
& +\epsilon\boldsymbol{Q}(\rho_i,-\theta_i)\begin{pmatrix}\mathrm{Re}(f_i)\\\mathrm{Im}(f_i)\end{pmatrix}
\end{align}
with $\mathcal{R}_i$ defined in~\eqref{eq:def-Ri} and with ${\varphi}_i$ and $\boldsymbol{Q}$ defined as
\begin{gather}
\varphi_i(\boldsymbol{l},\boldsymbol{j})=\langle \boldsymbol{l}-\boldsymbol{j}-\mathbf{e}_i,\boldsymbol{\theta} \rangle,\label{eq:varphi-ang}\\
\boldsymbol{Q}(\rho,\theta) = \begin{pmatrix}\cos \theta & -\sin \theta\\\frac{1}{\rho}\sin \theta&\frac{1}{\rho}\cos\theta\end{pmatrix},\\
f_i=\left\lbrace \begin{array}{cl}
      (\boldsymbol{u}_i^{\mathcal{E}})^\ast\boldsymbol{F}^a   & \text{if}\hspace{2mm}  r_i=1 \\
      {0}   & \text{otherwise}
    \end{array} \right. \label{eq:fi}.
\end{gather}
Here $\mathbf{e}_i\in\mathbb{R}^m$ is the unit vector aligned with the $i$-th axis.
\item {\emph{[Cartesian reduced dynamics]}} Rewriting the parameterization~\eqref{eq:p-to-qs} in the form
\begin{gather}
    q_i=q_{i,\mathrm{s}}e^{\mathrm{i}r_i\Omega t}=(q_{i,\mathrm{s}}^\mathrm{R}+\mathrm{i}q_{i,\mathrm{s}}^\mathrm{I})e^{\mathrm{i}r_i\Omega t},\nonumber\\
    \label{eq:cartesian-form}
    \bar{q}_i=\bar{q}_{i,\mathrm{s}}e^{-\mathrm{i}r_i\Omega t}=(q_{i,\mathrm{s}}^\mathrm{R}-\mathrm{i}q_{i,\mathrm{s}}^\mathrm{I})e^{-\mathrm{i}r_i\Omega t},
\end{gather}
for $i=1,\cdots,m$, where $q_{i,\mathrm{s}}^\mathrm{R}=\mathrm{Re}(q_{i,\mathrm{s}})$ and $q_{i,\mathrm{s}}^\mathrm{I}=\mathrm{Im}(q_{i,\mathrm{s}})$, we obtain the reduced dynamics~\eqref{eq:red-dyn} on the SSM in \emph{Cartesian} coordinates $(\boldsymbol{q}_{\mathrm{s}}^\mathrm{R},\boldsymbol{q}_{\mathrm{s}}^\mathrm{I})\in\mathbb{R}^m\times\mathbb{R}^m$ as
\begin{equation}
\label{eq:ode-reduced-slow-cartesian}
\begin{pmatrix}\dot{q}_{i,\mathrm{s}}^\mathrm{R}\\\dot{q}_{i,\mathrm{s}}^\mathrm{I}\end{pmatrix}=\boldsymbol{r}^{\mathrm{c}}_i(\boldsymbol{q}_{\mathrm{s}},\Omega,\epsilon)+\mathcal{O}(\epsilon|\boldsymbol{p}|)g_i^\mathrm{c}(\phi)
\end{equation}
for $i=1,\cdots,m$. Here the superscript \emph{c} stands for `Cartesian', $g_i^\mathrm{c}$ is a $2\pi$-periodic function, and
\begin{align}
 \boldsymbol{r}^{\mathrm{c}}_i  &=\begin{pmatrix}\mathrm{Re}(\lambda_i^{\mathcal{E}}) & r_i\Omega-\mathrm{Im}(\lambda_i^{\mathcal{E}})\\
\mathrm{Im}(\lambda_i^{\mathcal{E}})-r_i\Omega & \mathrm{Re}(\lambda_i^{\mathcal{E}})\end{pmatrix}\begin{pmatrix}q_{i,\mathrm{s}}^{\mathrm{R}}\\q_{i,\mathrm{s}}^{\mathrm{I}}\end{pmatrix}
\nonumber\\
& +\sum_{(\boldsymbol{l},\boldsymbol{j})\in\mathcal{R}_i}\begin{pmatrix}\mathrm{Re}\left(\gamma(\boldsymbol{l},\boldsymbol{j})\boldsymbol{q}_s^{\boldsymbol{l}}\bar{\boldsymbol{q}}_s^{\boldsymbol{j}}\right)\\\mathrm{Im}\left(\gamma(\boldsymbol{l},\boldsymbol{j})\boldsymbol{q}_s^{\boldsymbol{l}}\bar{\boldsymbol{q}}_s^{\boldsymbol{j}}\right)\end{pmatrix}+\epsilon\begin{pmatrix}\mathrm{Re}(f_i)\\\mathrm{Im}(f_i)\end{pmatrix}.
\end{align}
\item Any hyperbolic fixed point of the leading-order truncation of~\eqref{eq:ode-reduced-slow-polar} or~\eqref{eq:ode-reduced-slow-cartesian}, namely,
\begin{equation}
\begin{pmatrix}\dot{\rho}_i\\\dot{\theta}_i\end{pmatrix}=\boldsymbol{r}^{\mathrm{p}}_i(\boldsymbol{\rho},\boldsymbol{\theta},\Omega,\epsilon),\quad i=1,\cdots, m,\label{eq:ode-reduced-slow-polar-leading}
\end{equation} or
\begin{equation}
\label{eq:ode-reduced-slow-cartesian-leading}
\begin{pmatrix}\dot{q}_{i,\mathrm{s}}^\mathrm{R}\\\dot{q}_{i,\mathrm{s}}^\mathrm{I}\end{pmatrix}=\boldsymbol{r}^{\mathrm{c}}_i(\boldsymbol{q}_{\mathrm{s}},\Omega,\epsilon),\quad i=1,\cdots,m,
\end{equation}
corresponds to a periodic solution $\boldsymbol{p}(t)$ of the reduced dynamical system~\eqref{eq:red-dyn} on the SSM, $\mathcal{W}(\mathcal{E},\Omega t)$. If $r_\mathrm{d}$ defines the largest common divisor for the set of rational numbers $\{r_1,\cdots,r_m\}$, then the period of $\boldsymbol{p}(t)$ is given by $T=\frac{2\pi}{r_\mathrm{d}\Omega}$. The stability type of a hyperbolic fixed point of ~\eqref{eq:ode-reduced-slow-polar-leading} or~\eqref{eq:ode-reduced-slow-cartesian-leading} coincides with the stability type of the corresponding periodic solution on $\mathcal{W}(\mathcal{E},\Omega t)$.
\item A saddle-node (SN) bifurcation of a fixed point of the vector field~\eqref{eq:ode-reduced-slow-polar-leading} or~\eqref{eq:ode-reduced-slow-cartesian-leading} corresponds to a SN bifurcation of a periodic orbit of the reduced dynamical system~\eqref{eq:red-dyn} on $\mathcal{W}(\mathcal{E},\Omega t)$.
\item A Hopf bifurcation (HB) of a fixed point of~\eqref{eq:ode-reduced-slow-polar-leading} or~\eqref{eq:ode-reduced-slow-cartesian-leading} corresponds to a torus (TR) bifurcation of a periodic orbit of the reduced dynamical system~\eqref{eq:red-dyn} on $\mathcal{W}(\mathcal{E},\Omega t)$.  
\end{enumerate}
\end{theorem}
\begin{proof}
We present the proof of this theorem in Appendix~\ref{sec:appendix-theo1}.
\end{proof}
\end{sloppypar}

Although the vector field $\boldsymbol{R}_\epsilon(\boldsymbol{p},\phi)$ (see~\eqref{eq:red-dyn}) is $\phi$-dependent, (i)-(ii) in the above theorem suggest that one can factor out the $\phi$-dependent terms via proper coordinate transformations between the reduced coordinates $\boldsymbol{p}$ and the polar coordinates $(\boldsymbol{\rho},\boldsymbol{\theta})$~\eqref{eq:polar-form} or the Cartesian coordinates $\boldsymbol{q}_\mathrm{s}$~\eqref{eq:cartesian-form}. These transformations yield simplified reduced dynamics~\eqref{eq:ode-reduced-slow-polar} and~\eqref{eq:ode-reduced-slow-cartesian}. The use of the method of normal forms in the computation of $\boldsymbol{R}_\epsilon(\boldsymbol{p},\phi)$ enables the simplification.

As a result of the above simplification, a fixed point of the leading-order reduced dynamics~\eqref{eq:ode-reduced-slow-polar-leading} in $(\boldsymbol{\rho},\boldsymbol{\theta})$ or~\eqref{eq:ode-reduced-slow-cartesian-leading} in $\boldsymbol{q}_\mathrm{s}$ corresponds to a periodic orbit of the leading-order dynamics in $\boldsymbol{p}$. In addition, if the fixed point is hyperbolic, the corresponding periodic orbit $\boldsymbol{p}(t)$ is structurally stable under the addition of the higher-order terms, and the stability type of the periodic orbit $\boldsymbol{p}(t)$ on the SSM $\mathcal{W}(\mathcal{E},\Omega t)$ is the same as that of the fixed point, as indicated by (iii). Therefore, the task of finding the periodic orbit $\boldsymbol{p}(t)$ on the SSM is converted into a much simpler task: locating the fixed point of the leading-order dynamics~\eqref{eq:ode-reduced-slow-polar-leading} or~\eqref{eq:ode-reduced-slow-cartesian-leading}. Importantly, the stability type of the periodic orbit is easily inferred from that of the fixed point.

Further, (iv)-(v) in the above theorem establish the relations between the bifurcations of the fixed point and the bifurcations of the corresponding periodic orbit on the SSM $\mathcal{W}(\mathcal{E},\Omega t)$. In particular, one can predict the saddle-node/torus bifurcations of the periodic orbit from the saddle-node/Hopf bifurcations of the fixed point.

In Theorem~\ref{theo-po}, we have given two coordinate representations for the reduced dynamics of the resonant SSM. This is because the polar coordinate representation may have a singularity, as detailed in Part I. A periodic orbit in the reduced dynamics~\eqref{eq:red-dyn} is a trajectory on the invariant manifold $\mathcal{W}(\mathcal{E},\Omega t)$ and hence one can obtain the corresponding periodic orbit of the full system via the map $\boldsymbol{W}_\epsilon(\boldsymbol{p},\phi)$ (see eq.~\eqref{eq:ssm-decomp}). In addition, the stability type of the periodic orbit on $\mathcal{W}(\mathcal{E},\Omega t)$ holds for that of the full system as well because $\mathcal{W}(\mathcal{E},\Omega t)$ is attracting~\cite{haller2016nonlinear}. Further, a bifurcation observed in the reduced dynamics on $\mathcal{W}(\mathcal{E},\Omega t)$ also holds for the full system. Therefore, we can infer the bifurcation of periodic orbits in the high-dimensional system~\eqref{eq:full-first} based on the the reduced-order model on its SSM.

Without SSM reduction, the detection of bifurcations of periodic orbits in a high-dimensional system is generally nontrivial. Firstly, the computation of the monodromy matrix of a periodic orbit and the spectrum of that matrix are computationally expensive. Secondly, the event functions for detecting periodic orbit bifurcations can have numerical issues, as detailed in Appendix~\ref{sec:event-bif-po}. These challenges are not encountered under SSM-reduction because the periodic orbits are solved for as fixed points in a phase space of significantly reduced dimension.

\section{Quasi-periodic orbits}
\label{sec:tori}
\subsection{Two-dimensional tori and their bifurcations}
\begin{sloppypar}
In the continuation of fixed points in the leading-order reduced dynamics (see eqs.~\eqref{eq:ode-reduced-slow-polar-leading} or~\eqref{eq:ode-reduced-slow-cartesian-leading}), the fixed points may exhibit Hopf bifurcations. In such a bifurcation, a unique limit cycle will emerge from the fixed point of the reduced dynamics on the SSM, giving rise to a one-dimensional manifold of limit cycles under variations of $\Omega$ or $\epsilon$. For each limit cycle in the leading-order dynamics, a two-dimensional invariant torus is obtained in the truncated reduced dynamics of the parameterization $\boldsymbol{p}$ (see~\eqref{eq:red-dyn},~\eqref{eq:red-auto-block} and~\eqref{eq:red-truncation}), given by
\begin{equation}
\label{eq:red-dyn-leading}
    \dot{\boldsymbol{p}}=\boldsymbol{R}(\boldsymbol{p})+\epsilon\boldsymbol{S}_{0}(\phi).
\end{equation}
Then a invariant torus in the full system~\eqref{eq:full-first} can be further obtained via the truncated embedding map (see~\eqref{eq:ssm-decomp}) 
\begin{equation}
\label{eq:ssm-decomp-leading}
\boldsymbol{W}_\epsilon(\boldsymbol{p},\phi)=\boldsymbol{W}(\boldsymbol{p})+\epsilon\boldsymbol{X}_{0}(\phi).
\end{equation}
We denote the approximate SSM corresponding to the above truncation by $\mathcal{W}_{{0}}(\mathcal{E},\Omega t)$ and summarize our discussion as follows:

\begin{theorem}
\label{theo-3}
Consider a periodic orbit of the leading-order SSM dynamics~\eqref{eq:ode-reduced-slow-polar-leading} or~\eqref{eq:ode-reduced-slow-cartesian-leading}. Let $T_\mathrm{s}$ be the period of this orbit and define the \emph{internal} frequency $\omega_\mathrm{s}=2\pi/T_\mathrm{s}$ and \emph{rotation number} $\varrho=\frac{\omega_\mathrm{s}}{r_\mathrm{d}\Omega}$. The following statements hold:
\begin{enumerate}[label=(\roman*)]
\item If $\varrho\in\mathbb{Q}$, then the corresponding solution $\boldsymbol{p}(t)$ to truncated SSM dynamics~\eqref{eq:red-dyn-leading} is a periodic orbit on the approximate SSM, $\mathcal{W}_{{0}}(\mathcal{E},\Omega t)$. Otherwise, the solution $\boldsymbol{p}(t)$ densely covers the surface of a two-dimensional invariant torus on the $\mathcal{W}_{{0}}(\mathcal{E},\Omega t)$.
\item The stability type of the solution $\boldsymbol{p}(t)$ is the same as that of the periodic orbit to the leading-order SSM dynamics~\eqref{eq:ode-reduced-slow-polar-leading} or~\eqref{eq:ode-reduced-slow-cartesian-leading}.
\item If $\varrho\not\in\mathbb{Q}$, a saddle-node/period-doubling/torus bifurcation of a periodic orbit of~\eqref{eq:ode-reduced-slow-polar-leading} or~\eqref{eq:ode-reduced-slow-cartesian-leading} corresponds to a quasi-periodic saddle-node/period-doubling/Hopf bifurcation of a two-dimensional torus of~\eqref{eq:red-dyn-leading}.
\end{enumerate}
\end{theorem}
\begin{proof}
We present the proof of this theorem in Appendix~\ref{sec:appendix-theo3-proof}.
\end{proof}

Theorem~\ref{theo-3} implies that a periodic orbit in the leading-order SSM dynamics~\eqref{eq:ode-reduced-slow-polar-leading} or~\eqref{eq:ode-reduced-slow-cartesian-leading} always corresponds to a two-dimensional invariant torus of the truncated reduced dynamics~\eqref{eq:red-dyn-leading}. Such a torus is denoted as $\boldsymbol{p}_\mathrm{tor2}$ here. With the map $\boldsymbol{W}_\epsilon(\boldsymbol{p},\phi)$ (see eq.~\eqref{eq:ssm-decomp-leading}) applied, the corresponding solution of the full system, $\boldsymbol{z}_\mathrm{tor2}$, is obtained. In addition, $\boldsymbol{W}_\epsilon(\boldsymbol{p},\phi)$ maps the closed invariant curve of $\boldsymbol{p}_\mathrm{tor2}$ under the period-$T$ map of the flow to a closed invariant curve of $\boldsymbol{z}_\mathrm{tor2}$ under its period-$T$ map. Hence, the corresponding solution $\boldsymbol{z}_\mathrm{tor2}$ in the full system is also a two-dimensional invariant torus. In addition, the stability type of the torus $\boldsymbol{z}_\mathrm{tor2}$ is the same as that of the periodic orbit in the slow-phase dynamics. Moreover, we can infer the bifurcations of the torus $\boldsymbol{z}_\mathrm{tor2}$ from the bifurcations of the periodic orbit.
\end{sloppypar}

\subsection{Three-dimensional tori}
\begin{sloppypar}
At a torus bifurcation of a periodic orbit in the leading-order SSM dynamics~\eqref{eq:ode-reduced-slow-polar-leading} or~\eqref{eq:ode-reduced-slow-cartesian-leading}, a unique two-dimensional invariant torus bifurcates from the periodic orbit. Further, a one-dimensional solution manifold of two-dimensional invariant tori is obtained under the variation of $\Omega$ or $\epsilon$. An invariant two-dimensional torus on the manifold generically corresponds to a three-dimensional invariant torus solution to~\eqref{eq:red-dyn-leading}. Then a three-dimensional invariant torus in the full system can be further obtained via the map~\eqref{eq:ssm-decomp-leading}. To formalize these statements, we have:
\begin{proposition}
\label{prop4}
Consider a two-dimensional invariant torus with frequencies $\omega_{1,\mathrm{s}}$ and $\omega_{2,\mathrm{s}}$ of the leading-order reduced dynamics~\eqref{eq:ode-reduced-slow-polar-leading} or~\eqref{eq:ode-reduced-slow-cartesian-leading}. This 2-torus them implies the existence of an invariant 3-torus of the same stability type in the truncated SSM-reduced dynamics~\eqref{eq:red-dyn-leading}. If $\Omega$, $\omega_{1,\mathrm{s}}$ and $\omega_{2,\mathrm{s}}$ are rationally independent, then the 3-torus is foilated by a one-parameter family of 2-tori. If $\Omega$, $\omega_{1,\mathrm{s}}$ and $\omega_{2,\mathrm{s}}$ are rationally dependent, then the 3-torus is densely filled with quasi-periodic orbits. Precisely one family of their orbits with $\phi(0)=0$ represents quasi-periodic orbit of the original system~\eqref{eq:full-first}.
\end{proposition}
\end{sloppypar}

Proposition~\ref{prop4} implies that a two-dimensional invariant torus in the leading-order reduced dynamics~\eqref{eq:ode-reduced-slow-polar-leading} or~\eqref{eq:ode-reduced-slow-cartesian-leading} {generically} corresponds to a three-dimensional quasi-periodic invariant torus of the truncated reduced dynamics~\eqref{eq:red-dyn-leading} on the resonant SSM. With the map $\boldsymbol{W}_\epsilon(\boldsymbol{p},\phi)$ (see equation~\eqref{eq:ssm-decomp-leading}) applied, a corresponding solution $\boldsymbol{z}_\mathrm{tor3}$ of the full system is obtained. In addition, $\boldsymbol{W}_\epsilon(\boldsymbol{p},\phi)$ maps the closed invariant surface of $\boldsymbol{p}_\mathrm{tor3}$ under the period-$T$ map of the flow to a closed invariant surface of the corresponding $\boldsymbol{z}_\mathrm{tor3}$ under its period-$T$ map. Then the $\boldsymbol{z}_\mathrm{tor3}$ in the full system is also a three-dimensional invariant torus. In addition, the stability type of $\boldsymbol{z}_\mathrm{tor3}$ is the same as that of the two-dimensional invariant torus in the leading-order reduced dynamics~\eqref{eq:ode-reduced-slow-polar-leading} or~\eqref{eq:ode-reduced-slow-cartesian-leading}.

Our discussions of two- and three-dimensional invariant tori in this section are based on the truncated reduced dynamics~\eqref{eq:red-dyn-leading}. Numerical experiments show that the this truncation is already sufficient for accurate results. More importantly, the reduced-order models~\eqref{eq:ode-reduced-slow-polar-leading} and~\eqref{eq:ode-reduced-slow-cartesian-leading} are parametric models with $\Omega$ and $\epsilon$ as system parameters, enabling convenient and sufficiently accurate parameter continuation. A full analysis of the persistence of the invariant tori under the consideration of higher-order terms $\mathcal{O}(\epsilon|\boldsymbol{p}|)$ in~\eqref{eq:red-truncation} is beyond the scope of this paper, because these tori are only weakly normally hyperbolic with respect to the perturbations represented by these higher order terms~\cite{haller2012chaos}. Our numerical experiments indicate, however, that this persistence generally holds.

\section{Implementation}
\label{sec:implementation}
\subsection{Software toolboxes}
\begin{sloppypar}
In Part I, we have introduced the \texttt{SSM-ep} toolbox for the computation of the forced response curve (FRC) of periodic orbits of the full system~\eqref{eq:full-first}. Such a FRC is computed as a branch of fixed points of the corresponding leading-order reduced model~\eqref{eq:ode-reduced-slow-polar-leading} or~\eqref{eq:ode-reduced-slow-cartesian-leading}. The toolbox also supports the continuation of bifurcated equilibria in the leading-order model and hence the continuation of bifurcated periodic orbits in the full system~\eqref{eq:full-first}. Specifically, it supports the continuation of SN/HB bifurcation fixed points under variations of $\Omega$ \emph{and} $\epsilon$. The one-dimensional manifold of SN/HB equilibria obtained in this fashion is then mapped back to physical coordinates to yield a one-dimensional manifold of bifurcated periodic orbits of the full system. The continuation of (bifurcated) equilibria in~\eqref{eq:ode-reduced-slow-polar-leading} or~\eqref{eq:ode-reduced-slow-cartesian-leading} is achieved with the \texttt{ep} toolbox of \textsc{coco}~\cite{dankowicz2013recipes,COCO,ahsan2022methods}, a {general-purpose} toolbox for the bifurcation analysis of equilibria of smooth dynamical systems.
\end{sloppypar}

\begin{sloppypar}
We have further developed the \texttt{SSM-po} toolbox for the computation and bifurcation analysis of two-dimensional invariant tori in the full system~\eqref{eq:full-first}. This toolbox performs the continuation of periodic orbits in the leading-order reduced models~\eqref{eq:ode-reduced-slow-polar-leading}-\eqref{eq:ode-reduced-slow-cartesian-leading} and then maps the periodic orbits to tori in the full system. Specifically, it supports
\begin{itemize}
\item The switch from the continuation of equilibria in the reduced-order model~\eqref{eq:ode-reduced-slow-polar-leading} or~\eqref{eq:ode-reduced-slow-cartesian-leading} to the continuation of periodic orbits at HB equilibria.
\item The continuation of periodic orbits in the reduced-order model~\eqref{eq:ode-reduced-slow-polar-leading} or~\eqref{eq:ode-reduced-slow-cartesian-leading} under variations of $\Omega$ or $\epsilon$. The resulting one-parameter family of periodic orbits is then mapped back to physical coordinates to yield a one-parameter family of two-dimensional invariant tori of the full system~\eqref{eq:full-first}. 
\item The continuation of saddle-node (SN)/period-doubling (PD)/torus (TR) bifurcation of periodic orbits in the leading-order reduced dynamics~\eqref{eq:ode-reduced-slow-polar-leading} or~\eqref{eq:ode-reduced-slow-cartesian-leading} under variations of $\Omega$ \emph{and} $\epsilon$. The one-parameter family of SN/PD/TR bifurcation periodic orbits is then mapped back to physical coordinates to yield a one-parameter family of quasi-periodic SN/PD/HB bifurcation solutions.
\end{itemize}
The parameter continuation in the \texttt{SSM-po} toolbox is achieved with the help of the \texttt{po} toolbox of \textsc{coco}~\cite{dankowicz2013recipes,COCO,ahsan2022methods}, which is a {general-purpose} toolbox for the bifurcation analysis of periodic orbits of dynamical systems. The details of the mapping from periodic orbits in the leading-order SSM-reduced dynamics to two-dimensional invariant tori in the full system have been presented in Appendix~\ref{sec:app-po2tor}.
\end{sloppypar}

\begin{sloppypar}
We have also developed the \texttt{SSM-tor} toolbox for the computation and bifurcation analysis of three-dimensional invariant tori in the full system~\eqref{eq:full-first}. This toolbox performs the continuation of two-dimensional invariant tori in the leading-order models~\eqref{eq:ode-reduced-slow-polar-leading}-\eqref{eq:ode-reduced-slow-cartesian-leading} and maps the two-dimensional invariant tori to the three-dimensional invariant tori of the full system~\eqref{eq:full-first}. Specifically, it supports
\begin{itemize}
\item The switch from the continuation of periodic orbits in the leading-order model~\eqref{eq:ode-reduced-slow-polar-leading} or~\eqref{eq:ode-reduced-slow-cartesian-leading} to the continuation of two-dimensional invariant tori at TR bifurcations of periodic orbits.
\item The continuation of two-dimensional invariant tori in the leading-order SSM-reduced models~\eqref{eq:ode-reduced-slow-polar-leading}-\eqref{eq:ode-reduced-slow-cartesian-leading} under variations of $\Omega$ or $\epsilon$. The one-parameter family of two-dimensional invariant tori is then mapped back to physical coordinates to yield a one-parameter family of three-dimensional invariant tori in the full system~\eqref{eq:full-first}.
\end{itemize}
The parameter continuation in the \texttt{SSM-tor} toolbox is achieved with the help of the \texttt{Tor}-toolbox~\cite{li2020tor}, which is a {general-purpose} toolbox for the continuation of two-dimensional invariant tori in autonomous systems and non-autonomous systems with periodic forcing. A brief introduction to the \texttt{Tor} toolbox will be given in Appendix~\ref{sec:app-Tor}. The details of the mapping from two-dimensional invariant tori in the leading-order models~\eqref{eq:ode-reduced-slow-polar-leading}-\eqref{eq:ode-reduced-slow-cartesian-leading} to three-dimensional invariant tori in the full system have been presented in Appendix~\ref{sec:app-tor22tor3}.
\end{sloppypar}

\subsection{Computational cost}
\label{sec:compLoad}
\begin{sloppypar}
In Part I, we derived a more elaborate version of~\eqref{eq:ssm-decomp-leading} of the form
\begin{equation}
\label{eq:zt-map-full}
\boldsymbol{z}(t)=\boldsymbol{W}(\boldsymbol{p}(t))+\epsilon\left(\boldsymbol{x}_{{0}}e^{\mathrm{i}\Omega t}+\bar{\boldsymbol{x}}_{0}e^{-\mathrm{i}\Omega t}\right),
\end{equation}
where $\boldsymbol{x}_{{0}}$ is the solution of a system of $\Omega$ dependent linear equations. With that formulation, SSM analysis can be decomposed into four parts as follows:
\begin{itemize}
\item \emph{Autonomous SSM}: This includes the expansion coefficients in  $\boldsymbol{W}$ as well as the coefficients of the nonlinear terms in the vector field of the reduced dynamics (see~\eqref{eq:auto-ssm-red}). These coefficients are $\Omega$-independent in the neighborhood of a resonance relation and hence require a one-time computation. The computational cost of this part can be significant if the full system is high-dimensional and the selected expansion order of the SSM is high.
\item \emph{Reduced dynamics}:
This includes finding periodic orbits, 2-tori and 3-tori of the reduced dynamics~\eqref{eq:red-dyn-leading}. Thanks to the normal form analysis embedded in SSM reduction, these tasks are simplified as finding fixed points, periodic orbits and 2-tori of the leading-order reduced dynamics~\eqref{eq:ode-reduced-slow-polar-leading}-\eqref{eq:ode-reduced-slow-cartesian-leading}. These solutions are parameter dependent, and hence parameter continuation is required to cover the solution family effectively. Note that the phase space of the leading-order models is of low-dimension, irrespective of the dimension of the full system~\eqref{eq:full-first}. Therefore, the computational cost of this part is relatively small.
\item \emph{Nonautonomous SSM}: This includes the expansion coefficient vector $\boldsymbol{x}_{{0}}$. Note that the system of linear equations for solving $\boldsymbol{x}_{{0}}$ is of the same size as the dimension of $\boldsymbol{z}$. The computational cost of this part can therefore be significant as well if the full system is of high dimension and the number of samples for $\Omega$ is large. However, the computation here is parallelizable because the linear equations for each sampled $\Omega$ can be solved independently.
\item \emph{Evaluation of map $\boldsymbol{W}(\boldsymbol{p})$}: This includes the evaluations of the map $\boldsymbol{W}(\boldsymbol{p})$ in \eqref{eq:zt-map-full} for $\boldsymbol{p}(t)$ which are periodic orbits, 2-tori and 3-tori of the reduced dynamics~\eqref{eq:red-dyn-leading}. Recall that each 2-torus or 3-torus is approximated with a collection of trajectories $\{\boldsymbol{p}_j(t)\}_{j=1}^{n_\mathrm{traj}}$ (see Appendixes~\ref{sec:app-po2tor} and~\ref{sec:app-tor22tor3}), the computational time of this part can be considerable if $n_\mathrm{traj}\gg1$ and the number of tori is large. However, the computation here is also parallelizable because the evaluation can be performed independently for each torus and each trajectory of a given torus.
\end{itemize}
\end{sloppypar}
\begin{sloppypar}
In practice, we first compute lower- dimensional invariant sets and then move to higher-dimensional invariant sets. In particular, one may first calculate the FRC of periodic orbits. Once TR bifurcations of periodic orbits are detected, we may switch to the computation of FRC of two-dimensional invariant tori. Likewise, one may further switch to the computation of three-dimensional invariant tori at quasi-periodic Hopf bifurcations of two-dimensional invariant tori. Note that the autonomous part of the SSM in all these computations is the same as long as the resonance relation~\eqref{eq:full-first} does not change. Therefore, the autonomous part of the SSM obtained in the calculation of FRC of periodic orbits can be directly utilized in later computations as well.
\end{sloppypar}

\section{Examples}
\label{sec:example}
\subsection{Two coupled nonlinear oscillators in resonance}
\label{sec:example1}
Consider two coupled nonlinear oscillators with governing equations
\begin{gather}
\ddot{x}_1+c_1\dot{x}_1+x_1+b_1x_1x_2=\epsilon f_1\cos\Omega t,\nonumber\\
\ddot{x}_2+c_2\dot{x}_2+4x_2+b_2x_1^2=\epsilon f_2\cos\Omega t.\label{eq:eom-two-os}
\end{gather}
The eigenvalues of the linearized system are
\begin{gather}
\lambda_{1,2}=-\frac{c_1}{2}\pm\mathrm{i}\sqrt{1-0.25c_1^2}\approx\pm\mathrm{i},\nonumber\\
\lambda_{3,4}=-\frac{c_2}{2}\pm\mathrm{i}\sqrt{4-0.25c_2^2}\approx\pm2\mathrm{i},
\end{gather}
provided that $c_1\ll1$ and $c_2\ll1$. In that case, the system has a 1:2 internal resonance. We focus on the primary resonance of the first mode for which we have $\boldsymbol{r}=(1,2)$ in~\eqref{eq:res-forcing}. This first example has a four-dimensional, time-periodic, resonant SSM. In this case, SSM analysis does not involve any reduction, just passage to the leading-order reduced form~\eqref{eq:ode-reduced-slow-polar-leading}-\eqref{eq:ode-reduced-slow-cartesian-leading}.

\begin{sloppypar}
With $c_1=0.005$ N.s/m, $c_2=0.01$ N.s/m, $b_1=0.3$ N/$\mathrm{m}^3$, $b_2=1$ N/$\mathrm{m}^2$, $f_1=1$ N, $f_2=0$ N and $\epsilon=0.01$, we obtain the FRCs in modal coordinates $(\rho_1,\rho_2)$ and in physical coordinates $(||x_1||_{\infty},||x_2||_{\infty})$ for periodic orbits with $\Omega\in[0.7,1.1]$ in Fig.~\ref{fig:oneTwo}. Here and throughout the paper $||\bullet||_{\infty}:=\max_{t\in[0,T]}||\bullet(t)||$ gives the amplitude of the periodic or quasi-periodic response. Although $f_2=0$, the second mode is activated and $\mathcal{O}(\rho_2)\sim\mathcal{O}(\rho_1)$ due to the internal resonance. We also performed the continuation of periodic orbits of the original system to obtain reference solutions. Such a continuation run is conducted with the \texttt{po} toolbox of \textsc{coco}~\cite{dankowicz2013recipes,COCO}, where a collocation method is used to obtain periodic orbits. As can be seen in the last two panels in Fig.~\ref{fig:oneTwo}, the FRCs obtained from SSM analysis match well with the reference solutions by the collocation method provided that the amplitude of response is small. In addition, numerical experiments show that the discrepancies observed at large response amplitudes in the figure are eliminated when the expansion order of SSM is increased to seven or higher. We use the third-order expansion in this example because we are mainly concerned with quasi-periodic responses, whose amplitudes are small, as inferred from the positions of the Hopf bifurcation fixed points (HB1 and HB2). In this setting, the third-order expansion yields accurate results for invariant tori.
\end{sloppypar}

\begin{figure}[!ht]
\centering
\includegraphics[width=0.45\textwidth]{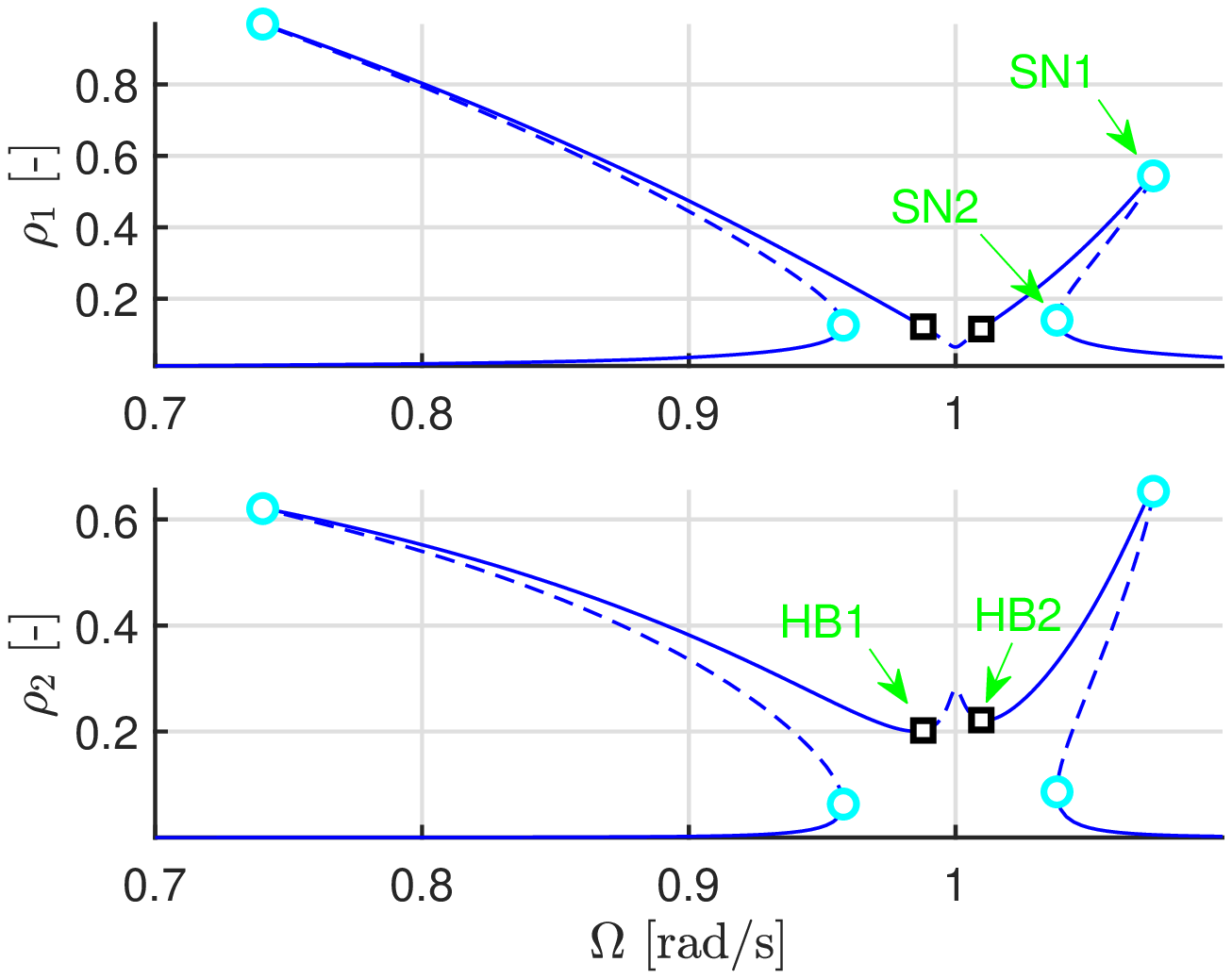}
\includegraphics[width=0.45\textwidth]{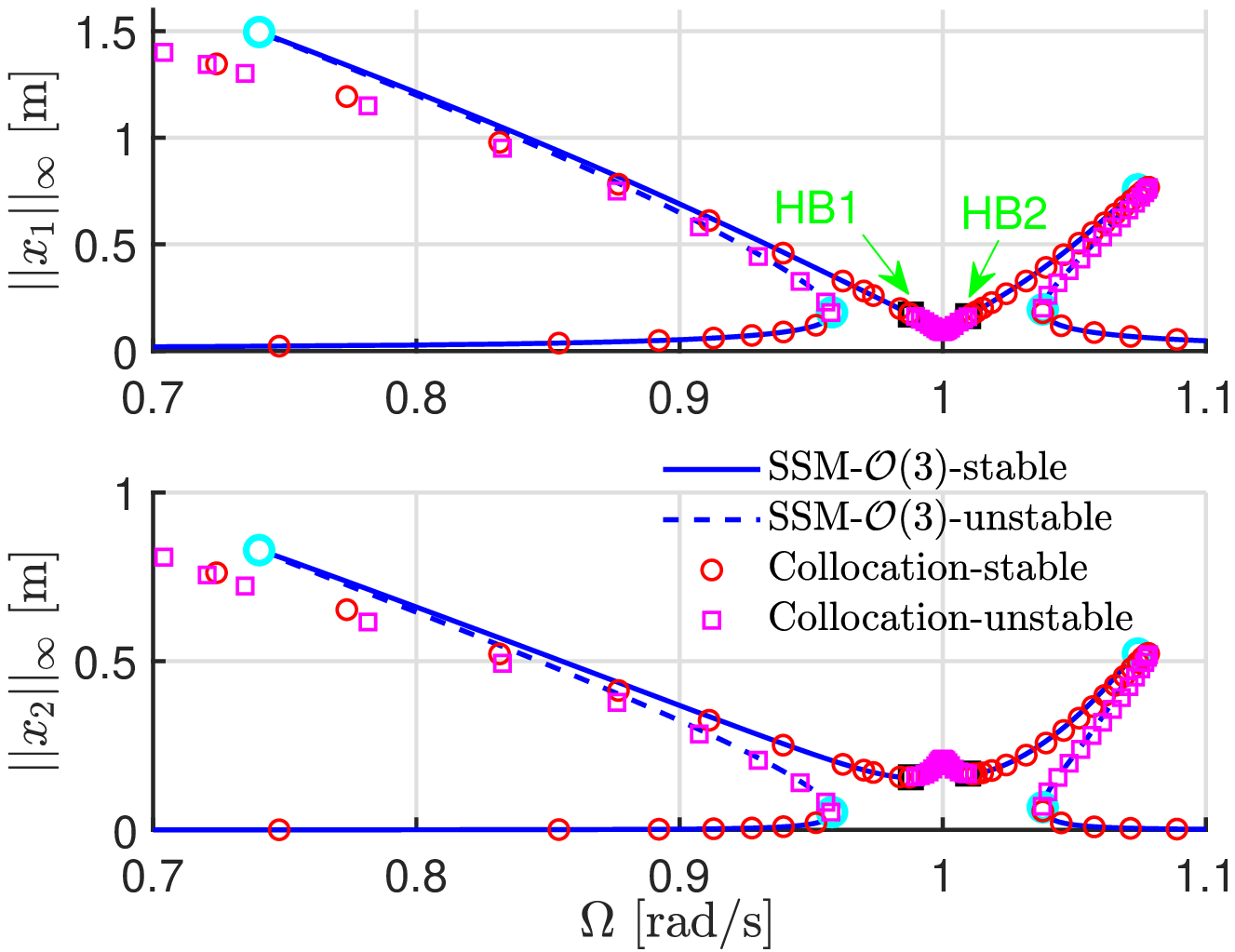}
\caption{FRCs in normal coordinates ($\rho_1,\rho_2$) and physical coordinates ($x_1,x_2$) of the coupled resonant oscillators in \eqref{eq:eom-two-os}. Here and throughout the paper, the solid lines indicate stable solution branches while dashed lines mark unstable solution branches; the cyan circles denote Saddle-Node (SN) bifurcation points and black squares denote Hopf bifurcation (HB) points. In the panels for FRCs in physical coordinates, the results obtained by continuation of periodic orbits with collocation methods using the \texttt{po} toolbox of \textsc{coco} are presented as well to validate the SSM-reduction results.}
\label{fig:oneTwo}
\end{figure}

\begin{sloppypar}
Both saddle-node (SN) and Hopf bifurcation (HB) fixed points are detected in the leading-order reduced model given by~\eqref{eq:ode-reduced-slow-cartesian-leading}. Specifically, four SN points and two HB points are observed, as seen in Fig.~\ref{fig:oneTwo}. These are codimension-one bifurcations and hence a one-parameter family, e.g., a bifurcation curve in the $(\Omega,\epsilon)$ plane, can be obtained for each type of bifurcation. We performed one-dimensional continuation of SN fixed points with $(\Omega,\epsilon)\in[0.7,1.1]\times[0.0001,0.05]$. The SN1 in Fig.~\ref{fig:oneTwo} is used as the starting point of such a continuation. The computed bifurcation curve of SN is plotted in the first panel of Fig.~\ref{fig:oneTwo-SNHB}. We also performed the continuation of SN bifurcation {periodic orbits} of the original system using the \texttt{po} toolbox of \textsc{coco}. The results of this continuation run are presented in the first panel of Fig.~\ref{fig:oneTwo-SNHB} as well for comparison. With $\epsilon=0.01$, two SN points are found and the one in the upper branch with smaller $\Omega$ corresponds to the SN2 in Fig.~\ref{fig:oneTwo}, while the one in the lower branch with larger $\Omega$ corresponds to the SN1 in Fig.~\ref{fig:oneTwo}. We observe that the solutions in the upper branch match well for all $\epsilon\in[0.0001,0.05]$, thanks to small response amplitudes on this branch (cf. Fig.~\ref{fig:oneTwo}). In contrast, the difference between the results of the two methods becomes significant for increasing $\epsilon$ along the lower branch, where the amplitude of the response is large (cf. Fig.~\ref{fig:oneTwo}). These discrepancies are eliminated again when the expansion order of SSM is increased to five or higher. The two branches merge and then the SN bifurcation does not exist for $\epsilon\to0$, yielding a cusp bifurcation.
\end{sloppypar}

\begin{figure}[!ht]
\centering
\includegraphics[width=0.45\textwidth]{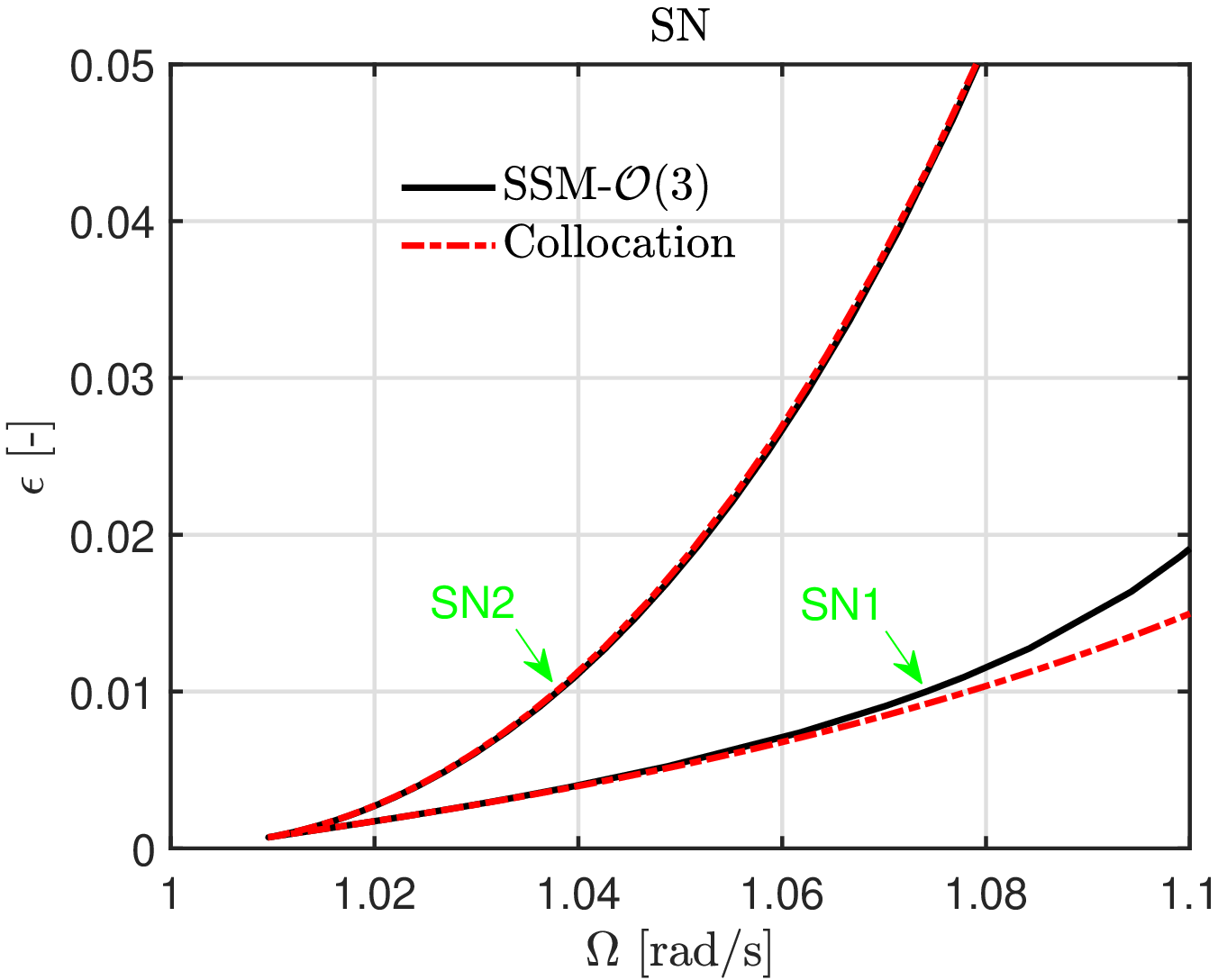}
\includegraphics[width=0.45\textwidth]{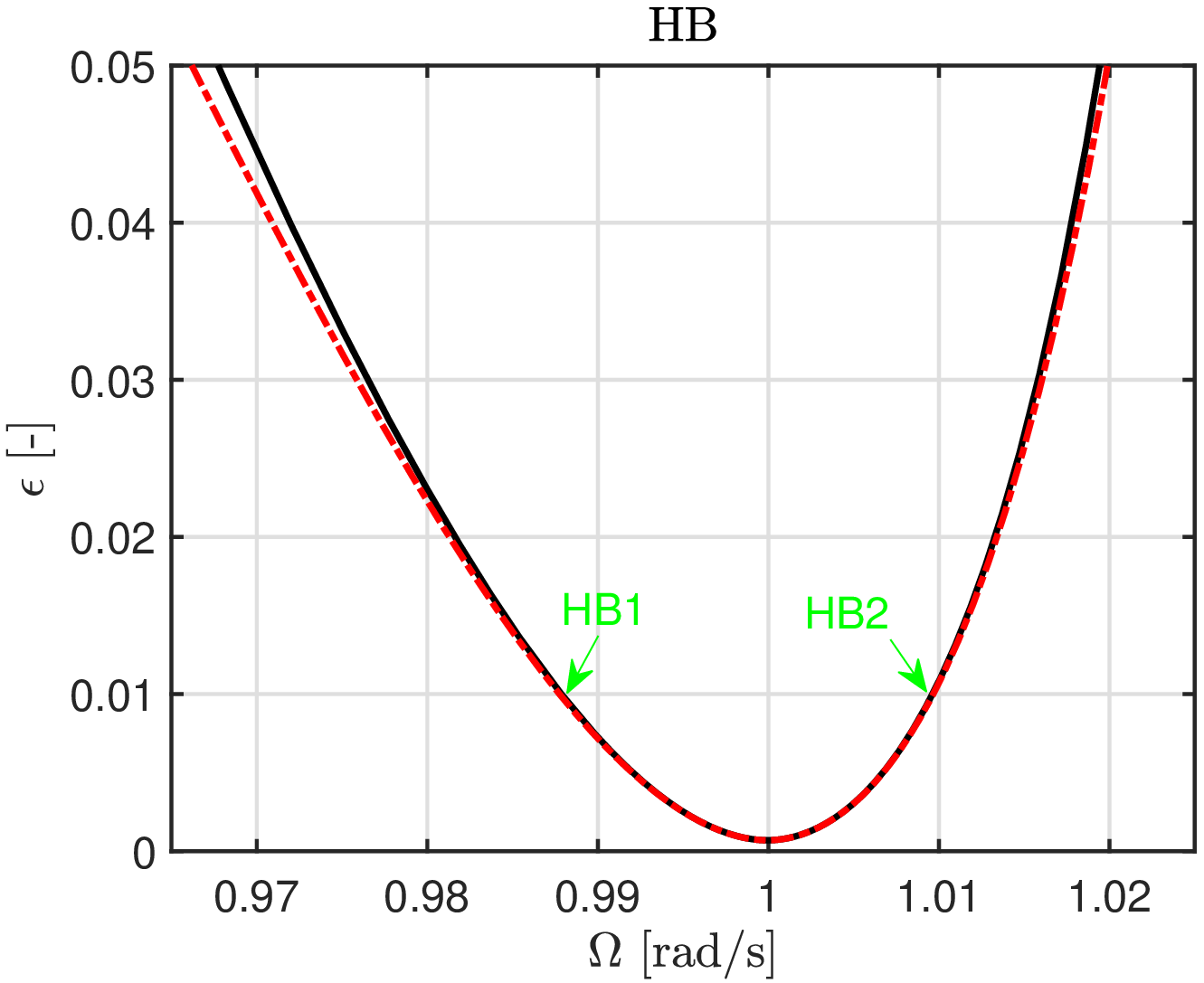}
\caption{Bifurcation curves of periodic orbits of the two coupled oscillators. In the first panel, the curve of saddle-node (SN) bifurcation periodic orbits is shown. Such SN periodic orbits are detected as SN fixed points in the slow-phase reduced dynamics. In the second panel, the curve of torus (TR) bifurcation periodic orbits is shown. Such TR periodic orbits are detected as Hopf bifurcation (HB) fixed points in the leading-order SSM-reduced dynamics. The curves of SN and TR periodic orbits are also obtained via the collocation method using the \texttt{po}-toolbox of \textsc{coco} as reference solutions.}
\label{fig:oneTwo-SNHB}
\end{figure}

Recall that the detected HB bifurcation fixed points of the leading-order SSM-reduced model~\eqref{eq:ode-reduced-slow-polar-leading}-\eqref{eq:ode-reduced-slow-cartesian-leading} corresponds to torus (TR) bifurcations of {periodic orbits} in the full system~\eqref{eq:full-first}. We performed one-dimensional continuation of the HB fixed point with $(\Omega,\epsilon)$ in the same computational domain as that of SN points. The HB1 in Fig.~\ref{fig:oneTwo} is used as the starting point of such a continuation. The computed bifurcation curve of HB points is plotted in the second panel of Fig.~\ref{fig:oneTwo-SNHB}. We also performed the continuation of TR bifurcations of periodic orbits of the original system using the \texttt{po} toolbox of \textsc{coco}. The results obtained by the two methods match well, as seen in the second panel of Fig~\ref{fig:oneTwo-SNHB}. Again, no HB (TR) bifurcation is found if the forcing amplitude $\epsilon$ is small enough. Indeed, the oscillators behave like linear oscillators for small $\epsilon$ values.

A unique limit cycle bifurcates from a HB fixed point and a family of such limit cycles can be formed under variation of $\Omega$ or $\epsilon$ from the critical parameter value for the HB point. We performed one-dimensional continuation of such limit cycles under varying $\Omega$. Several saddle-node and periodic-doubling bifurcation limit cycles along the continuation path are found. Detailed discussion of results from the continuation are given in Appendix~\ref{sec:ap-exap-limitcycle}.

\begin{sloppypar}
With these limit cycles obtained in the leading-order model~\eqref{eq:ode-reduced-slow-cartesian-leading}, we construct the corresponding invariant tori in the parameterization coordinates and then map the tori to invariant tori in the physical coordinates. For a given torus, one can calculate the amplitude of quasi-periodic response on such a tours. When $\Omega$ is varied, the amplitude is changed and we have the forced response curve for quasi-periodic responses as well. The FRCs for both periodic and quasi-periodic responses of the vibration of the first oscillator are presented in Fig.~\ref{fig:FRC_oneTwo_tori}. 
\end{sloppypar}

\begin{figure*}[!ht]
\centering
\includegraphics[width=0.8\textwidth]{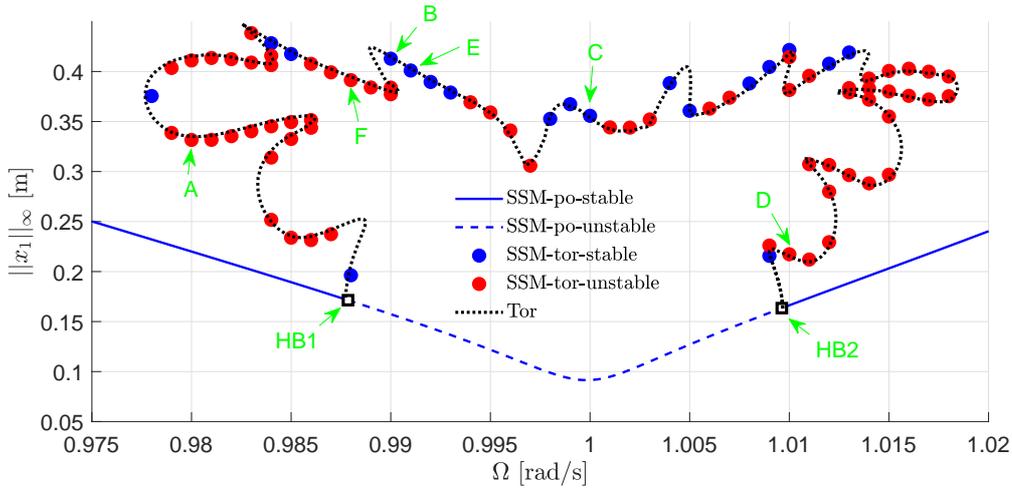}
\caption{FRCs for the periodic and quasi-periodic orbits of the first oscillator in \eqref{eq:eom-two-os} with $\Omega\approx1$. Here the solid/dashed lines denote the amplitudes of stable/unstable periodic orbits obtained by SSM analysis. Blue/red dots represent the amplitudes of stable/unstable quasi-periodic responses obtained by SSM analysis at uniformly sampled $\Omega$. The two black squares correspond to the two Hopf bifurcation (HB) fixed points in the leading-order dynamics~\eqref{eq:ode-reduced-slow-cartesian-leading} (or, equivalently, torus bifurcation periodic orbits in the original system~\eqref{eq:eom-two-os}). The dotted lines are the results obtained by applying \texttt{Tor}-toolbox to the original system.}
\label{fig:FRC_oneTwo_tori}
\end{figure*}

To validate the results for quasi-periodic orbits, we apply the \texttt{Tor} toolbox~\cite{li2020tor} to the original system directly to find two-dimensional invariant tori of the coupled oscillators. More details about the toolbox can be found at Appendix~\ref{sec:app-Tor}. With 50 Fourier modes and adaptive change of the collocation mesh, we obtain one-parameter family of invariant tori under varying $\Omega$. On such a solution manifold, $\Omega$ is free to change while $\epsilon$ is fixed. The results obtained from SSM-based analysis match well with the ones by \texttt{Tor}, as can be seen in Fig.~\ref{fig:FRC_oneTwo_tori}. 

We further consider two additional ways of validation based on the results from the \texttt{Tor} toolbox. The first method compares the internal frequency of a two-dimensional torus predicted by SSM-based analysis and the \texttt{Tor} toolbox. In the second method, we focus on the invariant intersection of such a torus with appropriate Poincar\'e section. We take sampled tori A-D in Fig.~\ref{fig:FRC_oneTwo_tori} in the second validation. As seen in Appendix~\ref{sec:ap-exap-twoval}, these two additional methods validate the effectiveness of SSM-based analysis as well.

The SSM analysis has several advantages over the direct calculation using the \texttt{Tor} toolbox. The computations above are performed in a Intel Xeon W processor (2.3 GHz). The computational time for the continuation of tori with the \texttt{Tor} toolbox is about one and half hours, while the one with SSM analysis is just about four minutes. Such a significant speed-up gain is not surprising because the SSM method performs the continuation of periodic orbits in the leading-order dynamics~\eqref{eq:ode-reduced-slow-cartesian-leading} (each periodic orbit is a single trajectory), while the \texttt{Tor} toolbox performs the continuation of two-dimensional tori (each torus is approximated with 101 trajectories here).

The current release of \texttt{Tor}-toolbox does not provide stability analysis to the computed invariant tori, while the stability type of the invariant tori obtained by SSM analysis is the same as that of the corresponding limit cycles in the leading-order dynamics~\eqref{eq:ode-reduced-slow-cartesian-leading}, which is simple to determine. We take E and F in Fig.~\ref{fig:FRC_oneTwo_tori} as representative samples of stable and unstable invariant tori and then successfully validate their stability types using numerical integration. More details about this validation can be found in Appendix~\ref{sec:ap-exap-stab}.

\subsection{A Bernoulli beam with nonlinear support spring}
\begin{sloppypar}
As our second example, we consider a cantilever beam with a nonlinear spring support at its free end. The beam is modeled using Bernoulli beam theory and hence the only nonlinearity in this example comes from the support spring, whose force-displacement relation is given by
\begin{equation}
F=k_\mathrm{l}w+k_\mathrm{nl}w^3.
\end{equation}
Here $w$ is the transverse displacement, $F$ is the spring force, and $k_\mathrm{l}$ and $k_\mathrm{nl}$ denote linear and nonlinear stiffness, respectively.
\end{sloppypar}

Let $b$ and $h$ be the width and height of the cross section, and $l$ be the length of the beam. We set $h=b=10\,\mathrm{mm}$ and $l=2700\,\mathrm{mm}$. Material properties are specified with density $\rho= 1780\times10^{-9}\,\mathrm{kg/{mm}^3}$ and Young's modulus $E=45\,\times10^6\,\mathrm{kPa}$. Following classic finite element discretization, two degrees of freedom are introduced at each node: the transverse displacement and the rotation angle. The displacement field is approximated with Hermite interpolation applied to each element. The equation of motion of the discretized beam model can be written as
\begin{equation}
\label{eq:eom-beams}
\boldsymbol{M}\ddot{\boldsymbol{x}}+\boldsymbol{C}\dot{\boldsymbol{x}}+\boldsymbol{K}\boldsymbol{x}+\boldsymbol{N}(\boldsymbol{x})=\epsilon\boldsymbol{f}\cos\Omega t,
\end{equation}
where $\boldsymbol{x}\in\mathbb{R}^{2N_{\mathrm{e}}}$ is the assembly of all degrees of freedom, and $N_{\mathrm{e}}$ is the number of elements used in the discretization. We assume Rayleigh damping for the beam elements (without the support spring) by letting
\begin{equation}
\label{eq:RayleighDamping}
\boldsymbol{C}=\alpha\boldsymbol{M}+\beta\boldsymbol{K}_\mathrm{b},
\end{equation}
where $\boldsymbol{K}_\mathrm{b}$ is the stiffness matrix of beam elements, to be obtained from $\boldsymbol{K}$ by removing the contribution of the linear stiffness of the support spring, i.e. $k_\mathrm{l}$.
If $\omega_i$ denotes the $i$-th natural frequency of the undamped linear system, we have
\begin{align}
\label{eq:weakDampFreq}
\lambda_{2i-1,2i} & \approx-\frac{\alpha+\beta\omega_i^2}{2}\pm\mathrm{i}\omega_i\sqrt{1-
\left(\frac{\alpha}{2\omega_i}+\beta\omega_i\right)^2}\nonumber\\
& \approx\pm\mathrm{i}\omega_i,
\end{align}
if $0\leq\alpha\ll\omega_i$ and $0\leq\beta\ll1$. In this example, we set $\alpha=1.25\times10^{-4}\,\mathrm{s}$ and $\beta=2.5\times10^{-5}\,\mathrm{s}^{-1}$ such that the system is weakly damped and the above approximation holds.

We set $k_\mathrm{l}=27$ N/m such that $\omega_2\approx3\omega_1$ and hence the system has near 1:3 internal resonance between the first two modes. We further set $k_\mathrm{nl}=60$ N/$\mathrm{m}^3$ such that the two bending modes are nonlinearly coupled. Let $\boldsymbol{\phi}_i$ be the linear normal mode corresponding to $\omega_i$, normalized with respect to $\boldsymbol{M}$. We set $\boldsymbol{f}=\omega_1^2\boldsymbol{M}\boldsymbol{\phi}_1$ such that only the first mode is externally excited, namely, $\boldsymbol{\phi}_i^{\mathrm{T}}\boldsymbol{f}=0$ for $i\geq2$. In the following computations, the beam is uniformly discretized with 40 elements and hence the system has 80 degrees-of-freedom. In this case, we have $\omega_1=15.60\,\mathrm{rad/s}$ and $\omega_2=46.58\,\mathrm{rad/s}$. Numerical experiments show that bifurcations observed in the reduced-order model of this discrete model are persistent when the number of elements is increased. We focus on the case of 40 elements here for simplicity.

With $\epsilon=0.002$, we compute a 4-dimensional SSM to account for the 1:3 internal resonance. Reduction to this SSM reduces the dimension of the phase space from 160 to 4. The FRCs obtained from this reduction in normal coordinates $(\rho_1,\rho_2)$ for $\Omega\in[15.30,15.95]$ are shown in Fig.~\ref{fig:Bernoulli_modal}. Mode interaction is observed around $\Omega=\omega_1$. Although the second bending mode is not excited externally, the response amplitude of the second mode is of the same order as the response amplitude of the first mode, namely, $\mathcal{O}(\rho_2)\sim\mathcal{O}(\rho_1)$. The nontrivial $\rho_2$ is induced by the internal resonance and the cubic nonlinearity of the support spring. In addition, $\rho_1$ drops significantly while $\rho_2$ arrives its peak around $\Omega=\omega_1$. This implies energy transfer between the two modes due to the internal resonance.

\begin{figure}[!ht]
\centering
\includegraphics[width=0.45\textwidth]{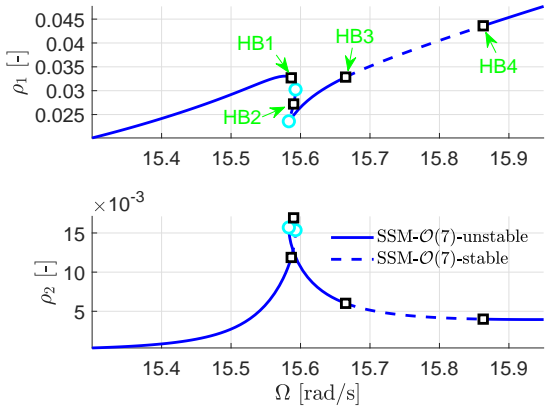}
\caption{FRCs in normal coordinates of the discretized cantilever Bernoulli beam with a nonlinear support spring at its free end.}
\label{fig:Bernoulli_modal}
\end{figure}

The FRC in physical coordinates for the vibration amplitude at the support end of the beam is presented in Fig.~\ref{fig:Bernoulli_physical}. To validate the effectiveness of SSM reduction, we also calculate the periodic orbits of the full system directly using the collocation method with the \texttt{po} toolbox of \textsc{coco}. Specifically, we use a fixed mesh with 10 subintervals and five base points in each subinterval in the collocation scheme. The maximum continuation step size and the maximum residual for the predictor in the one-dimensional atlas algorithm of \textsc{coco} are increased from the default values to 100 and 10, respectively, to adapt to this high-dimensional continuation problem. As seen in Fig.~\ref{fig:Bernoulli_physical}, the results from SSM-reduction match well with the results from the collocation method. Remarkably, the computational time for the collocation method is about nine and half hours while the one for SSM-reduction is about seven seconds.

\begin{figure}[!ht]
\centering
\includegraphics[width=0.45\textwidth]{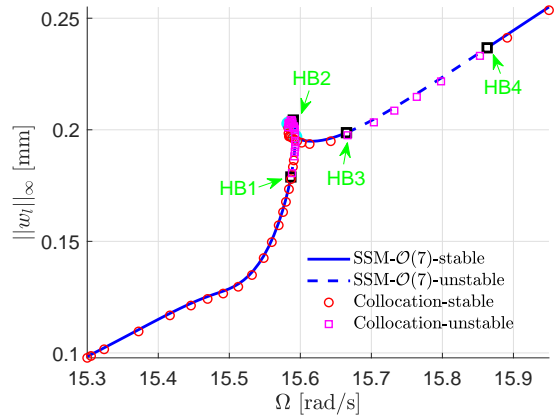}
\caption{FRC for the amplitudes of periodic orbits at the end of the cantilever Bernoulli beam with a nonlinear support spring at its end.}
\label{fig:Bernoulli_physical}
\end{figure}

Four Hopf and two saddle-node bifurcation fixed points are detected in the continuation of equilibria of the leading-order SSM-reduced model, as seen in Fig.~\ref{fig:Bernoulli_modal}. Each HB fixed point corresponds to a TR bifurcation of periodic orbit of the full system~\eqref{eq:eom-beams}. The detection of these periodic orbits via collocation-based continuation in the full system is challenging because the event functions for bifurcation detection are close to zero in the whole continuation run. To tackle this issue, we have used a subset of eigenvalues of the monodromy matrix of a periodic orbit instead of all eigenvalues, as discussed in Appendix~\ref{sec:event-bif-po}. When we use three eigenvalues of the monodromy matrix for the bifurcation detection, both the four HB points and the two SN points are found in the collocation-based continuation. However, none of them arises if we use four eigenvalues, and the SN points are not detected if we use two eigenvalues in the event functions. Such a high sensitivity to the number of eigenvalues used indicates the difficulty of detecting bifurcations of periodic orbits in high-dimensional systems. In contrast, these bifurcations can be easily found with the continuation of fixed points in the corresponding leading-order SSM-reduced models~\eqref{eq:ode-reduced-slow-polar-leading}-\eqref{eq:ode-reduced-slow-cartesian-leading}.

We switch from the continuation of equilibria of the leading-order SSM-reduced model to the continuation of periodic orbits at HB2 (see~Fig.~\ref{fig:Bernoulli_modal}), where a unique limit cycle bifurcates from the fixed point. Such a continuation run proceeds until the solution branch approaches to the HB1 point. Consistent results are obtained if we perform the switch at that point. In this continuation run, both stable and unstable limit cycles are observed, as seen in Fig.~\ref{fig:Bernoulli_TandRho}, where the plots of the period and the size of the computed limit cycles in the leading-order SSM-reduced model as functions of $\Omega$ are presented. In addition, TR and SN bifurcation limit cycles are detected in the continuation run.

\begin{figure}[!ht]
\centering
\includegraphics[width=0.48\textwidth]{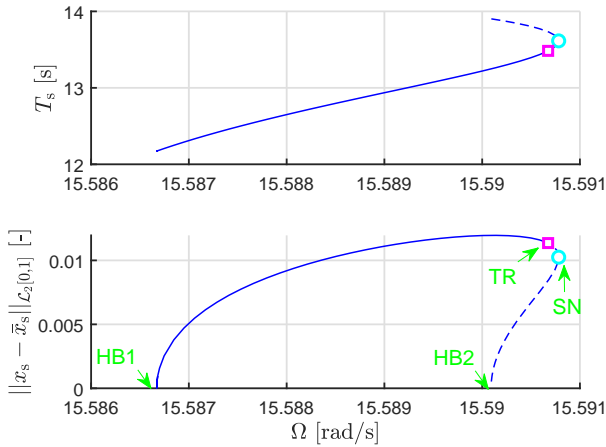}
\caption{Projections of the continuation path of the limit cycles in the leading-order SSM-reduced model~\eqref{eq:ode-reduced-slow-cartesian-leading} of the discrete Bernoulli beam model. The upper and lower panels present the period and the size of the limit cycles as functions of $\Omega$, respectively. The circles and diamonds correspond to saddle-node (SN) and Neimark-Sacker (TR) bifurcation periodic orbits respectively. The limit cycles here are mapped to the two-dimensional invariant tori in the full system with the two frequencies, $\Omega$ and $\omega_\mathrm{s}=2\pi/T_\mathrm{s}$. Formal definition of the size of limit cycles is given by~\eqref{eq:size-po}.}
\label{fig:Bernoulli_TandRho}
\end{figure}

We construct the corresponding two-dimensional invariant tori in the normal form coordinates $\boldsymbol{p}$ and then map them back to the two-dimensional invariant tori in the physical coordinates. The FRC for quasi-periodic orbits that stay on these two-dimensional invariant tori is presented in the first panel of Fig.~\ref{fig:Bernoulli_tori_34}. The FRC of quasi-periodic orbits intersects with the FRC of periodic orbits at HB1 and HB2. Here the family of two-dimensional invariant tori born out of HB1 is stable while the family of two-dimensional invariant tori born out of HB2 is unstable, indicating that the bifurcations at HB1 and HB2 are supercritical and subcritical, respectively (see~Fig.~\ref{fig:Bernoulli_TandRho}). In addition, we see from the upper panel of Fig.~\ref{fig:Bernoulli_tori_34} the coexistence of a stable torus, an unstable torus, a stable periodic orbit and an unstable periodic orbit for $\Omega\in[\Omega_\mathrm{HB2},\Omega_\mathrm{TR}]$, where $\Omega_\mathrm{HB2}=15.5901$ corresponds to the HB2 bifurcation point and $\Omega_\mathrm{TR}=15.5907$ corresponds to the TR bifurcation solution (see~Fig.~\ref{fig:Bernoulli_TandRho}). Hence the perturbation to an unstable invariant torus could result in a periodic orbit in steady state.

We repeat the same procedure for the remaining two HB points, HB3 and HB4, which are the boundary points for a segment of unstable periodic orbits (see~Fig.~\ref{fig:Bernoulli_physical}). At HB3, we switch from the continuation of equilibria of the reduced-order model to the continuation of limit cycles that bifurcate from HB3. Such a continuation run proceeds until $\Omega$ approaches the critical value at HB4. In this continuation run, all limit cycles obtained are stable. Therefore, a one-parameter family of stable limit cycles is obtained under varying $\Omega$, on which each limit cycle generically corresponds to a two-dimensional invariant torus in the full system~\eqref{eq:eom-beams}. We construct the corresponding invariant tori in physical coordinates; the FRC for quasi-periodic orbits that stay on these invariant tori is shown in the second panel of Fig.~\ref{fig:Bernoulli_tori_34}. The response curve of the quasi-periodic orbits intersects the response curve of the periodic orbits at the two HB points, HB3 and HB4. The response amplitude can be increased significantly when an unstable periodic orbit is perturbed and then converges to a stable quasi-periodic orbit.

\begin{figure}[!ht]
\centering
\includegraphics[width=0.45\textwidth]{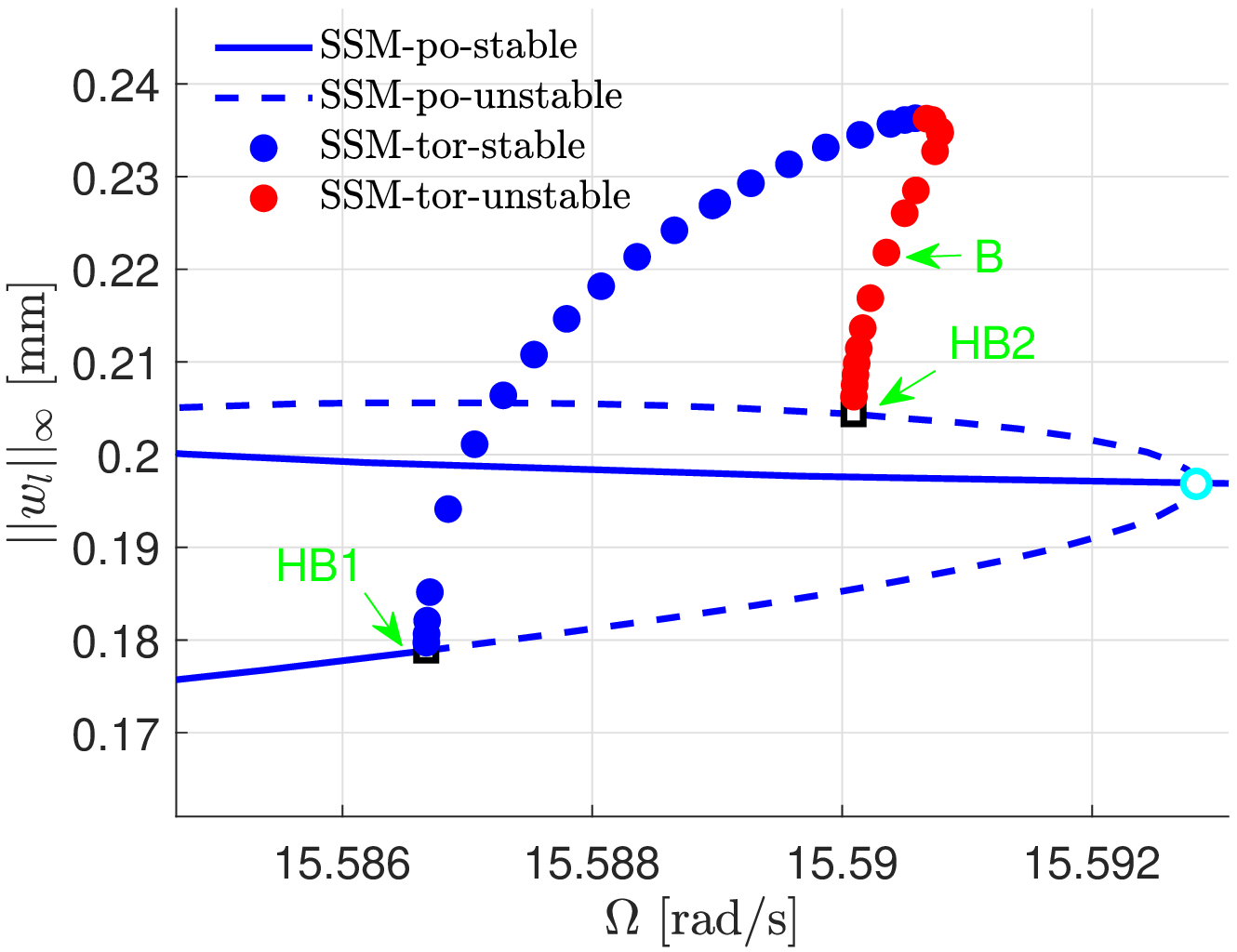}
\includegraphics[width=0.45\textwidth]{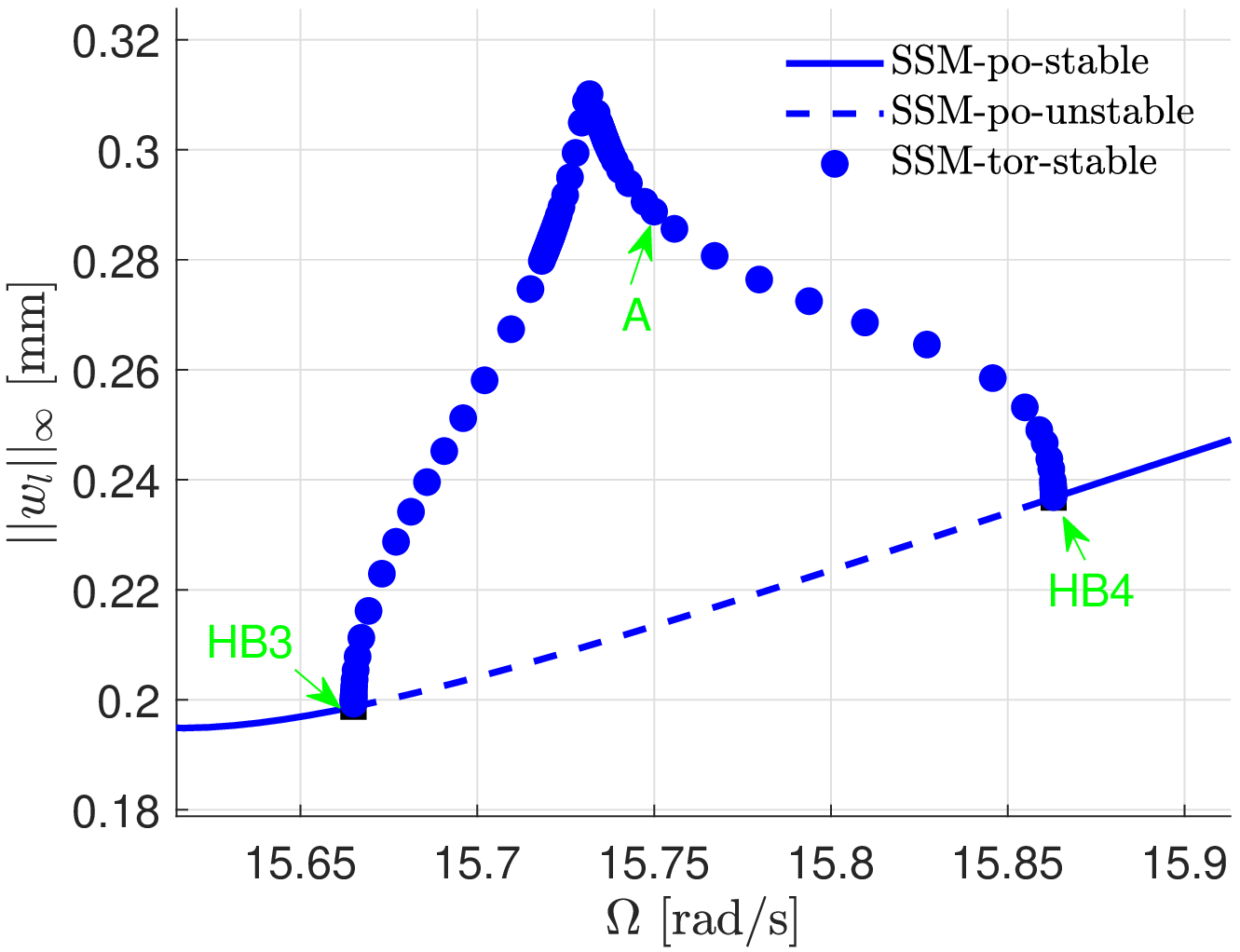}
\caption{FRCs for the vibration amplitude of quasi-periodic orbits at the end of the cantilever Bernoulli beam where it is supported by a nonlinear spring. The quasi-periodic orbits here bifurcate from HB1 and HB2 in the first panel, while the ones in the lower panel bifurcate from HB3 and HB4.}
\label{fig:Bernoulli_tori_34}
\end{figure}

To validate the results of the invariant tori obtained by SSM reduction, one may apply the \texttt{Tor} toolbox directly to the full system, just as we did in our first example. This would be, however, impracticable due to the high dimensionality of the problem. For instance, if we use 10-harmonics approximation in~\texttt{Tor}, we will have a multisegment boundary-value problem with 21 coupled segments, resulting in a 3360-dimensional phase space. Here we consider an alternative validation approach. Specifically, we check the convergence of the results of the invariant tori with the increment of the expansion orders of the SSM. Indeed, Theorem~\ref{th:SSM-existence-uniqueness} guarantees the existence and uniqueness of such an SSM and then the calculated SSM should converge to the unique one as the expansion order increases.

We have set the expansion order to be seven in previous computations such that the results of periodic orbits obtained by SSM-reduction match well with that of the collocation method. Here we focus on two representative tori, including a stable one with $\Omega=15.75$ and an unstable one with $\Omega=15.5905$, denoted as A and B, respectively, in Fig.~\ref{fig:Bernoulli_tori_34}, and study their convergence with respect to the expansion orders of the SSM. As can be seen in the first two panels of Fig.~\ref{fig:Bernoulli_ssm_convg}, the internal frequency (denoted by $\omega_\mathrm{s}$, which is related to the period $T_\mathrm{s}$ of the limit cycle in the reduced-order model via $\omega_\mathrm{s}=2\pi/T_\mathrm{s}$; see~Fig.~\ref{fig:Bernoulli_TandRho}) and the amplitude of each invariant torus converges fast with the increase of the expansion orders. In particular, the results for quasi-periodic responses converge well when the order is 11 or higher, and the results at order 7 are already close to the converged results.

\begin{figure}[!ht]
\centering
\includegraphics[width=0.39\textwidth]{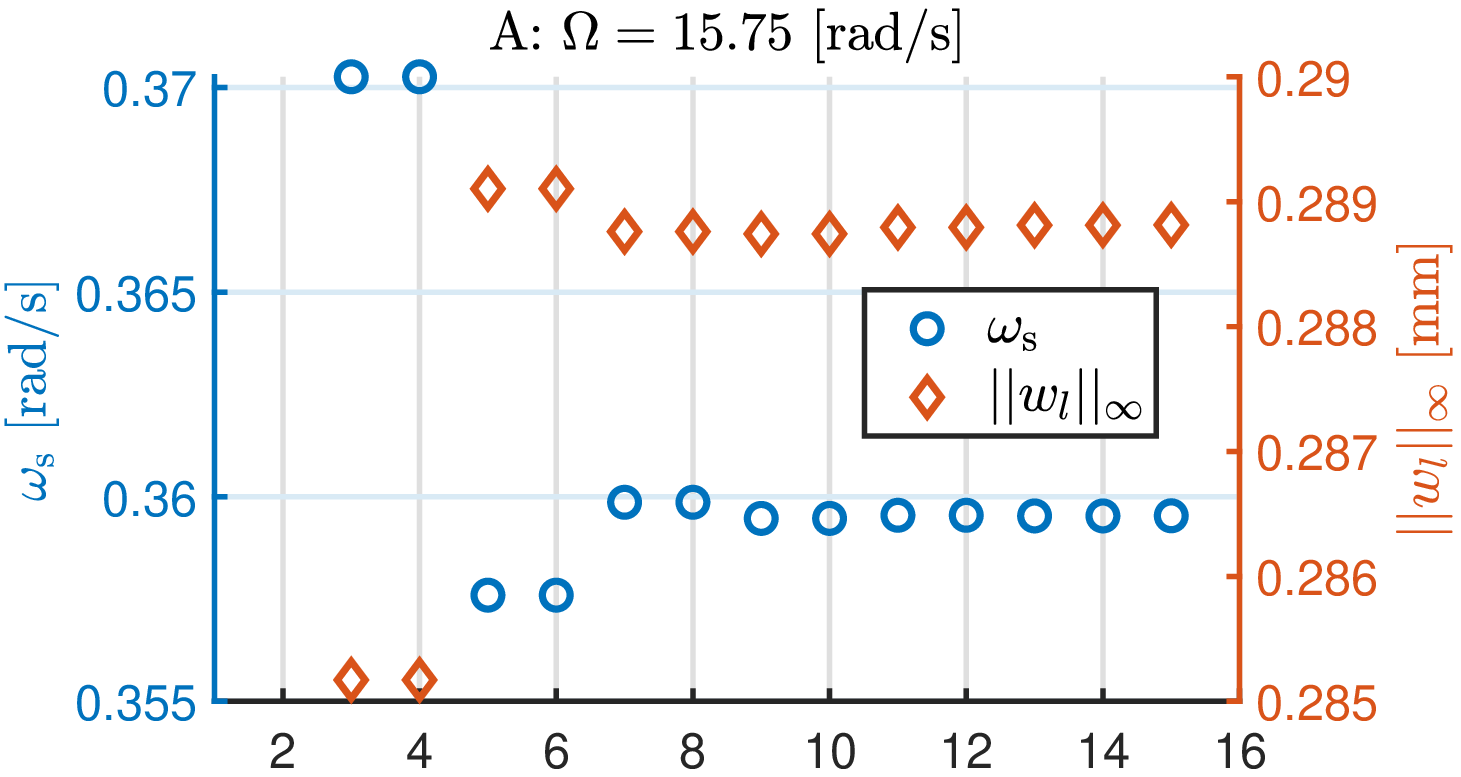}
\includegraphics[width=0.39\textwidth]{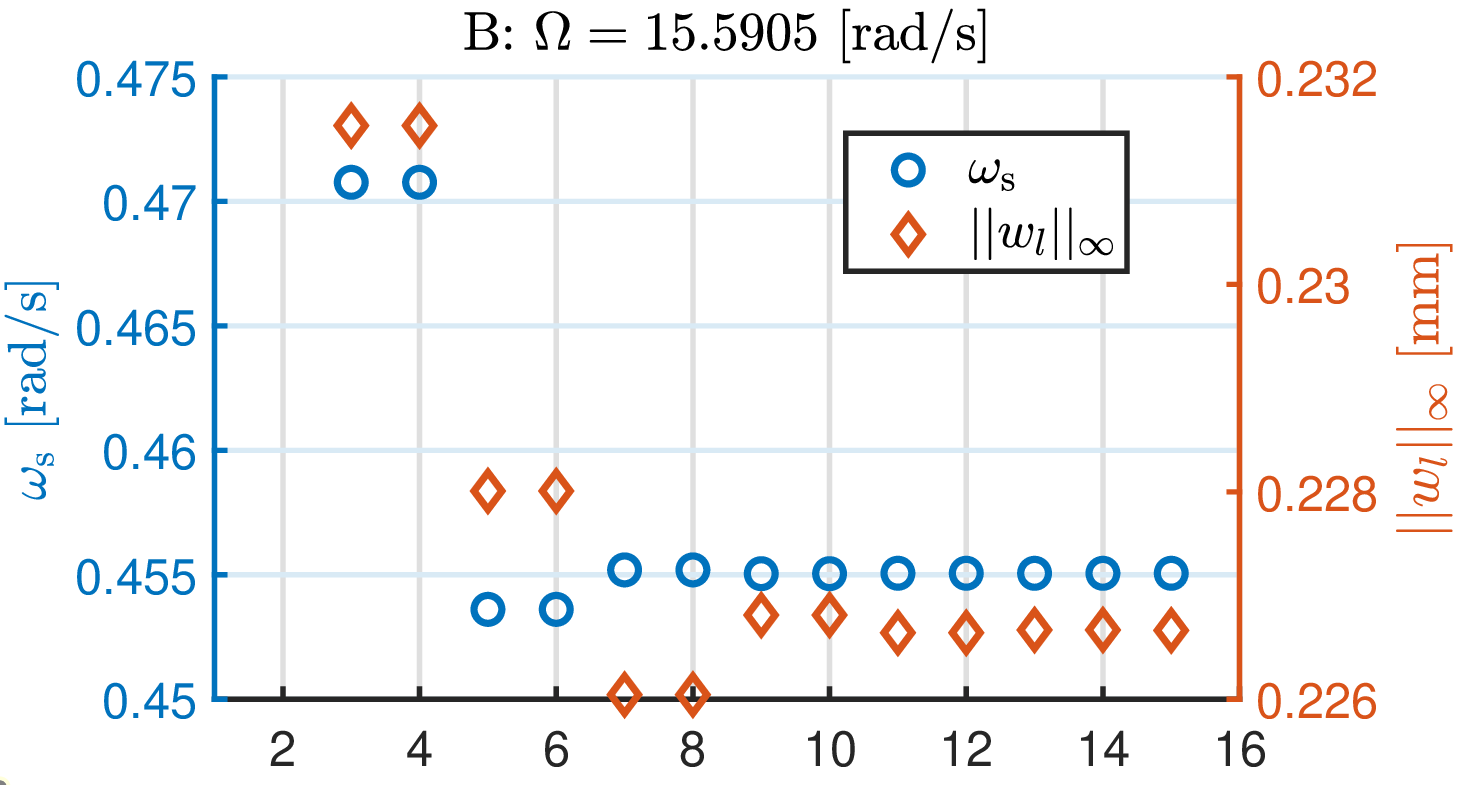}\\
\includegraphics[width=0.41\textwidth]{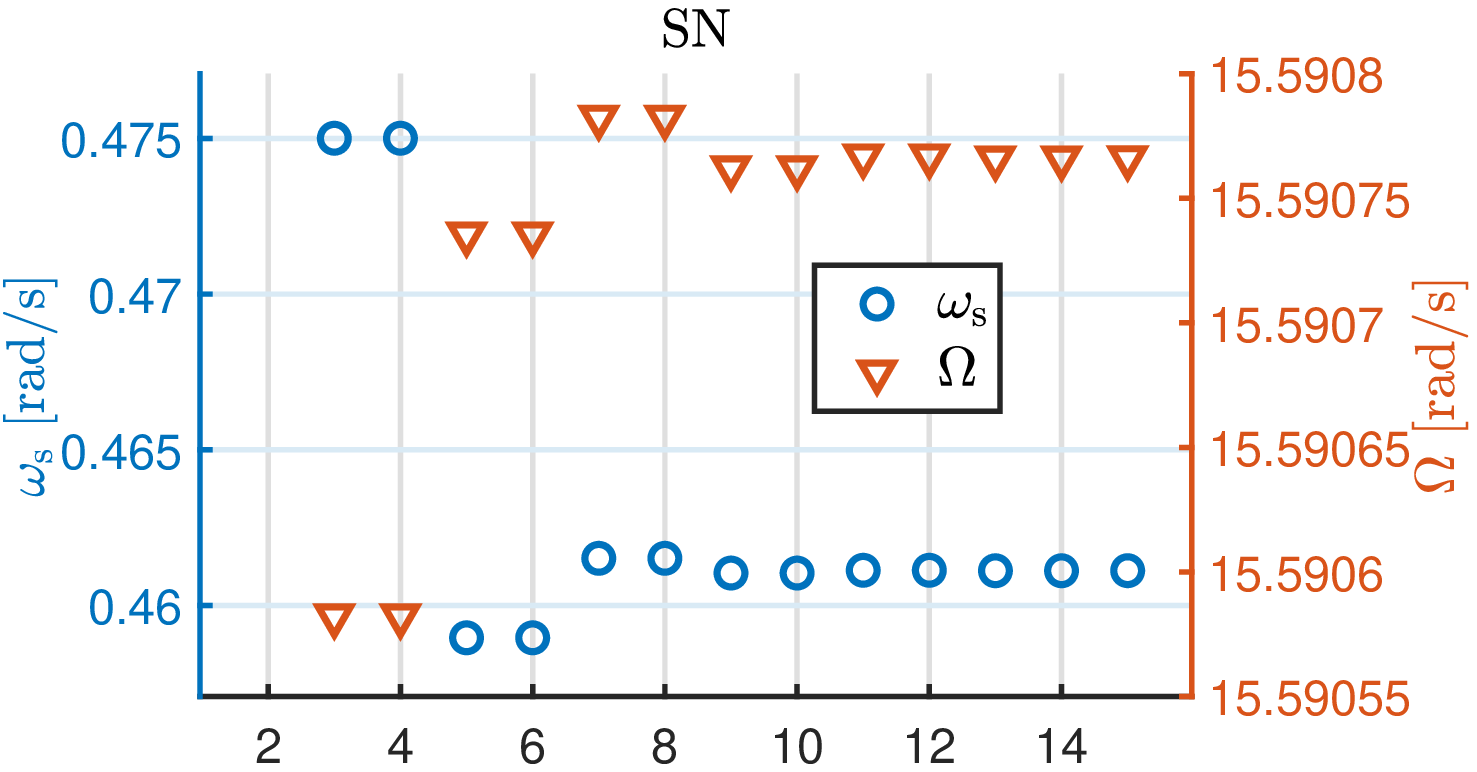}
\includegraphics[width=0.41\textwidth]{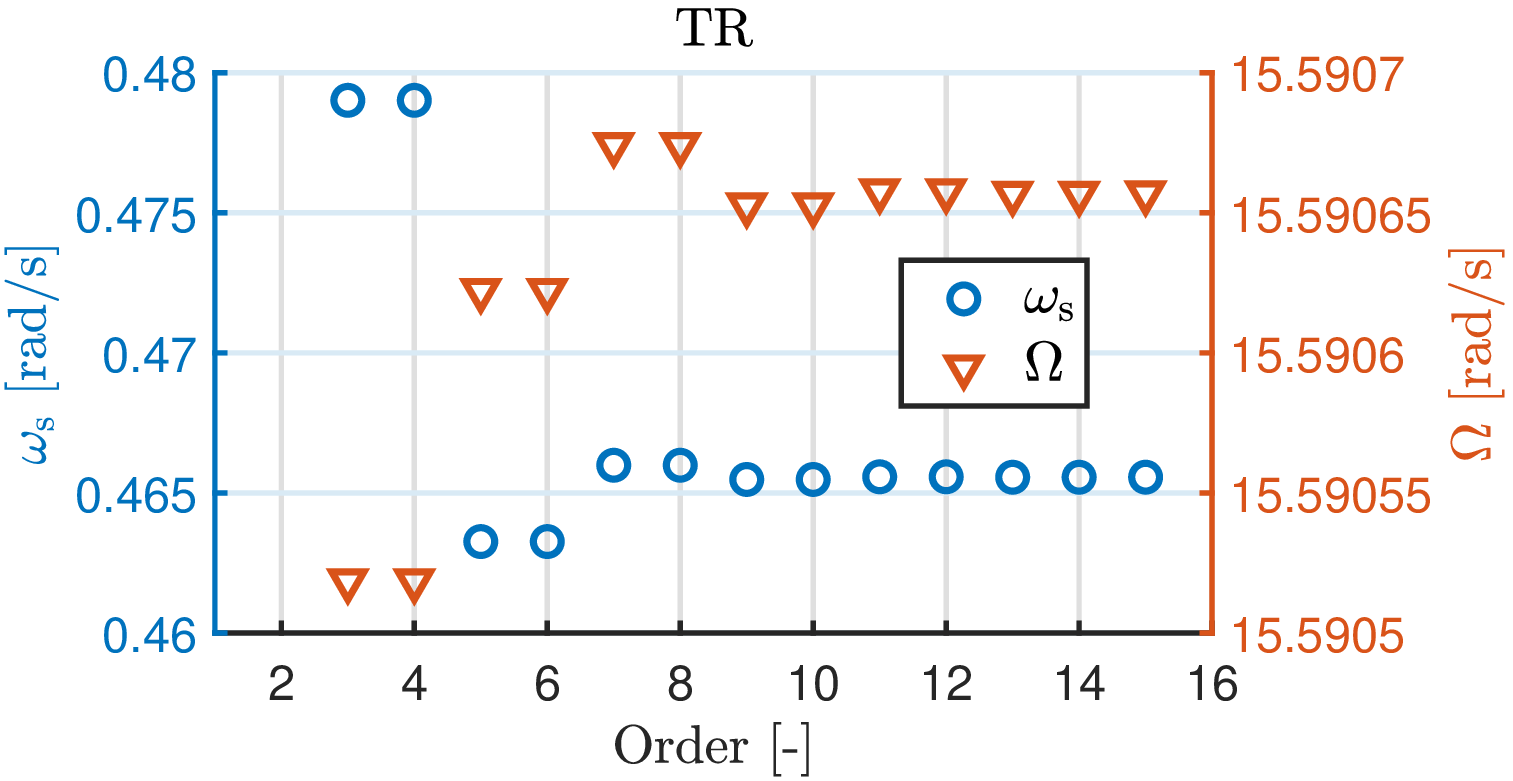}
\caption{Convergence of invariant torus solutions to the Bernoulli beam example under increasing expansion orders for the SSM. The first two panels show the convergence of the internal frequency and amplitude of tori A and B (see~Fig.~\ref{fig:Bernoulli_tori_34}), and the last two panels display the convergence of the two bifurcated invariant tori (cf.~Fig.~\ref{fig:Bernoulli_TandRho}).}
\label{fig:Bernoulli_ssm_convg}
\end{figure}

Recall that we have detected SN and TR bifurcation limit cycles in the leading-order SSM-reduced dynamics, as seen in Fig.~\ref{fig:Bernoulli_TandRho}. These correspond to SN and HB bifurcations of quasi-periodic orbits in the full system, respectively. We also check the convergence of these two types of invariant tori under increasing expansion orders of the SSM. The last two panels in Fig.~\ref{fig:Bernoulli_ssm_convg} present the two frequency components of these two invariant tori as functions of the expansion order. Again, these two invariant tori are converged when the expansion order is 11 or higher. The remaining computations in this example were performed with SSM expansion order 11.

\begin{sloppypar}
A unique two-dimensional invariant torus bifurcates from a TR bifurcation limit cycle, yielding a family of two-dimensional invariant tori under variations of $\Omega$ or $\epsilon$. We perform this continuation of two-dimensional invariant tori in the leading-order SSM-reduced model~\eqref{eq:ode-reduced-slow-cartesian-leading}. Specifically, we switch from the continuation of periodic orbits of the reduced-order model to the continuation of invariant tori at the TR bifurcation periodic orbit in Fig.~\ref{fig:Bernoulli_TandRho}. The two-dimensional invariant tori above correspond to the three-dimensional invariant tori of the full system. For such a invariant 3-torus, three frequency components exist: the \emph{external} excitation frequency $\Omega$ and two \emph{internal} frequencies $\omega_{1,\mathrm{s}}$ and $\omega_{2,\mathrm{s}}$. Here $2\pi/\omega_{2,\mathrm{s}}$ is of the same order as the period of the TR bifurcation periodic orbit. The dependence of \emph{internal} frequencies on the \emph{external} frequency is shown in Fig.~\ref{fig:Bernoulli_slow_T2_om1and2}.
We have $\mathcal{O}(\Omega)\sim10$, $\mathcal{O}(\omega_{2,\mathrm{s}})\sim0.1$ and $\mathcal{O}(\omega_{1,\mathrm{s}})\sim0.001$. Therefore, we have $\Omega\gg\omega_{2,\mathrm{s}}\gg\omega_{1,\mathrm{s}}$ and hence there exist three time scales in the computed three-dimensional invariant tori.
\end{sloppypar}

\begin{figure}[!ht]
\centering
\includegraphics[width=0.48\textwidth]{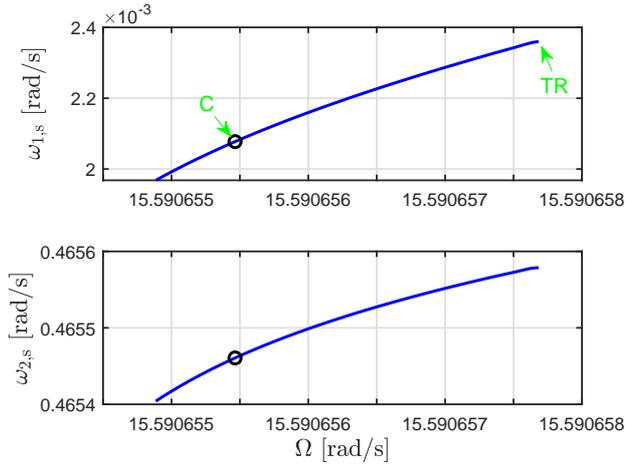}
\caption{Projections of the continuation path of two-dimensional tori in leading-order reduced dynamics~\eqref{eq:ode-reduced-slow-cartesian-leading} of the discretized Bernoulli beam model. The two panels here present the two internal frequencies of the tori as functions of the external frequency $\Omega$.}
\label{fig:Bernoulli_slow_T2_om1and2}
\end{figure}

\begin{sloppypar}
We can further obtain the corresponding three-dimensional invariant tori in the physical coordinates. As an illustration, the three-dimensional invariant torus corresponding to the two-dimensional invariant torus C in the leading-order SSM-reduced model (see~Fig.~\ref{fig:Bernoulli_slow_T2_om1and2}) is computed and its projection on $(w_{0.5l},w_l,\dot{w}_l)$ is plotted in the first panel of Fig.~\ref{fig:Bernoulli_T3_C}, where $w_{0.5l}$ denotes the transverse deflection at the midspan of the beam. For such a three-dimensional invariant torus, the intersection points of the period-$2\pi/\Omega$ Poincar\'e section with the torus form two-dimensional invariant torus, shown for torus C in the lower panel of Fig.~\ref{fig:Bernoulli_T3_C}.
\end{sloppypar}

\begin{figure}[!ht]
\centering
\includegraphics[width=0.45\textwidth]{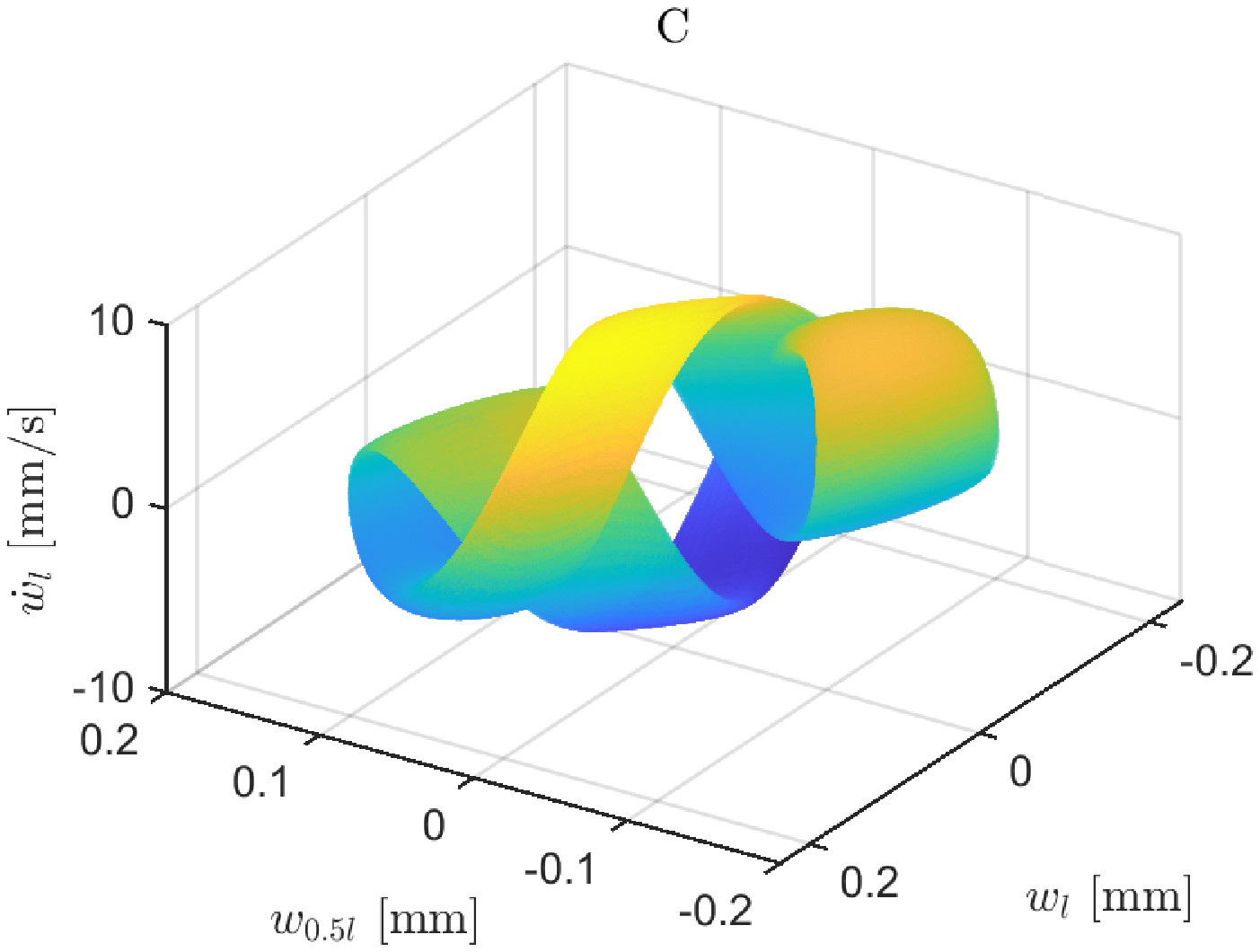}
\includegraphics[width=0.45\textwidth]{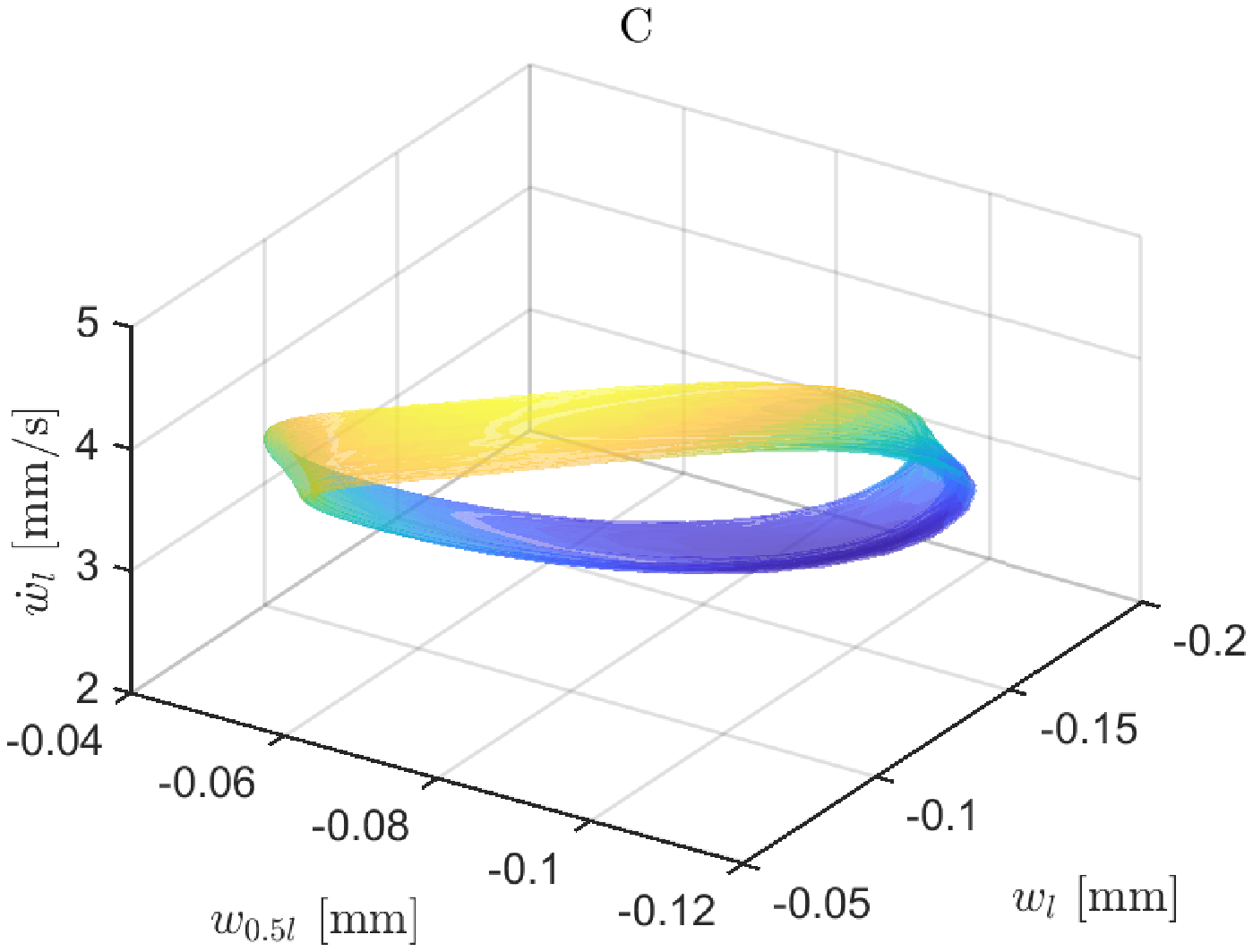}
\caption{Visualization of the three-dimensional invariant torus corresponding to the two-dimensional invariant torus C in the reduced-order model, cf.~Fig.~\ref{fig:Bernoulli_slow_T2_om1and2}, and its invariant surface of the period-$2\pi/\Omega$ map.}
\label{fig:Bernoulli_T3_C}
\end{figure}

\subsection{A forced von K\'arm\'an beam with support spring}
\label{sec:vonKarmanBeam}
In our third example, we revisit the von K\'arm\'an beam example studied in Part I, where the FRCs for periodic orbits of the beam discretized with various numbers of elements were computed. We demonstrated a significant speed-up gain for SSM-reduction relative to other techniques, including harmonic balance and the collocation. Here, we focus on the case of quasi-periodic orbits of the beam discretized with a different number of elements. As we will see, the calculation of FRCs for quasi-periodic orbits via SSM-reduction is also effective and efficient.

\begin{sloppypar}
Consider a clamped-pinned von K\'arm\'an beam with a supporting \emph{linear} spring at its midspan. The nonlinearity in this model comes from the axial stretching due to large transverse deflections. Unlike in the Bernoulli beam example, the nonlinearity here is distributed.
The geometric and material properties of this von K\'arm\'an beam are the same as those of the Bernoulli beam in the previous example, namely, we have $h=b=10\,\mathrm{mm}$, $l=2700\,\mathrm{mm}$, $\rho= 1780\times10^{-9}\,\mathrm{kg/{mm}^3}$ and $E=45\times10^6\,\mathrm{kPa}$, where $b,h,l$ are the width, height and length of the beam, $\rho$ is the density and $E$ denotes the Young's modulus. Following the finite element discretization in~\cite{jain2018exact,FEcode}, three degrees of freedom are introduced at each node: the axial and transverse displacements, and the rotation angle. The equation of motion can be written in the same form as~\eqref{eq:eom-beams} but with $\boldsymbol{x}\in\mathbb{R}^{3N_{\mathrm{e}}-2}$, where $N_\mathrm{e}$ again denotes the number of elements of the discrete model. We again use Rayleigh damping $\boldsymbol{C}=\alpha\boldsymbol{M}+\beta\boldsymbol{K}$ but with $\alpha=0$ and $\beta=\frac{2}{9}\times10^{-5}\,\mathrm{s}^{-1}$. A harmonic transverse force $\epsilon F\cos\Omega t$ is applied at the midspan of the beam.

Let the stiffness of the support spring be $k_\mathrm{s}$. We can tune the value of $k_\mathrm{s}$ to trigger 1:3 internal resonance between the first two bending modes. Indeed, we can set $k_\mathrm{s}=37\,\mathrm{kg/s^2}$ such that $\omega_2\approx3\omega_1$, as detailed in Part I. When the beam is discretized with 100 elements, we have $\omega_1=33.20\,\mathrm{rad/s}$ and $\omega_2=99.59\,\mathrm{rad/s}$. In the following computations, we set $F=1000$ and $\epsilon=0.02$. In Part I, we investigated the primary resonance of the first mode of the beam discretized with various numbers of elements. Specifically, we calculated the FRCs of periodic orbits of the finite element models with $N_{\mathrm{e}}\in\{$8, 20, 40, 100, 200, 500, 1,000, 3,000, 10,000$\}$ using the SSM reduction. The results from the SSM reduction for $N_\mathrm{e}\leq200$ were validated with the results from the full system obtained by the harmonic balance method using \textsc{nlvib} tool~\cite{krack2019harmonic}, the collocation method from the \texttt{po} toolbox of \textsc{coco}~\cite{dankowicz2013recipes} and direct numerical integration.

In the above-mentioned SSM analysis, HB fixed points were detected in the continuation of equilibria in the reduced-order model, indicating the existence of limit cycles that bifurcate from the HB points. Here we calculate the FRC of periodic orbits of the beam discretized with four elements to illustrate such HB points. The FRC is calculated with both the SSM reduction and the collocation method applied to the full system using the \texttt{po} toolbox of \textsc{coco}. The settings for algorithm parameters of the collocation and the continuation here are the same as those adopted in Part I. Namely, the number of subintervals is 10, the number of base points in each subinterval is 5, the maximal step size is 100 and the maximal residual for the predictor is 10. We show the resulting FRCs in the first two panels of Fig.~\ref{fig:FRCs-vonBeam-4}, where the dependence of the amplitudes of the transverse vibration at the 1/4 and midspan of the beam on $\Omega$ are plotted. The FRCs obtained by the two methods match well, with two HB fixed points detected around $\Omega=34.5$.
\end{sloppypar}
\begin{figure}[!ht]
\centering
\includegraphics[width=0.45\textwidth]{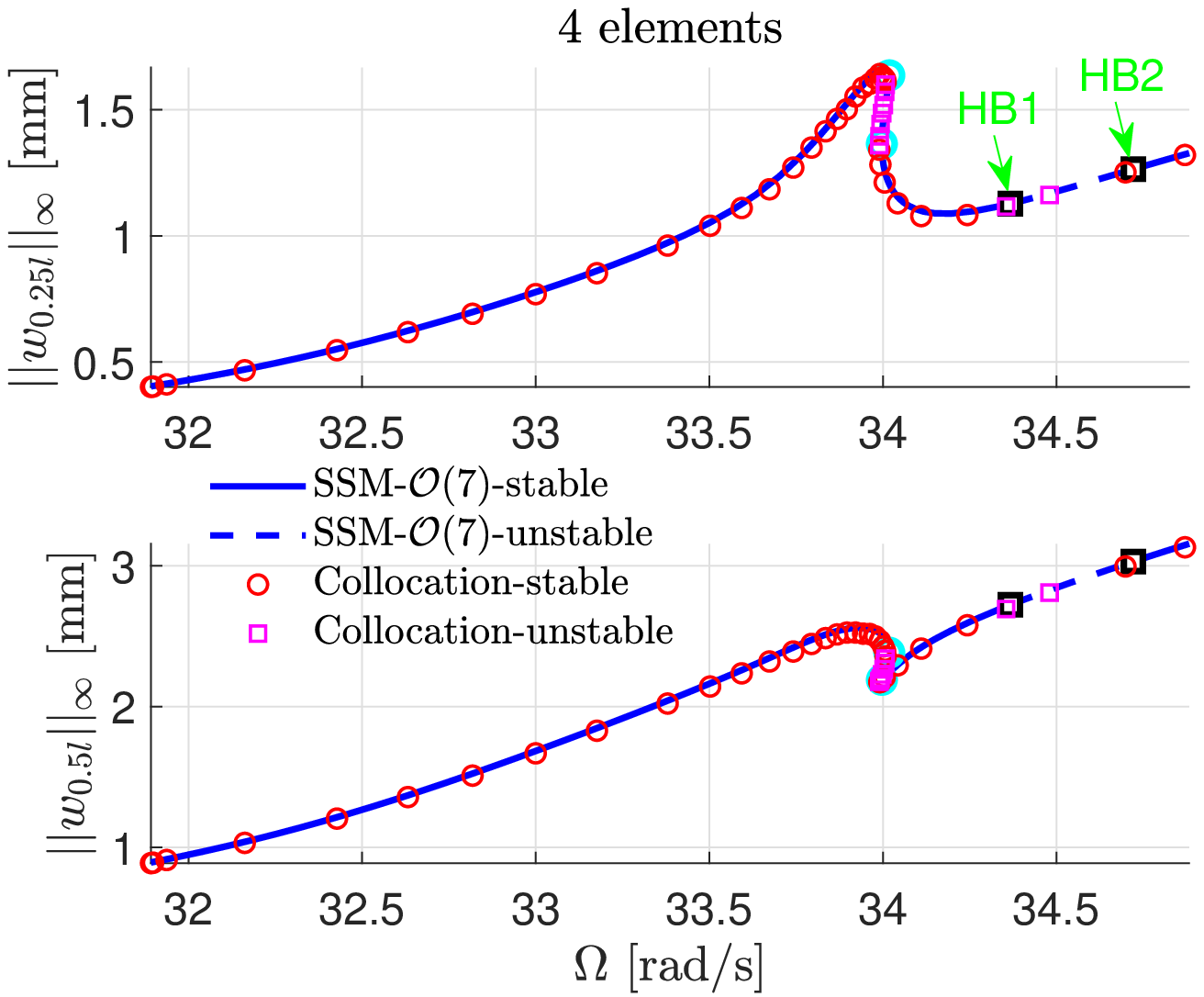}
\includegraphics[width=0.45\textwidth]{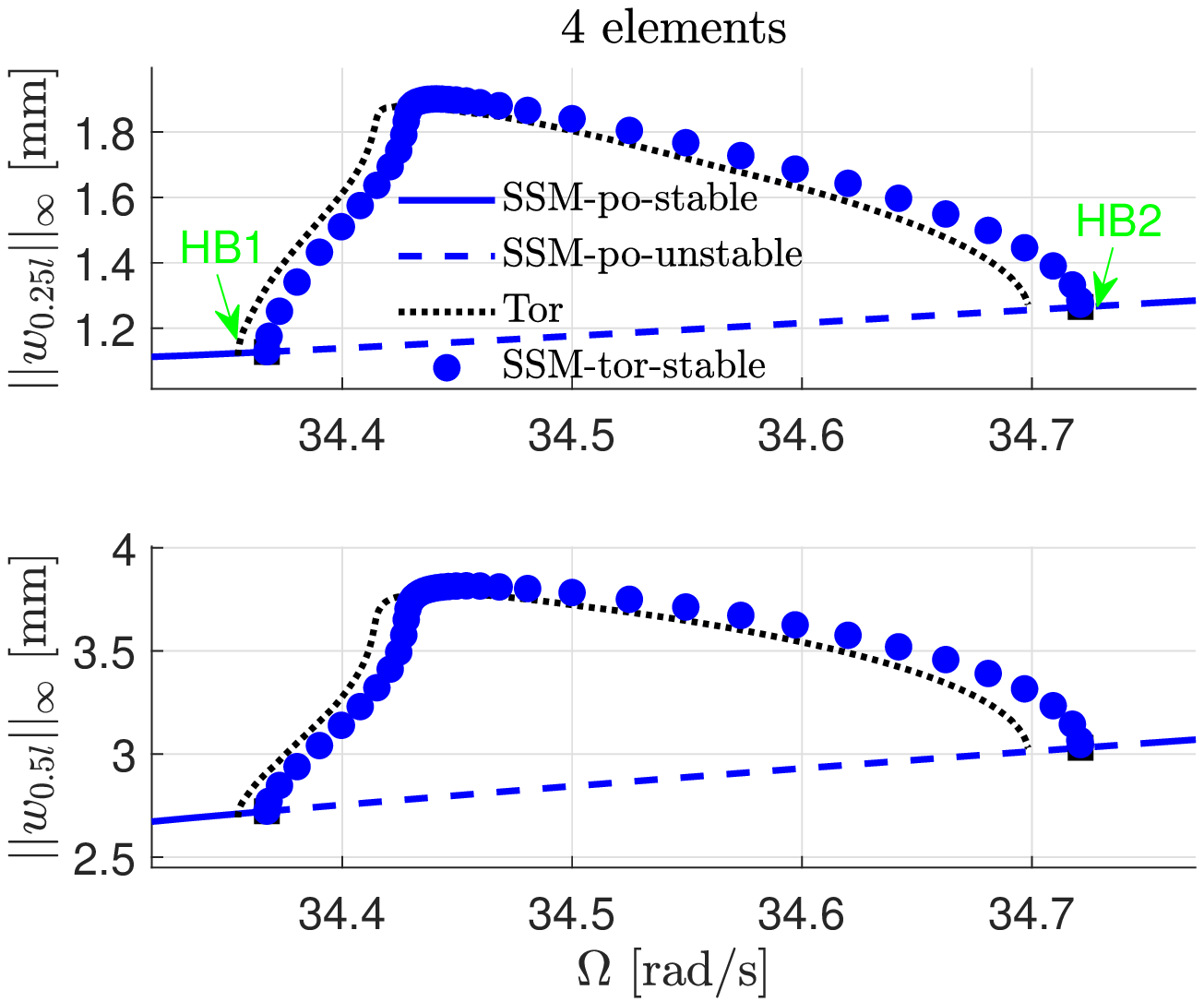}
\caption{FRCs for the periodic and quasi-periodic responses of transverse deflections at the 1/4 and midspan ($w$ at $0.25l$ and $0.5l$) of the clamped-pinned von K\'arm\'an beam discretized with four elements.}
\label{fig:FRCs-vonBeam-4}
\end{figure}

We switch from the continuation of equilibria in the reduced-order model to the continuation of limit cycles at HB1. Such a continuation run proceeds until $\Omega$ approaches the other Hopf bifurcation point, HB2. A one-parameter family of limit cycles in the reduced-ordre model is obtained and all of them are stable. For each limit cycle in the reduced-order model, an invariant torus of the full system is constructed. The FRCs for the quasi-periodic orbits that stay on these invariant tori have been plotted in the last two panels of Fig.~\ref{fig:FRCs-vonBeam-4}. The FRCs of the invariant tori intersect the FRCs of the periodic orbits at the two HB points.

We also perform the continuation of invariant tori of the full system directly using the \texttt{Tor} toolbox to validate the results of the SSM reduction. In the continuation run of periodic orbits of the full system with the \texttt{po} toolbox, two TR bifurcations of periodic orbits are found with modified event functions. The details of this modification are given in Appendix~\ref{sec:event-bif-po}. We then switch from the continuation of the periodic orbits of the full system to the continuation of the two-dimensional invariant tori of the full system at a TR bifurcation of periodic orbits. Here the number of harmonics used in the approximation of invariant tori is 10. As seen in the last two panels of Fig.~\ref{fig:FRCs-vonBeam-4}, the FRCs of quasi-periodic orbits obtained by the two methods match well overall, and only small discrepancies around the two bifurcation points were observed. In the above computation of FRCs of quasi-periodic orbits, the computation time for the SSM reduction is about one minute while the one for the \texttt{Tor} toolbox is more than two hours, which demonstrates the significant speed-up gain of the SSM reduction. In addition, numerical experiments suggest that the switch of continuation from periodic orbits to invariant tori in the \texttt{Tor} toolbox fails if the system is high-dimensional. In contrast, the dimension of the reduced-order model on the resonant SSM is four, independently of the number of elements of the discrete model.

Next, we have a close look at the discrepancies of the quasi-periodic responses obtained by the two methods around the two HB points, as seen in the last two panels of Fig.~\ref{fig:FRCs-vonBeam-4}. Four possible factors leading to these discrepancies are: i) the expansion order of the autonomous part of the SSM is not high enough, ii) the leading-order approximation of the non-autonomous part of the SSM is not sufficient to yield accurate results, iii) the number of harmonics used in the direct computation with the \texttt{Tor} toolbox is not enough, and iv) the uniform, fixed mesh used in the collocation method is not of high enough fidelity to yield accurate results. We increase the number of harmonics from 10 to 20 in the computation with the \texttt{Tor} toolbox. The results with 20 harmonics are in agreement with those from 10 harmonics. Therefore, the potential factor iii) can be ruled out. Factor ii) is out of reach in this study and we hence will focus on the remaining two factors in the following two paragraphs.

\begin{sloppypar}
Note that the numerical discrepancies can be mainly characterized by the positions of the two HB points, namely, by the critical $\Omega$ values where HB occurs. We denote these two critical frequencies as $\Omega_{\mathrm{HB1}}$ and $\Omega_{\mathrm{HB2}}$. We have
\begin{equation}
    \Omega_{\mathrm{HB1}}^{\mathrm{SSM-7}}=34.367\,\mathrm{rad/s},\,\, \Omega_{\mathrm{HB2}}^{\mathrm{SSM-7}}=34.721\,\mathrm{rad/s}
\end{equation}
for the SSM reduction with expansion order equal to 7 (cf.~Fig.~\ref{fig:FRCs-vonBeam-4}), and
\begin{equation}
    \Omega_{\mathrm{HB1}}^{\mathrm{coll}}=34.355\,\mathrm{rad/s},\,\, \Omega_{\mathrm{HB2}}^{\mathrm{coll}}=34.699\,\mathrm{rad/s}
\end{equation}
for the collocation method with the fixed mesh. The relative errors in terms of $\Omega_{\mathrm{HB1}}$ and $\Omega_{\mathrm{HB2}}$ obtained by the two methods are 0.035\% and 0.063\%, respectively, which are quite small. The convergence of the two critical frequencies with the increment of the expansion order of the SSM is presented in Fig.~\ref{fig:vonKarmanHBorders}, from which we see that the two critical frequencies are converged when the expansion order is 11 or higher. The convergent critical frequencies are given by
\begin{equation}
    \Omega_{\mathrm{HB1}}^{\mathrm{SSM*}}=34.364\,\mathrm{rad/s},\,\, \Omega_{\mathrm{HB2}}^{\mathrm{SSM*}}=34.717\,\mathrm{rad/s}.
\end{equation}
The relative errors in terms of $\Omega_{\mathrm{HB1}}$ and $\Omega_{\mathrm{HB2}}$ are updated as 0.026\% and 0.052\% after the substitution of the converged critical frequencies. Therefore, the results at order 7 are very close to the converged results, and increasing the expansion order of SSM slightly decreases the discrepancies.
\end{sloppypar}

\begin{figure}[!ht]
\centering
\includegraphics[width=0.4\textwidth]{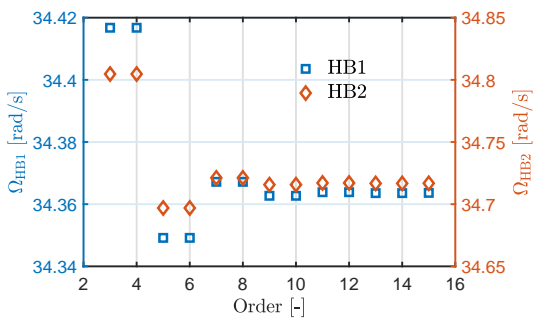}
\caption{Convergence of critical frequencies for the two HB points in FRC of the clamped-pinned von K\'arm\'an beam with increasing expansion order for the SSM.}
\label{fig:vonKarmanHBorders}
\end{figure}

We also apply the collocation method with adaptive meshes to the full system to obtain the FRC of the periodic orbits. Specifically, the mesh is initialized as before (10 subintervals with five base points in each subinterval) and then adaptively changed after every continuation step. In this run, the two critical frequencies are obtained as
\begin{equation}
    \Omega_{\mathrm{HB1}}^{\mathrm{coll*}}=34.352\,\mathrm{rad/s},\,\, \Omega_{\mathrm{HB2}}^{\mathrm{coll*}}=34.718\,\mathrm{rad/s}.
\end{equation}
The relative errors for $\Omega_{\mathrm{HB1}}$ and $\Omega_{\mathrm{HB2}}$ are further updated as 0.035\% and 0.003\% after the substitution of the collocation results with adaptive mesh.
So the discrepancy around the second HB point (see~Fig.~\ref{fig:FRCs-vonBeam-4}) can be significantly reduced if we use the collocation method with adaptive mesh. These discrepancies might be further reduced when the contribution of higher-order non-autonomous SSM is considered.

Note that it is impractical to switch from the continuation of periodic orbits of the full system to the continuation of invariant tori using the collocation method with adaptive mesh. In particular, the number of base points for the critical periodic orbit at HB1 is increased from 41 to 401 when the mesh is allowed to be changed, and then the computational time for a {single} continuation step of the invariant tori is increased to a few hours. In addition, the Newton iteration of such high-dimensional nonlinear equations for the continuation problem of the invariant tori fails to converge at the switch point. These challenges are circumvented in the SSM reduction.

We now increase the number of elements to further demonstrate the effectiveness and efficiency of the SSM reduction. For each discretized beam model, the above SSM analysis is performed. Specifically, the calculation of FRC for the periodic orbits of a full system is conducted first. With the detected HB fixed points in the corresponding leading-order SSM-reduced model at hand, we switch to the computation of the FRC for quasi-periodic orbits of the full system. We still use the expansion order 7 because the results already have good accuracy (see~Fig.~\ref{fig:vonKarmanHBorders} and the Fig.~11 of Part I) and this order is consistent with the computation of this example in Part I. The FRCs obtained in this fashion for quasi-periodic orbits are presented in Fig.~\ref{fig:FRCs-vonBeam-physics}, where the number of elements is increased from 8 to 200. Results for $N_{\mathrm{e}}\in\{$500, 1,000, 3,000, 10,000$\}$ are not plotted here because the results at $N_{\mathrm{e}}=200$ are already converged with respect to increases in the number of elements.

\begin{figure*}[!ht]
\centering
\includegraphics[width=0.45\textwidth]{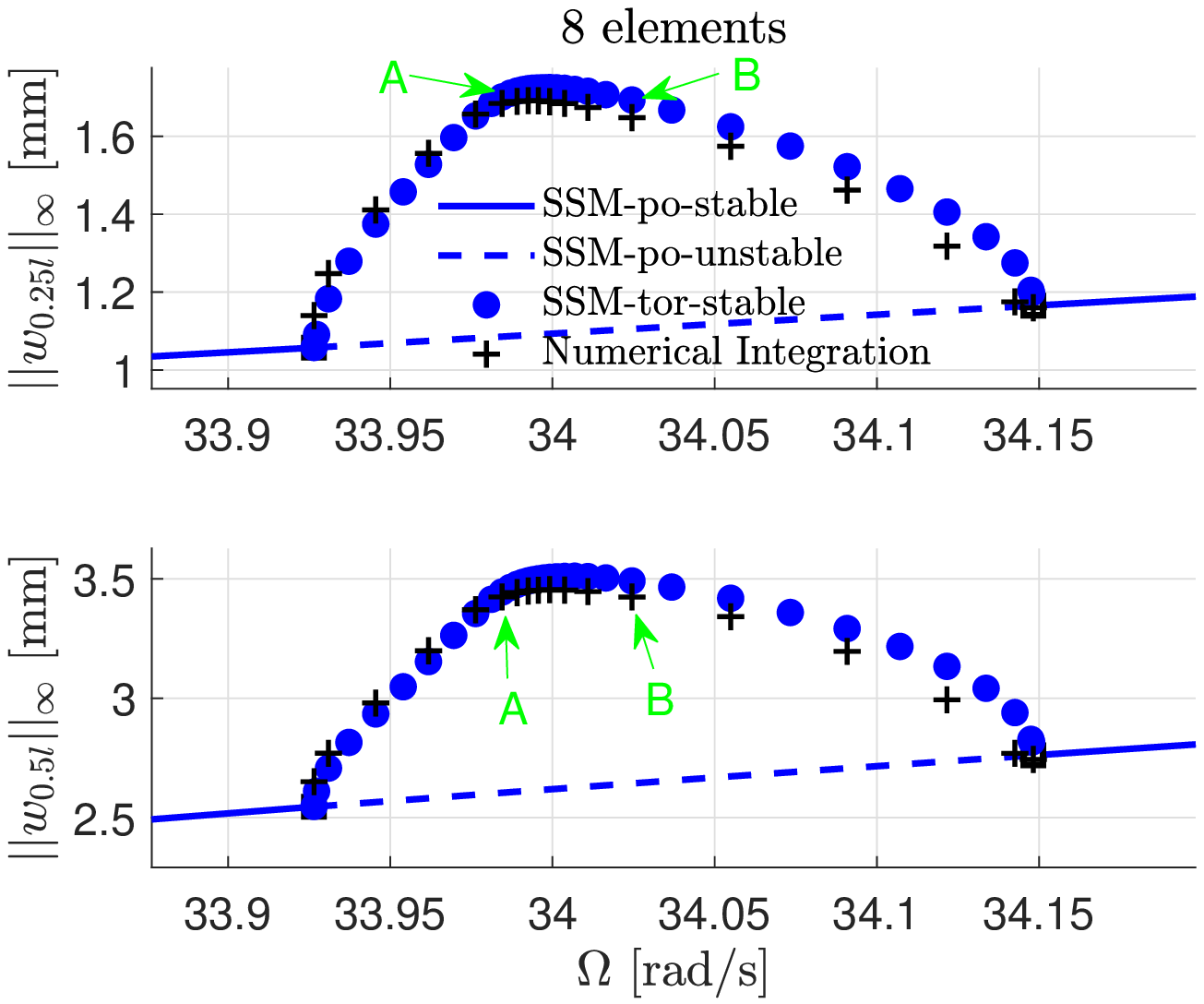}
\includegraphics[width=0.45\textwidth]{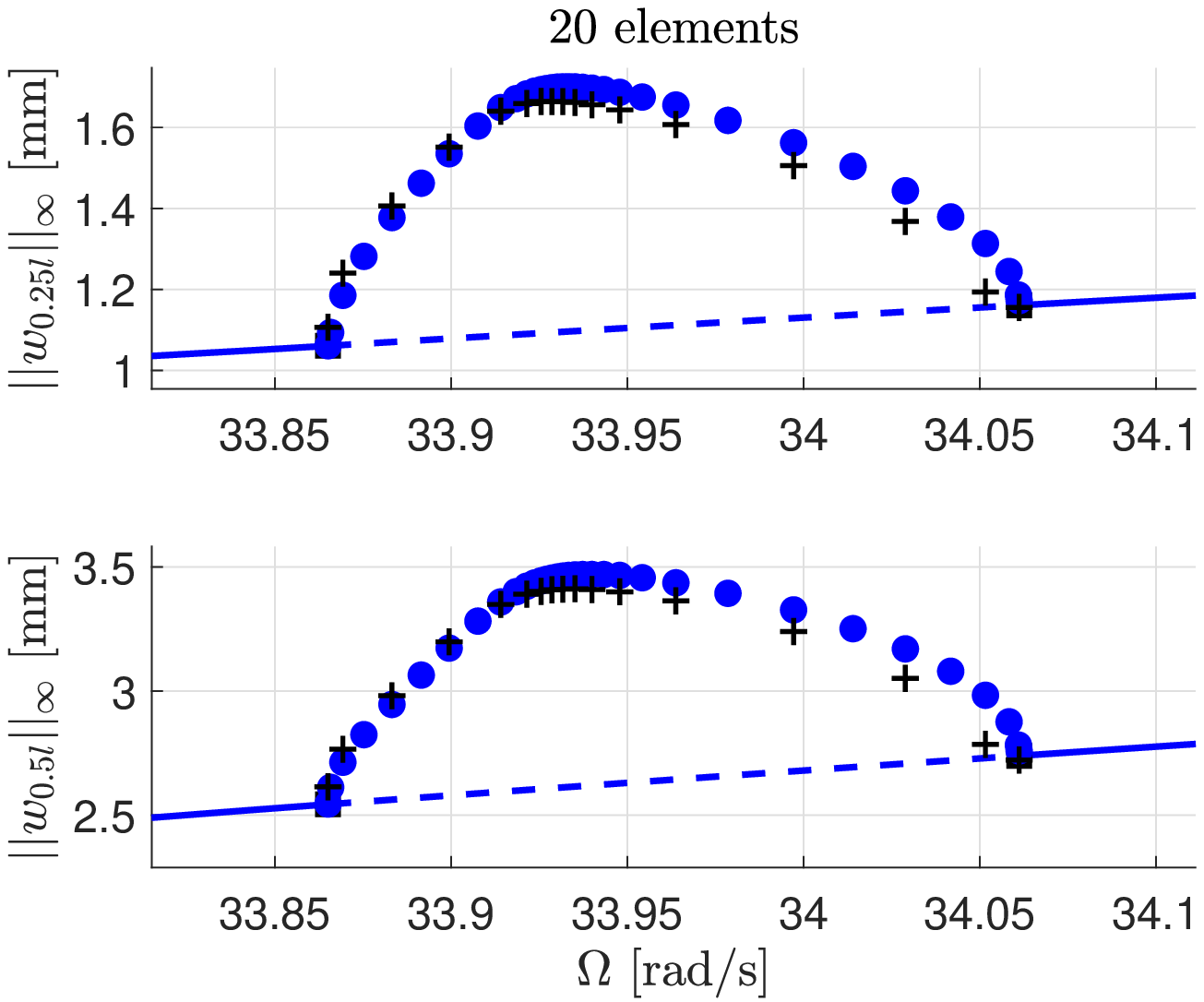}\\
\includegraphics[width=0.45\textwidth]{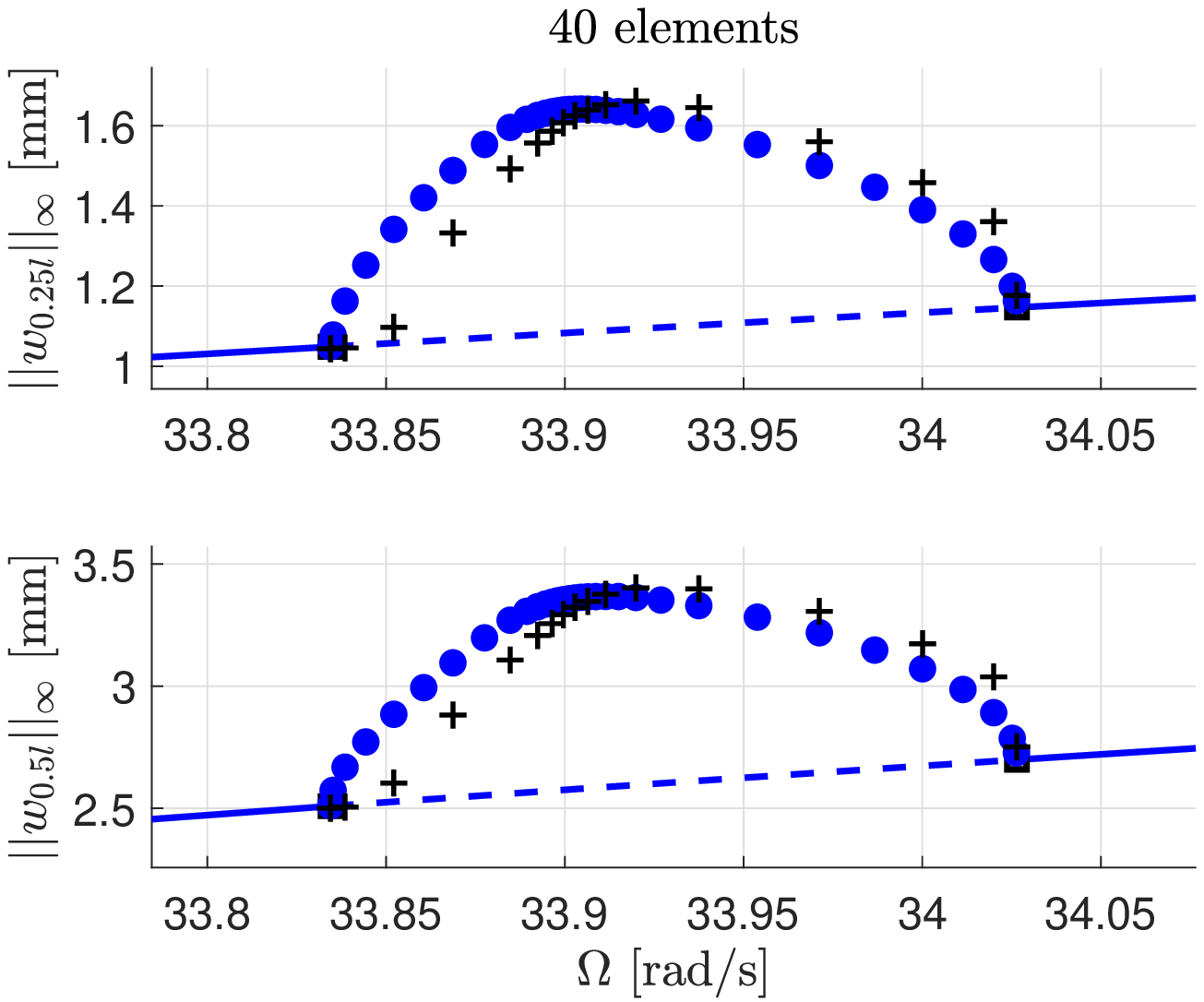}
\includegraphics[width=0.45\textwidth]{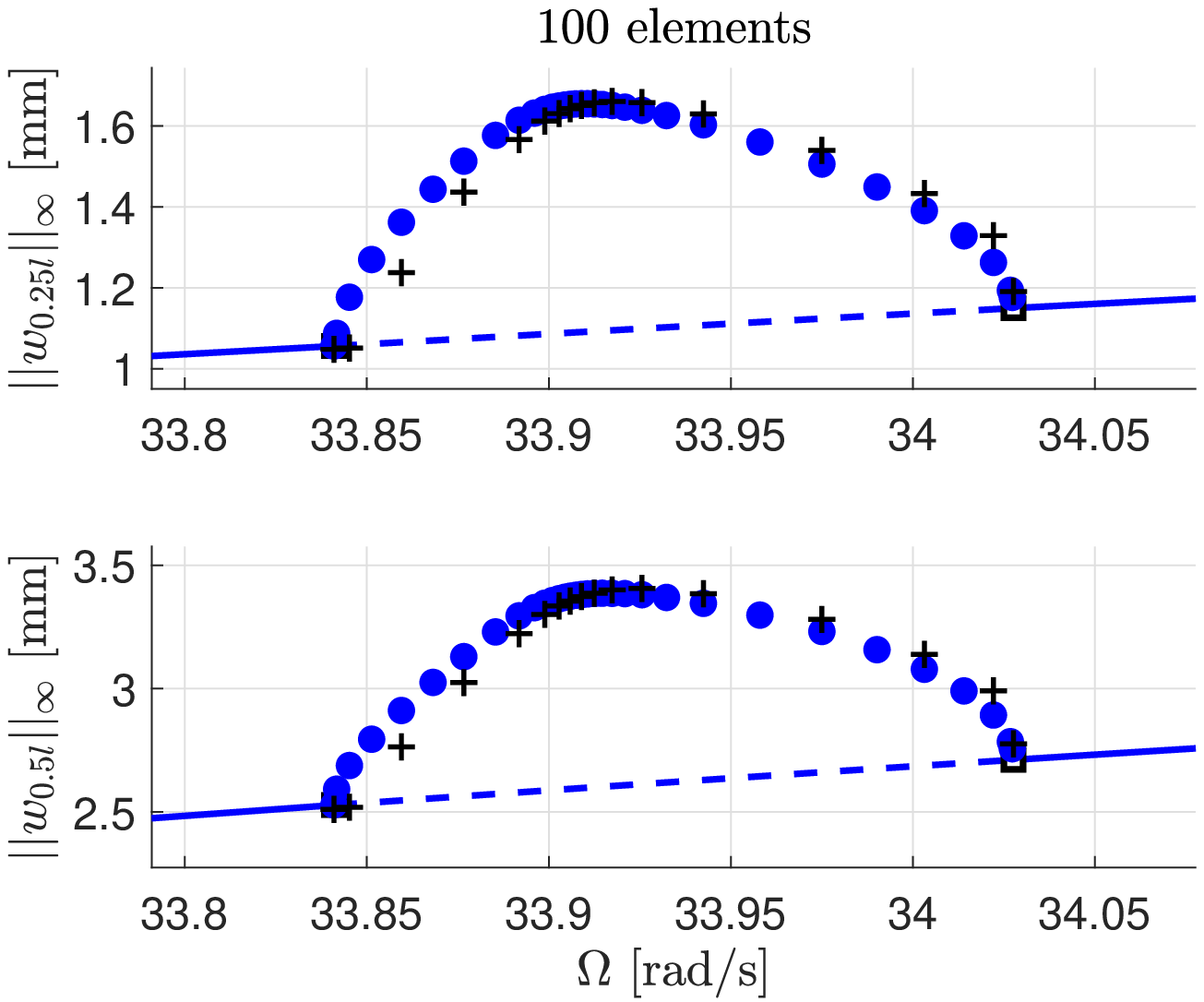}\\
\includegraphics[width=0.45\textwidth]{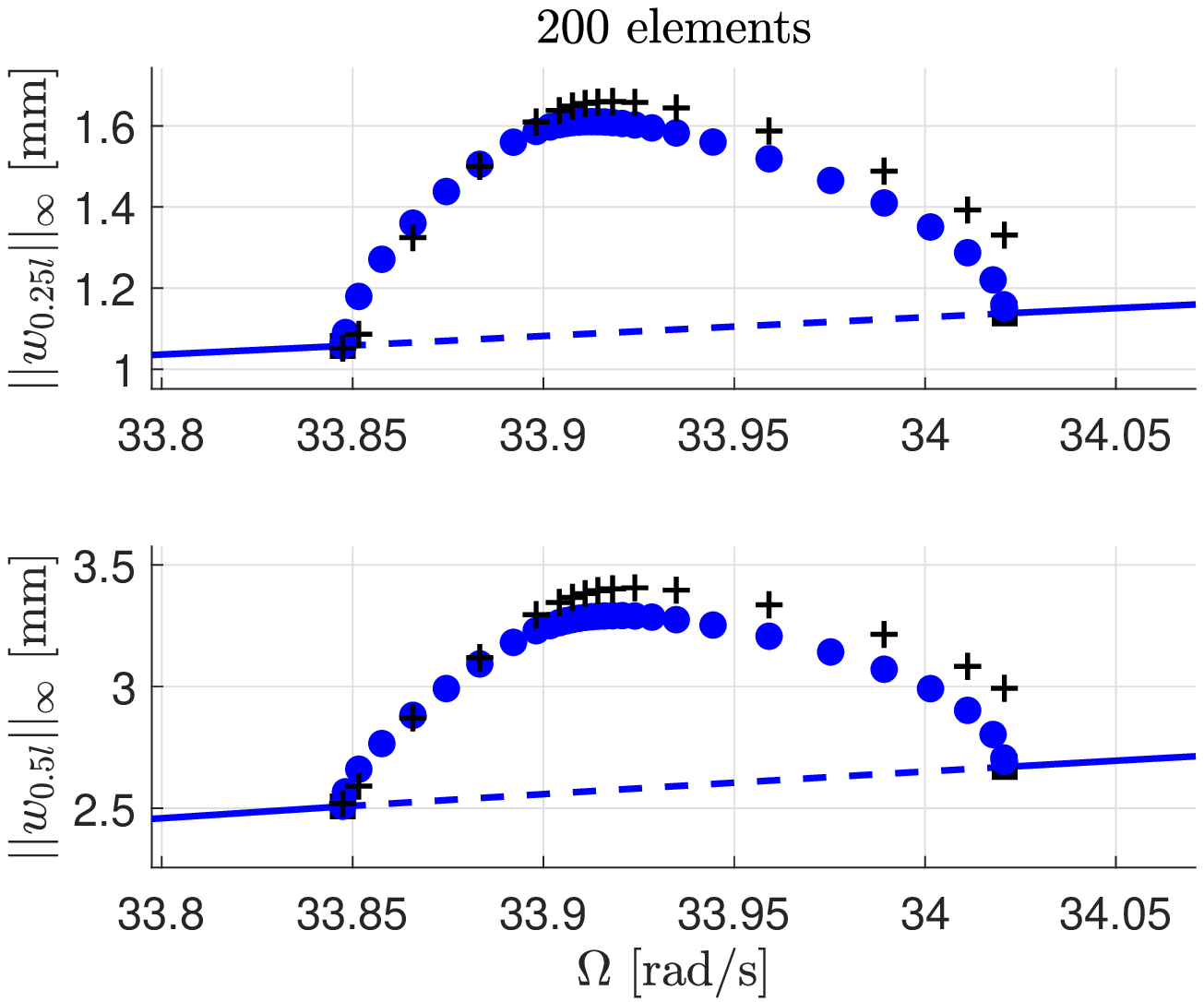}
\caption{FRCs in physical coordinates (the amplitude of transverse displacement $w$ at $0.25l$ and $0.5l$) of the clamped-pinned von K\'arm\'an beam discretized with different numbers of elements. Here we are only concerned with quasi-periodic responses. The reader may refer to Fig.~13 in Part I for the validation of periodic responses obtained from SSM reduction.}
\label{fig:FRCs-vonBeam-physics}
\end{figure*}

To validate the results for the tori obtained from SSM reduction, we have performed numerical integration of the \emph{full} system. The location of the invariant torus obtained from the leading-order SSM-reduced dynamics~\eqref{eq:ode-reduced-slow-cartesian-leading}, however, is only approximate. Due to the physically relevant weak damping we use, short trajectory segments will show a toroidol shape, even if there is no actual invariant. At the same time, computing longer trajectory segments accurately comes with significant numerical cost in our high-dimensional setting. In Appendix~\ref{sec:app-tf-choice} we describe an algorithm for selecting an optimal simulation time $t_\mathrm{f}$ that addresses both of these issues.

Using the algorithm described in Appendix~\ref{sec:app-tf-choice}, we consider two quasi-periodic responses of the beam discretized with 8 elements: A and B in the upper left panels of Fig.~\ref{fig:FRCs-vonBeam-physics}. For each response trajectory, the projection of the period-$2\pi/\Omega$ map onto the plane $(w_{0.5l},\dot{w}_{0.5l})$ is plotted along with the prediction of the leading-order SSM reduction in Fig.~\ref{fig:vonKarmanABC}. These intersection points stay on an invariant curve of the map if the trajectory stays on a torus. The two panels in Fig.~\ref{fig:vonKarmanABC} present the two representative cases of the results of the numerical integration. In the first panel, the trajectory of the numerical integration stays very close to the torus predicted by the SSM reduction throughout the simulation, indicating high accuracy of the SSM reduction. In the second panel of Fig.~\ref{fig:vonKarmanABC}, the trajectory of the numerical integration converges to an invariant torus which is close to the one predicted by the SSM reduction, as indicated by the closeness of the two closed curves. We also apply the algorithm in Appendix~\ref{sec:app-tf-choice} to other tori. As can be seen in Fig.~\ref{fig:FRCs-vonBeam-physics}, the FRCs for quasi-periodic orbits from SSM reduction match well with the numerical integration of the full system.

\begin{figure}[!ht]
\centering
\includegraphics[width=0.38\textwidth]{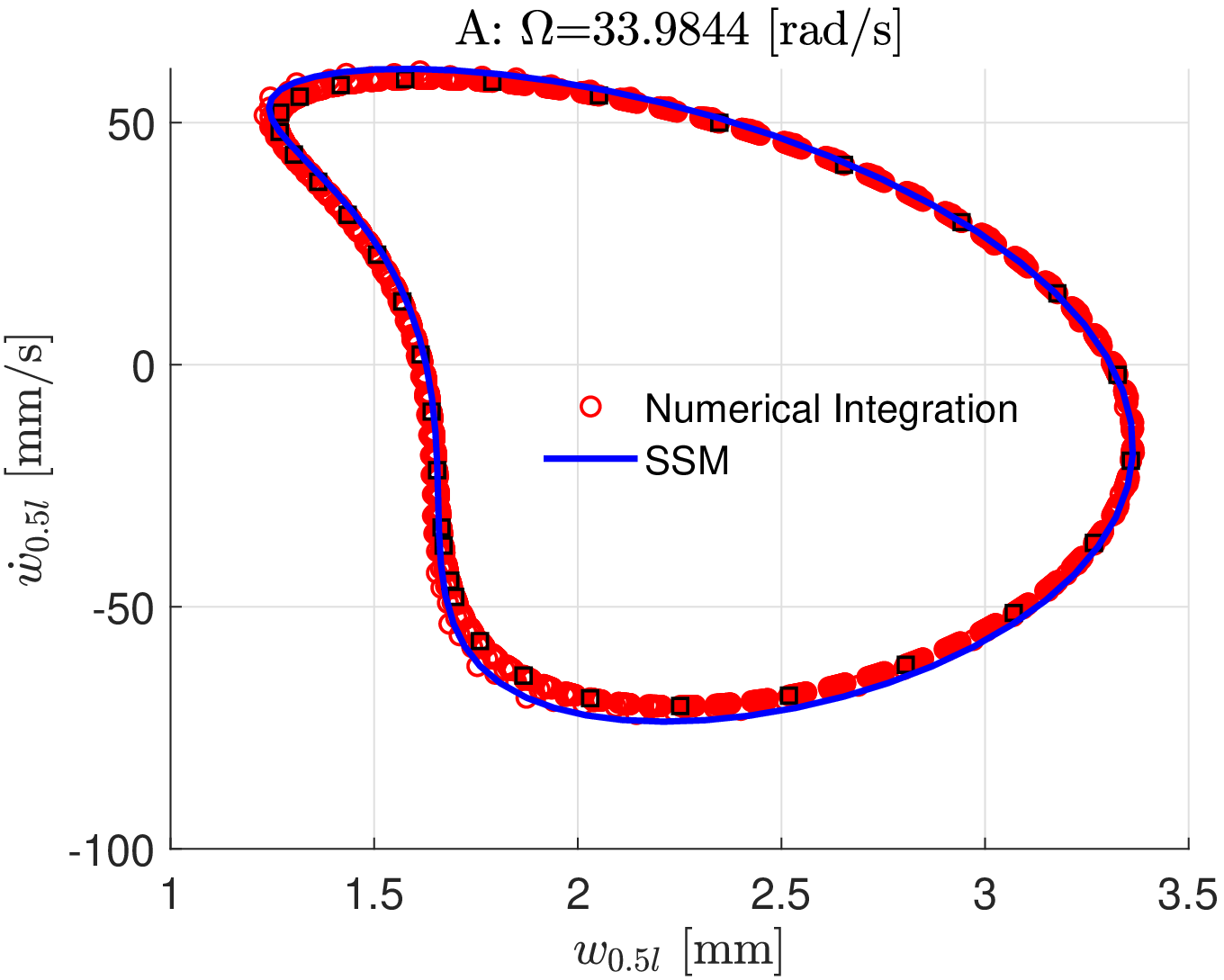}
\includegraphics[width=0.38\textwidth]{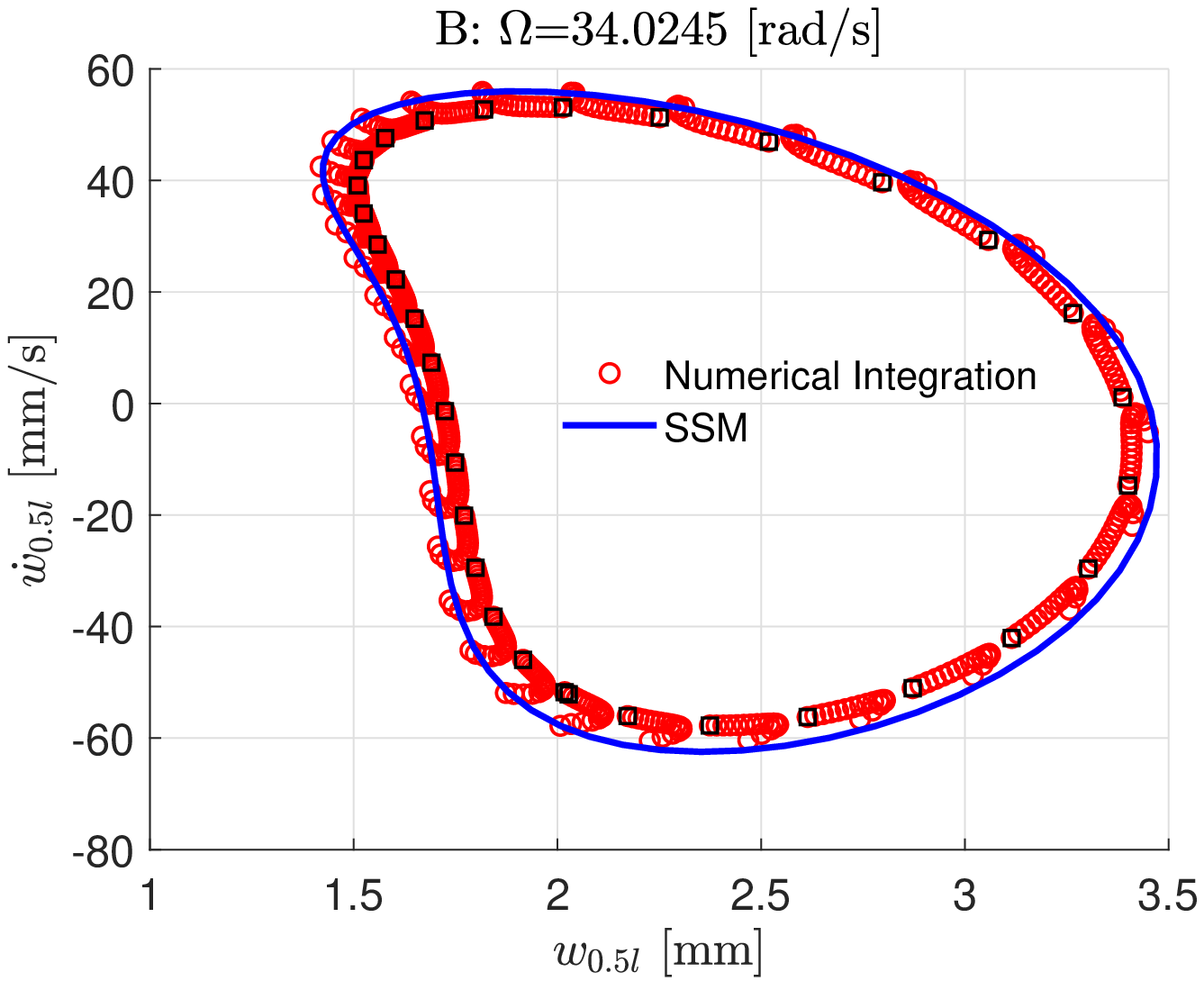}
\caption{Intersections of the period-$2\pi/\Omega$ map of sampled tori of clampled-pinned von K\'arm\'an beam discretized with 8 elements, approximated by the limit cycles from the SSM analysis (solid lines) and a collection of points in the numerical integration (circles). Here the black squares are a subset of red circles in the final stage of the integration. They are used to indicate the evolution of the intersections.}
\label{fig:vonKarmanABC}
\end{figure}

The computational times of FRCs for periodic orbits and quasi-periodic orbits with various numbers of degrees of freedom (DOFs) are summarized in Fig.~\ref{fig:vonKarmanBeam_compTime}. For the SSM reduction, it only took about one hour to obtain the FRC for both periodic orbits and quasi-periodic orbits in the case of 10,000 elements with 29,998 DOFs. When the number of DOF is less than or equal to 598, the computational time of invariant tori is larger than that of the periodic orbits, which is reasonable because the continuation of periodic orbits in the reduced-order model needs more time compared to the continuation of fixed points. Note that the autonomous SSM obtained in the computation of FRC of periodic orbits is used directly in the computation of FRC of quasi-periodic orbits. For this reason, the computational time for the autonomous SSM is not accounted for in the runtime of FRC for quasi-periodic orbits. When the number of DOF is higher, this part of the computational time becomes significant. Indeed, as can be seen in Fig.~\ref{fig:vonKarmanBeam_compTime}, the computational time for FRC of periodic orbits exceeds that of quasi-periodic orbits for $N_\mathrm{e}\geq3,000$.

\begin{figure}[!ht]
\centering
\includegraphics[width=0.48\textwidth]{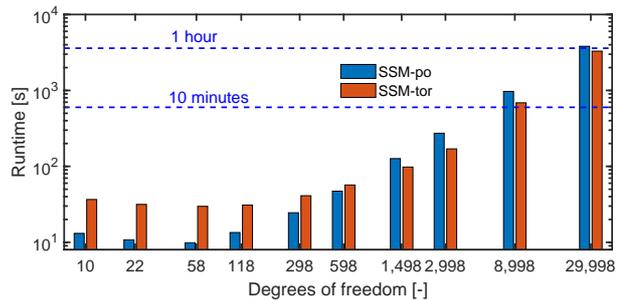}
\caption{Computational times of FRCs of the clamped-pinned von K\'arm\'an beam discretized with different number of DOFs. The number of DOFs is given by $3N_{\mathrm{e}}-2$ when the beam is discretized with $N_{\mathrm{e}}$ elements. Here we have $N_{\mathrm{e}}\in\{$4, 8, 20, 40, 100, 200, 500, 1,000, 3,000, 10,000$\}$. The upper bound of $N_{\mathrm{e}}$ is set to be 10,000 to avoid the accumulation of truncation errors induced by over-refined meshes. Here and in the next figure, the labels `SSM-po' and `SSM-tor' denote the computational times of FRCs of periodic and quasi-periodic orbits, respectively.}
\label{fig:vonKarmanBeam_compTime}
\end{figure}

We conclude this example by showing how the computational costs evolve when the expansion order of SSM increases. We fix the number of elements $N_\mathrm{e}=10,000$ and conduct the computations of FRCs for periodic orbits and quasi-periodic orbits with various expansion orders of the SSM. As seen in Fig.~\ref{fig:vonKarmanBeam_compTime_orders}, the computational time for FRC of periodic orbits increases significantly with the increment of the orders, while the computational time for FRC of quasi-periodic orbits increases slowly with increasing expansion orders.

\begin{figure}[!ht]
\centering
\includegraphics[width=0.4\textwidth]{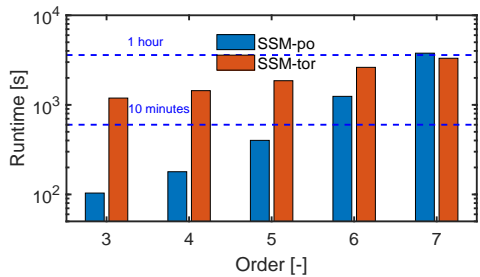}
\caption{Computational times of FRCs of the clamped-pinned von K\'arm\'an beam discretized with 29,998 DOFs as functions of the expansion order of the SSM.}
\label{fig:vonKarmanBeam_compTime_orders}
\end{figure}

As discussed in section~\ref{sec:compLoad}, the computational time for FRC of periodic orbits can be decomposed into the following four parts: autonomous SSM, reduced dynamics, nonautonomous SSM and the evaluation of map $\boldsymbol{W}(\boldsymbol{p})$. The computational time for the autonomous SSM is dominated among these four parts when the expansion order is high. Specifically, the proportions of the computational times used to compute the autonomous SSM at expansion orders 3, 5 and 7 are 14\%, 73\% and 95\%, respectively.

The computational time for FRC of quasi-periodic orbits here can be decomposed into three parts: reduced dynamics, nonautonomous SSM and evaluation of map $\boldsymbol{W}(\boldsymbol{p})$, because the autonomous SSM has been obtained from the computation of FRC of periodic orbits. We find that more than 97\% of the computational time is used to evaluate the map $\boldsymbol{W}(\boldsymbol{p})$ for tori, independently of the expansion orders. This can be explained by the fact that the number of evaluations of the map $\boldsymbol{W}$ is huge. Specifically, the map needs to be evaluated $n_\mathrm{tor}\cdot n_\mathrm{traj}\cdot n_\mathrm{pt}$ times, where $n_\mathrm{tor}$ is the number of computed tori, $n_\mathrm{traj}$ is the averaged number of trajectories for approximating a torus, and $n_\mathrm{pt}$ is the number of points for each of the trajectories. We have $n_\mathrm{tor}\approx50$, $n_\mathrm{traj}\approx60$ and $n_\mathrm{pt}=128$, resulting in approximately 400,000 evaluations. As discussed in section~\ref{sec:compLoad}, these evaluations can be parallelized to speed up the computation.

\subsection{A forced simply supported von K\'arm\'an plate}
\label{sec:vonKarmanPlate}
In our fourth example, we revisit the von K\'arm\'an square plate example studied in Part I. Due to the symmetry of the plate, the natural frequencies of the second and third bending modes are equal, leading to a 1:1 internal resonance. In Part I, we computed the FRC for the periodic orbits of the plate. Here we are concerned with the quasi-periodic responses of the plate.

\begin{sloppypar}
The geometry of the plate is defined as $(x,y,z)\in[0,l]\times[0,l]\times[-0.5h,0.5h]$, where $l$ and $h$ are the length (width) and thickness of the plate, respectively. The plate is discretized with triangular elements. Provided that the length of the plate is uniformly divided into $n_{\mathrm{p}}$ subintervals, the number of elements of the discretized plate is $N_{\mathrm{e}}=2n_{\mathrm{p}}^2$. The reader may refer to Part I for more details about the mesh. 

If $u,v,w$ denote the displacements along the $x,y,z$ directions respectively, then six DOFs need to be introduced at each node: $(u,v,w,w_x,w_y,u_y-v_x)$. Therefore, the number of DOFs of the discrete plate with boundary conditions accounted for is given by $n=6(n_{\mathrm{p}}^2+1)$.  The equation of motion of the finite element model has the same form as~\eqref{eq:eom-beams}. The reader may refer to~\cite{allman1976simple,allman1996implementation,FEcode} for the mass and stiffness matrices and the coefficients of the nonlinear terms. We again use Rayleigh damping $\boldsymbol{C}=\alpha\boldsymbol{M}+\beta\boldsymbol{K}$. The plate is excited by a harmonic force $\epsilon F\cos\Omega t$ applied at a point A with coordinates $(x,y,z)=(0.2l,0.3l,0.5h)$. In the following computations, we use the same geometric and material parameters as in Part I. Specifically, we have $l=1\,\textrm{m}$, $h=0.01\,\textrm{m}$, Young's modulus $E=70\times 10^9\,\mathrm{Pa}$, Poisson's ratio ratio $\nu=0.33$ and density $\rho=2700\,\mathrm{kg}/\textrm{m}^3$. The two damping coefficients in the Rayleigh damping are chosen as $\alpha=1$ and $\beta=1\times10^{-6}$ such that the system is weakly damped (see~\eqref{eq:weakDampFreq}).

In Part I, the FRCs of periodic orbits for $n_\mathrm{p}=10,20,40,100,200$ have been computed. In the case of $n_\mathrm{p}=200$, the number of DOFs is 240,006, and the computational time for the FRC is less than one day, which highlights the remarkable computational efficiency of the SSM reduction. In the FRC with $n_\mathrm{p}=10$, two SN fixed points are found but no HB fixed points are detected. However, a pair of HB fixed points is observed at $n_\mathrm{p}=20$ which persist for meshes with higher fidelity. With $n_\mathrm{p}=20$, the number of DOF is 2,406. For this case, the FRCs over $\Omega\in[0.95, 1.1]\omega_2$ for the transverse deflections at the point A and the point B with coordinates $(x,y,z)=(0.7l,0.3l,0.5h)$ are presented in Fig.~\ref{fig:FRC_vonKarmanPlate_AB}. Here $\omega_2$ is the natural frequency of the second bending mode; around $\omega_2\approx767.2$ the amplitude of the linear response reaches its a peak. More details about the linear analysis can be found in Part I. Due to the geometric nonlinearity and the internal resonance, the FRC computed using the SSM reduction is very different from that obtained from linear analysis for large response amplitudes. In particular, two HB points, denoted HB1 and HB2, are detected in the continuation of equilibria in the leading-order reduced dynamics~\eqref{eq:ode-reduced-slow-polar-leading}, as seen in Fig.~\ref{fig:FRC_vonKarmanPlate_AB}. They correspond to torus bifurcation periodic orbits of the full system. In contrast, no bifurcation arises in the linear response.
\end{sloppypar}

\begin{figure}[!ht]
\centering
\includegraphics[width=0.45\textwidth]{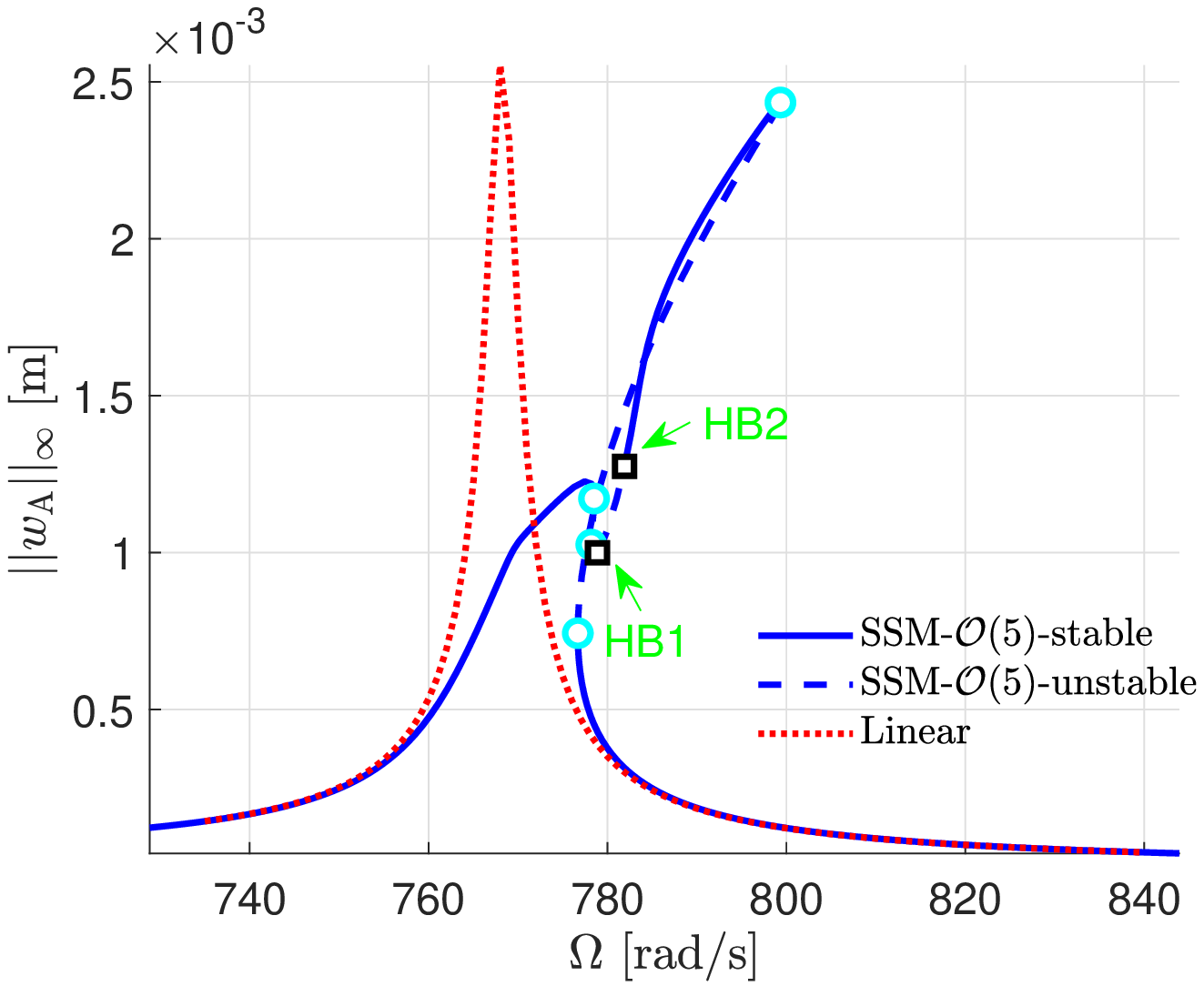}
\includegraphics[width=0.45\textwidth]{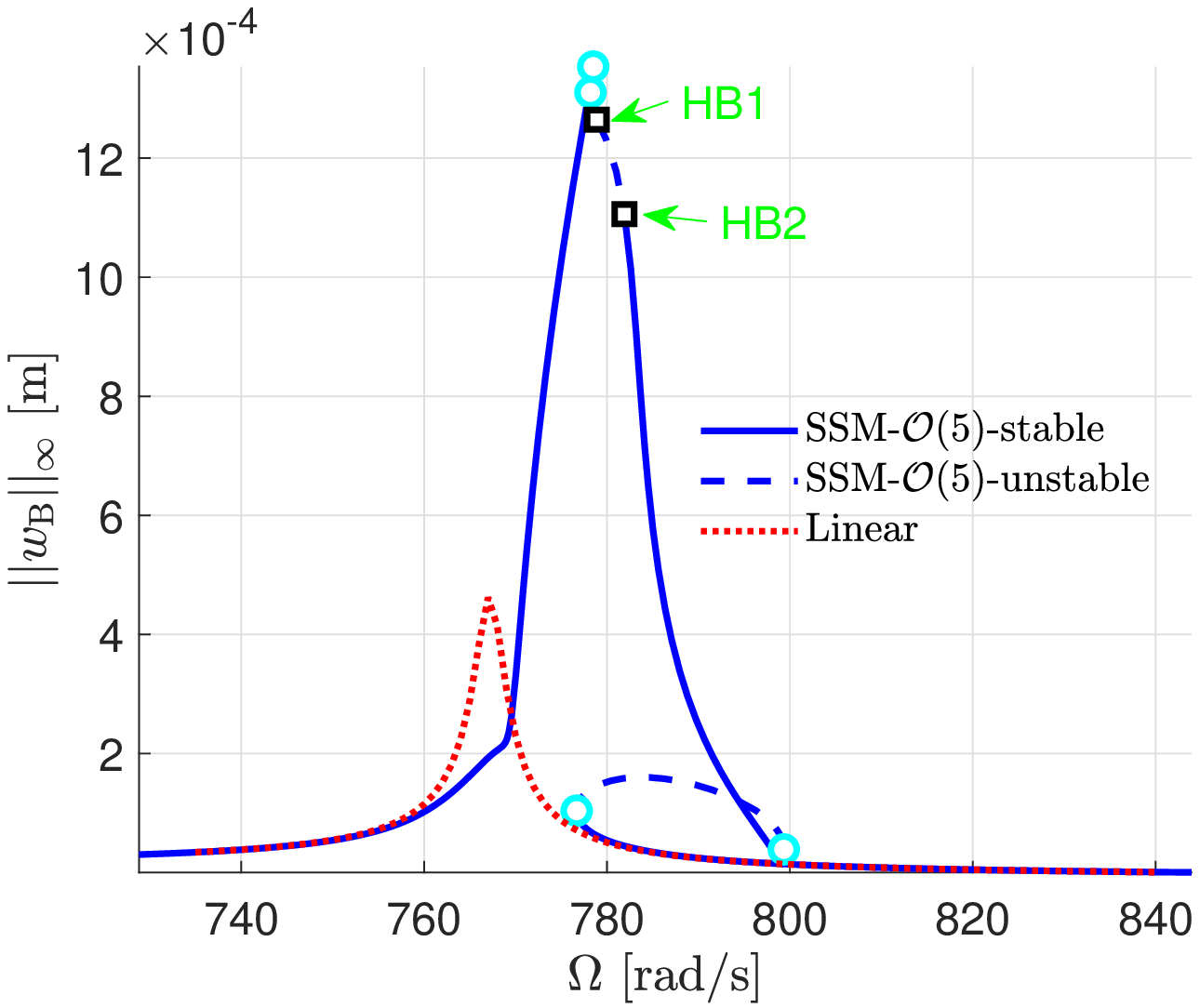}
\caption{FRCs in physical coordinates (the amplitudes of the periodic responses for the transverse displacement $w$ at the points A and B) for the von K\'arm\'an plate discretized with 2,406 DOF. The FRCs are already converged at $\mathcal{O}(5)$ expansion for the SSM. Here the red dotted lines are the results of linear analysis.}
\label{fig:FRC_vonKarmanPlate_AB}
\end{figure}

Next, we switch from the continuation of fixed points in the SSM-reduced model to the continuation of limit cycles at the HB2 point. This continuation proceeds until it reaches the HB1 point. In this run, we have found both stable and unstable periodic orbits. As seen in Fig.~\ref{fig:vonKarmanPlate_TandRho}, the limit cycles bifurcating from the two HB points are stable initially but become unstable via period-doubling bifurcations. In addition, the upper panel shows that $T_\mathrm{s}\approx2$ and then $\pi\approx\omega_s\ll\Omega\approx800$, highlighting the significant difference between these two time scales.

\begin{figure}[!ht]
\centering
\includegraphics[width=0.45\textwidth]{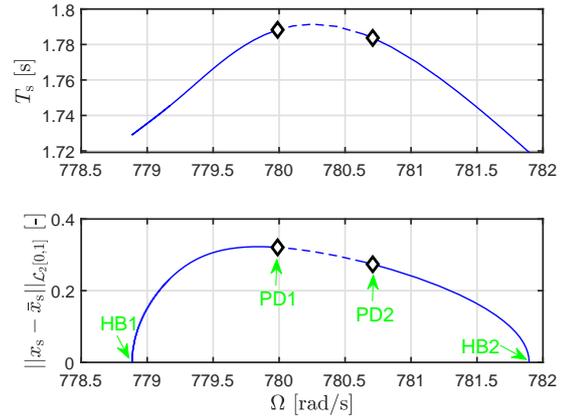}
\caption{Projections of the continuation path of the limit cycles in the SSM-reduced model of the 1:1 resonant discrete von K\'am\'an plate model. The upper and lower panels present the period and the size of the limit cycles as functions of $\Omega$, respectively. The diamonds correspond to period-doubling (PD) bifurcations of periodic orbits. The limit cycles here are mapped to the two-dimensional invariant tori of the full system with two frequency componnets, $\Omega$ and $\omega_\mathrm{s}=2\pi/T_\mathrm{s}$.}
\label{fig:vonKarmanPlate_TandRho}
\end{figure}

\begin{sloppypar}
The limit cycles shown in Fig.~\ref{fig:vonKarmanPlate_TandRho} signal the two-dimensional invariant tori of the full system. The FRCs of quasi-periodic orbits contained in these tori are presented in Fig.~\ref{fig:vonKarmanPlate_tori}. The solution branch of periodic orbits and the branch of quasi-periodic orbits intersect at the two torus bifurcation periodic orbits, which are detected as HB equilibria in the SSM-reduced model. Recall that the stable tori become unstable through period-doubling bifurcations (see~Fig.~\ref{fig:vonKarmanPlate_TandRho}). At such a PD bifurcation, a unique quasi-periodic orbit whose internal frequency is half of the internal frequency of the bifurcating torus is generated. Further, a secondary branch of quasi-periodic orbits appears under variation of $\Omega$. The quasi-periodic orbits on this secondary branch may be stable and a perturbed unstable torus may converge to the new family of tori. We leave a detailed study of this phenomenon to future research.
\end{sloppypar}

\begin{figure}[!ht]
\centering
\includegraphics[width=0.45\textwidth]{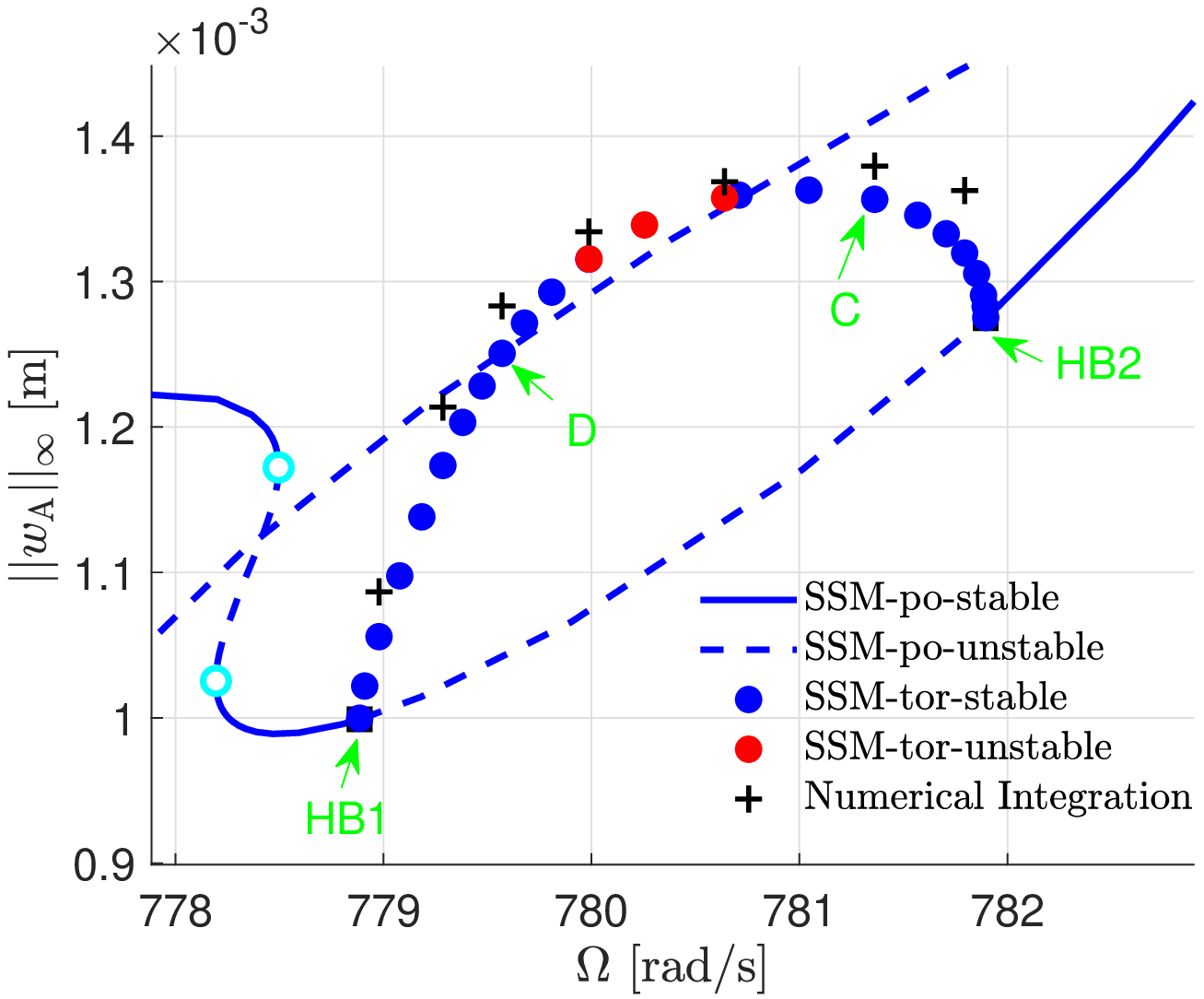}
\includegraphics[width=0.45\textwidth]{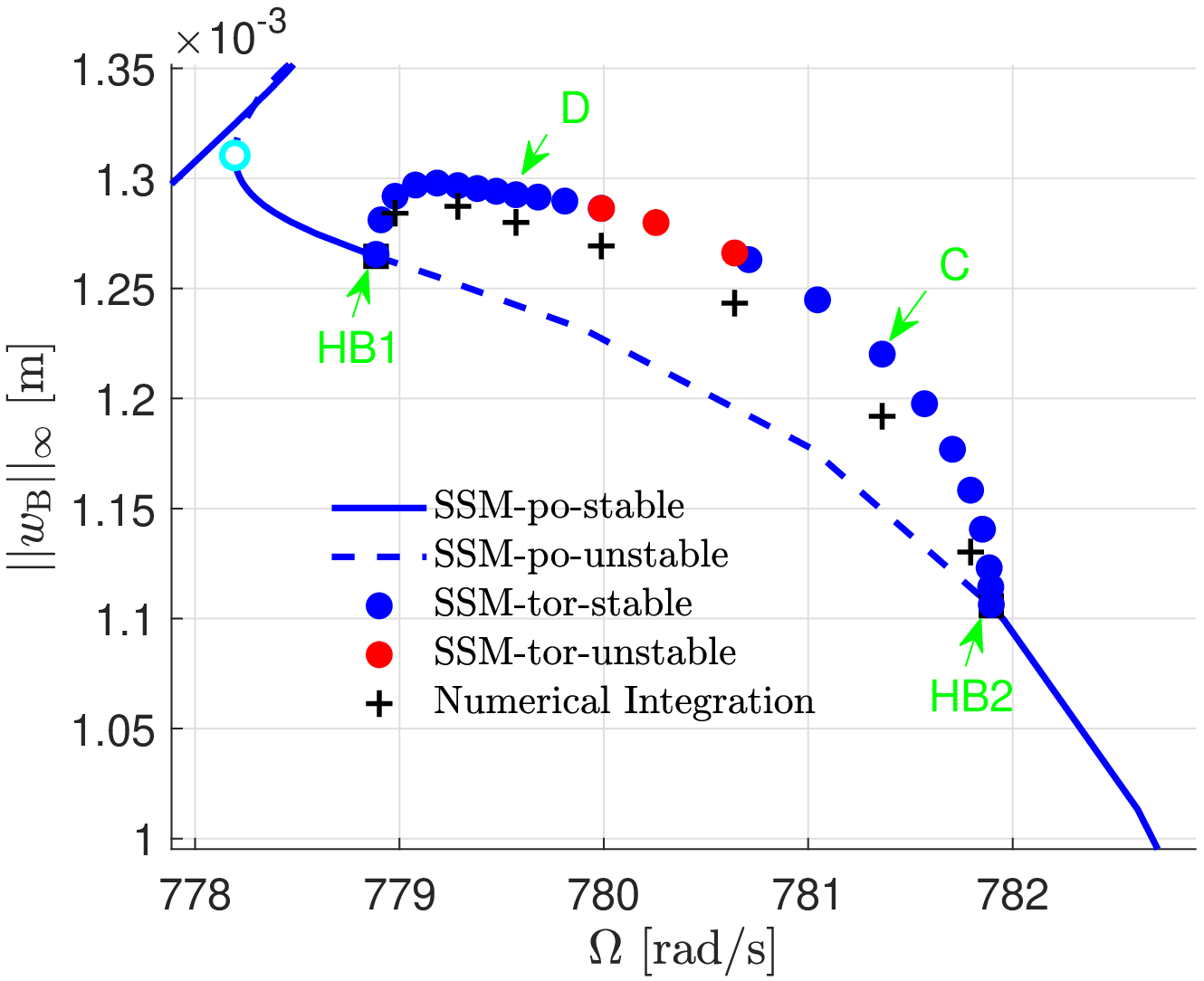}
\caption{FRCs in physical coordinates (the amplitudes of the periodic and quasi-periodic responses for the transverse displacement $w$ at the points A and B) for the von K\'arm\'an plate discretized with 2,406 DOFs.}
\label{fig:vonKarmanPlate_tori}
\end{figure}

\begin{sloppypar}
To validate the invariant tori obtained from SSM reduction, we consider the two methods used in our previous two examples: a convergence study of the SSM and a direct numerical integration of the full systems. For the convergence study, we focus on the two quasi-periodic PD bifurcation tori shown in~Fig.~\ref{fig:vonKarmanPlate_TandRho}. As seen in Fig.~\ref{fig:vonKarmanPlate_convg}, the two frequencies of each torus have converged when the expansion the SSM is $\mathcal{O}(7)$ or higher. In addition, the results at $\mathcal{O}(5)$ are very close to the convergent results while the results at $\mathcal{O}(3)$ and $\mathcal{O}(4)$ are not yet accurate. 

We also performed numerical integration of the full system, just as we did in the previous example. As seen in Fig.~\ref{fig:vonKarmanPlate_tori}, the results from SSM reduction match well with numerical integration. Detailed plots of the tori C and D are shown in Fig.~\ref{fig:vonKarmanPlate_tori_CD}, where we can see how the trajectories of the numerical integration stay close to the tori predicted from SSM reduction. Here we did not mark the complete trajectories with lines because they are too dense due to the long-time integration. Instead, we plot dots along the trajectories with 50 dots per cycle.
\end{sloppypar}

\begin{figure}[!ht]
\centering
\includegraphics[width=0.40\textwidth]{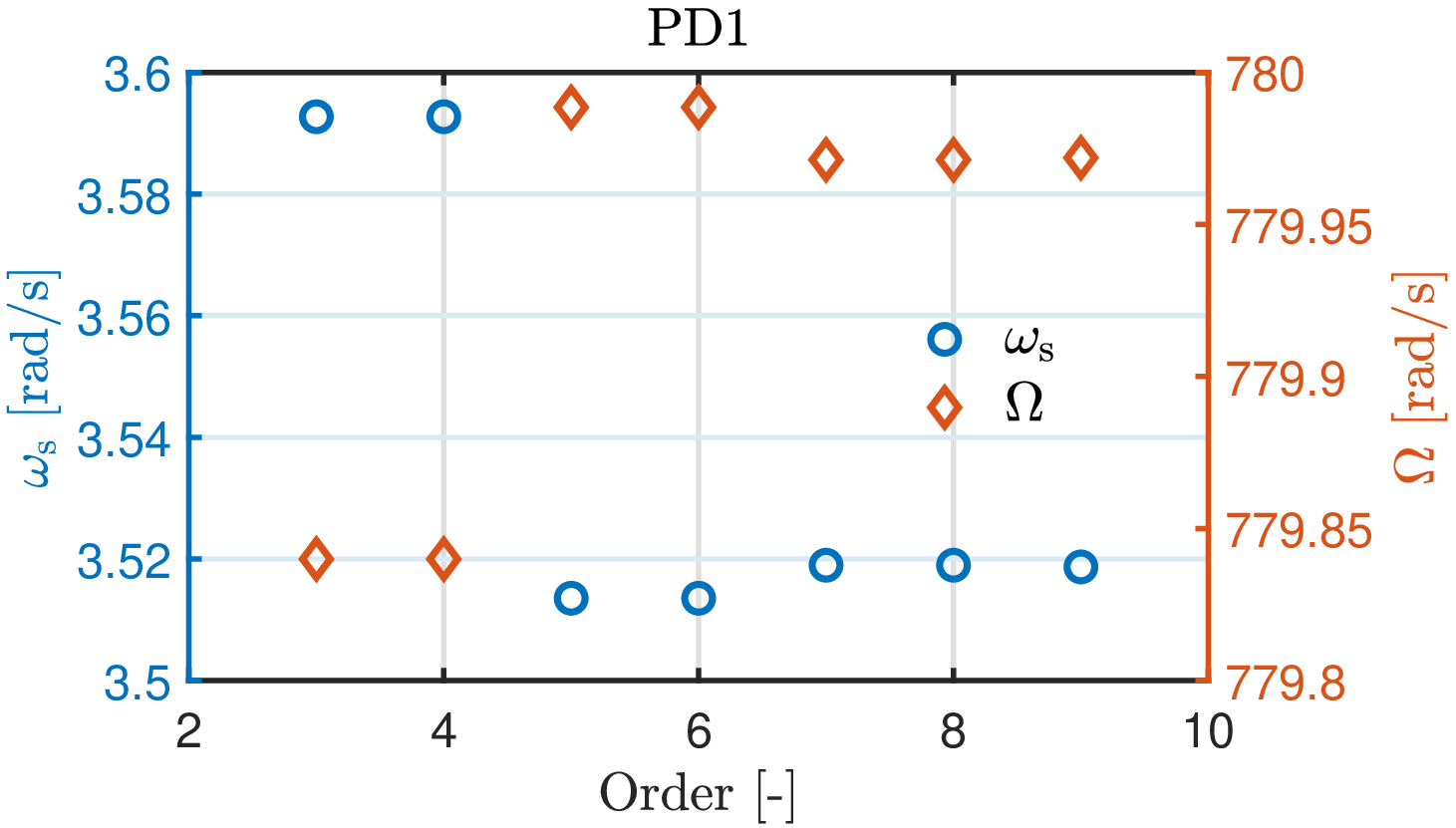}
\includegraphics[width=0.40\textwidth]{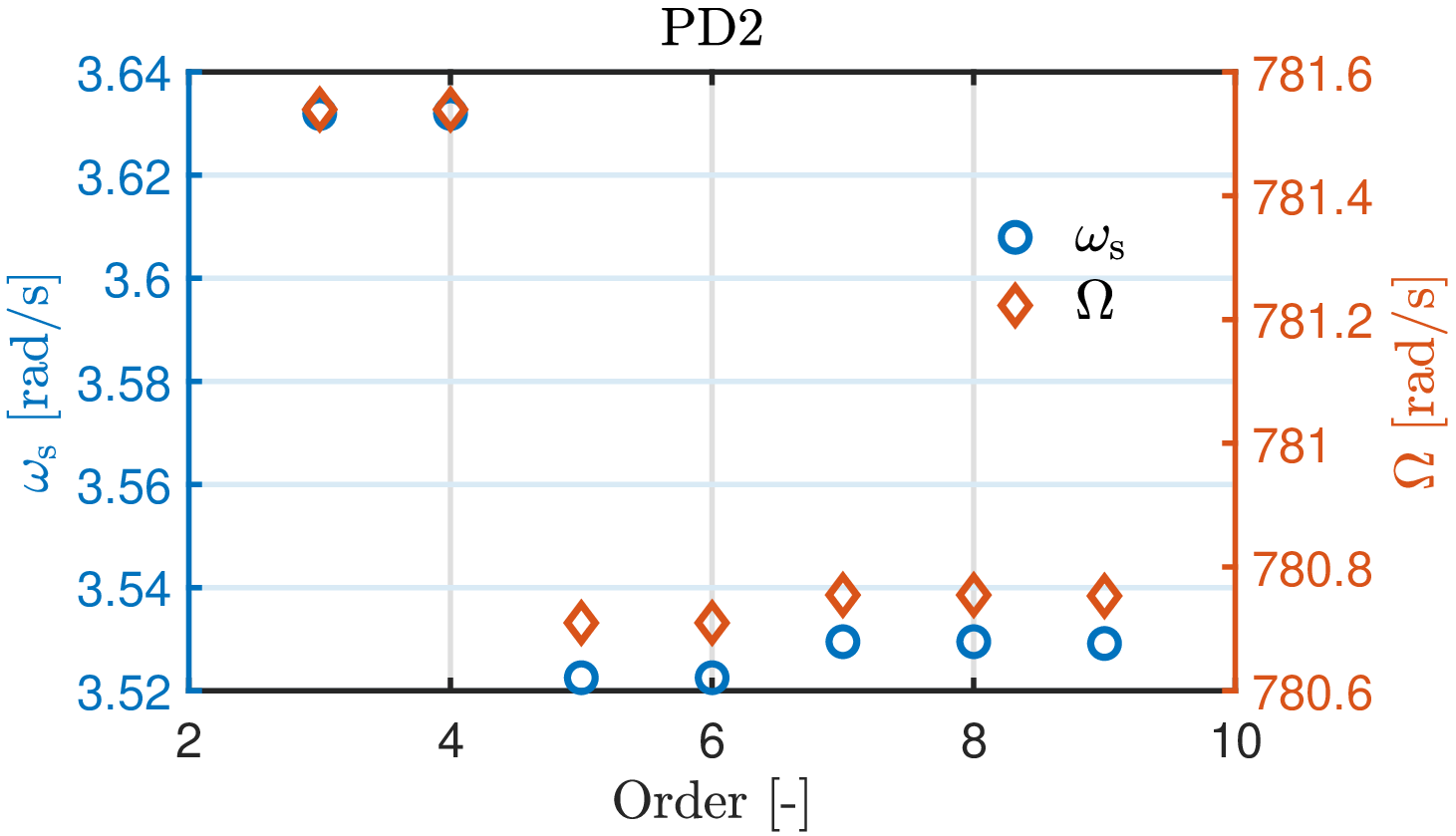}
\caption{Convergence of the two frequencies of the two bifurcating invariant tori shown in Fig.~\ref{fig:vonKarmanPlate_TandRho} with increasing expansion order for the SSM.}
\label{fig:vonKarmanPlate_convg}
\end{figure}

\begin{figure}[!ht]
\centering
\includegraphics[width=0.45\textwidth]{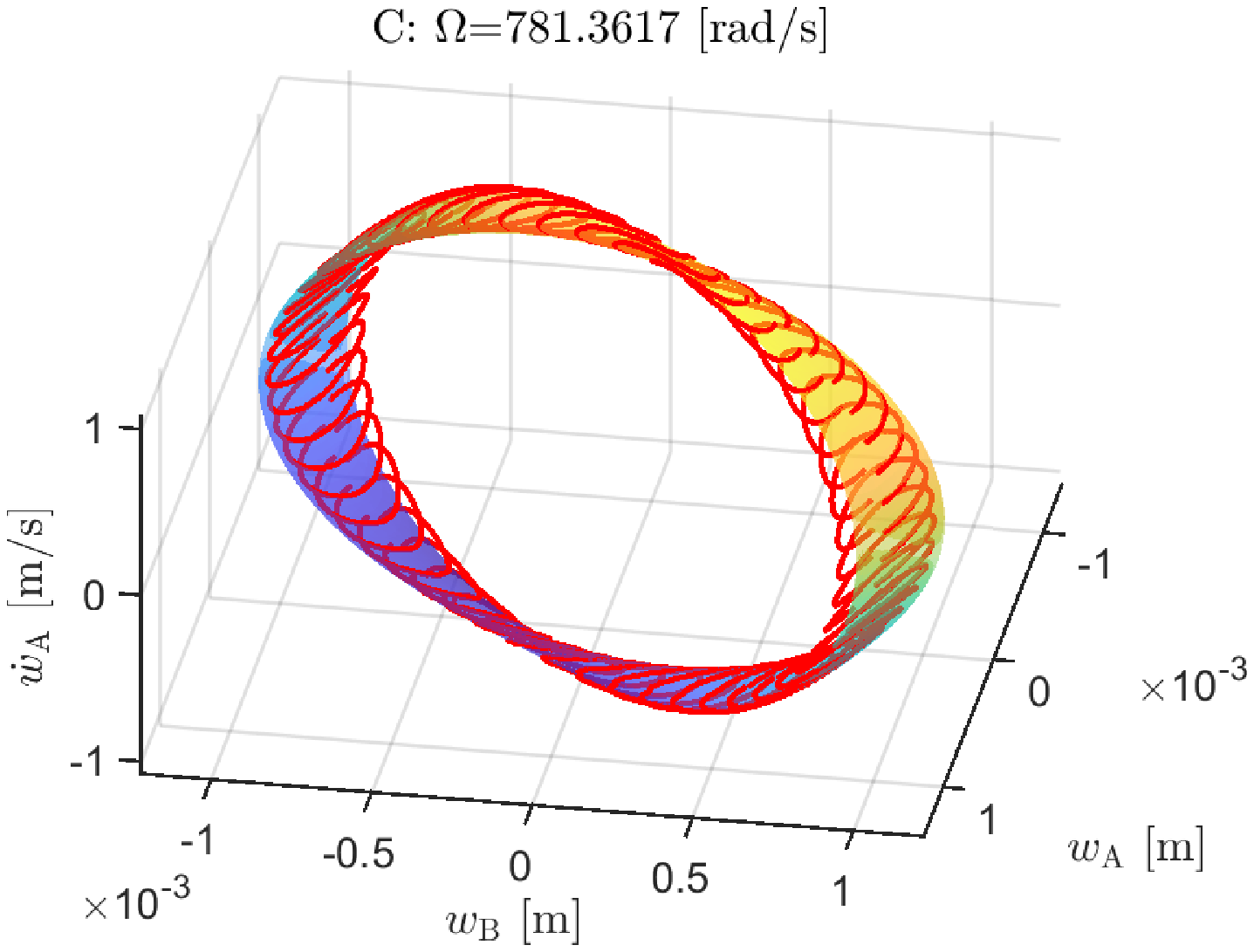}
\includegraphics[width=0.45\textwidth]{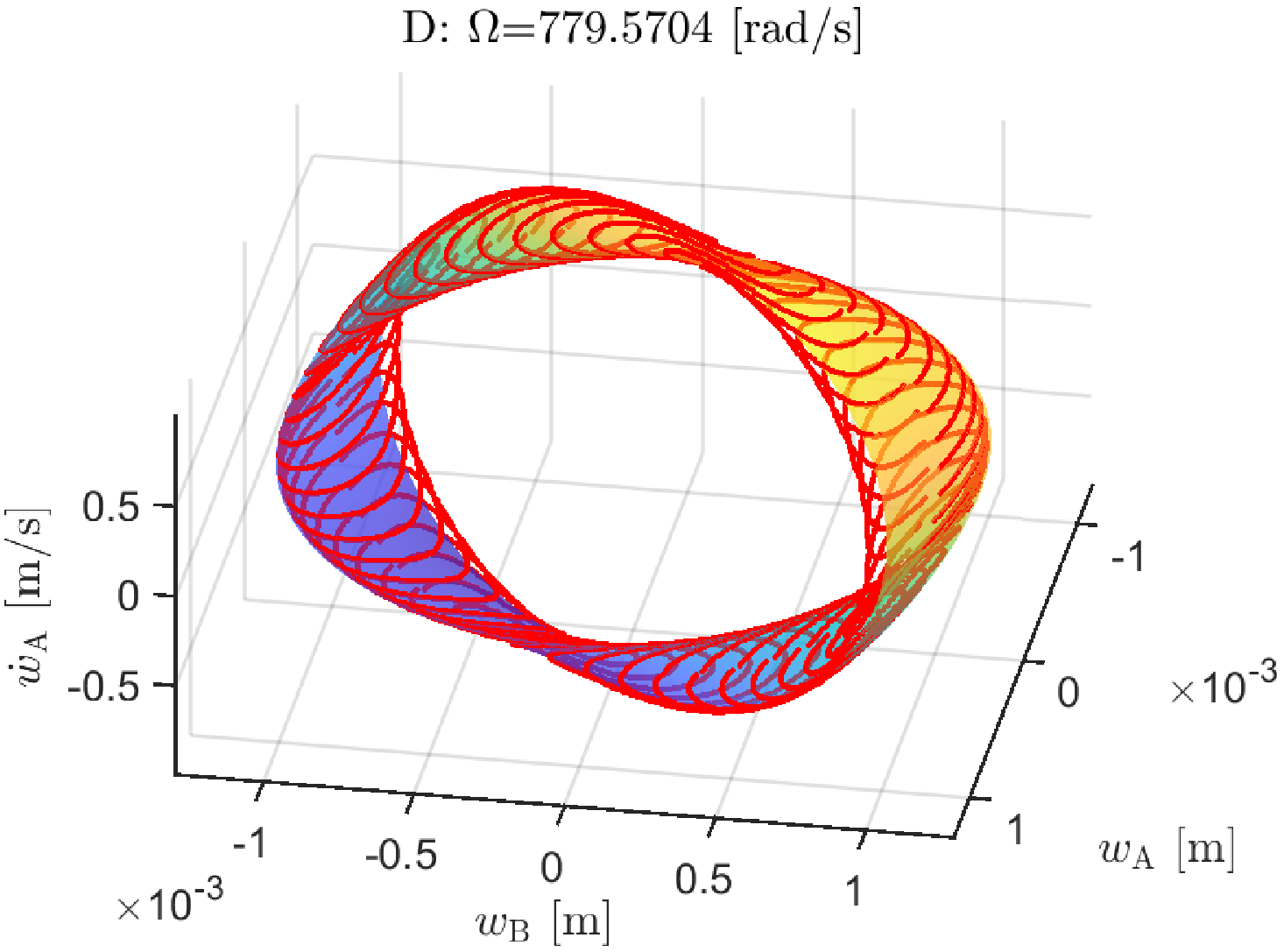}
\caption{Invariant tori of the discretized square von K\'am\'an plate obtained from SSM reduction (surface plot) and the trajectories obtained by the numerical integration of the full system with initial states on the tori (red dots). Upper panel: torus C. Lower panel: torus D; see Fig.~\ref{fig:vonKarmanPlate_tori}.}
\label{fig:vonKarmanPlate_tori_CD}
\end{figure}

\section{Conclusion}
\label{sec:conclusion}
We have shown how various bifurcations of periodic and quasi-periodic response in very high-dimensional forced-damped mechanical systems with an internal resonance can be efficiently and accurately predicted from reduced-order models (ROMs) obtained on spectral submanifolds (SSMs) of these systems. A symmetry in the leading-order dynamics on resonant SSMs enables the identification of invariant 2-tori and 3-tori of the full system as limit cycle and invariant 2-tori of the truncated ROM on the SSM. This enables the detection of limit cycle and torus bifurcations in complex mechanical systems in previously unreachable dimensions and previously unattainable speeds.

To automate the detection of limit cycle and torus bifurcations in resonant SSMs, we have developed the open-source \textsc{matlab} toolboxes, which are available as parts of SSMTool-2.2~\cite{ssmtool2}. The \texttt{SSM-ep} and \texttt{SSM-po} are based on the \texttt{ep} and \texttt{po} toolboxes of \textsc{coco}, respectively, which support stability and bifurcation analysis. The \texttt{SSM-tor} toolbox is based on a recently developed \textsc{coco} toolbox, \texttt{Tor}, which does not yet provide stability and bifurcation analysis of two-dimensional invariant tori. Further enhancements of this toolbox is therefore highly desirable.

We have considered four mechanical examples to illustrate our results. In the first example, the quasi-periodic responses of two coupled oscillators with 1:2 internal resonance were calculated using both the SSM and the semi-discretization method. The results obtained by the two methods match closely. In the second example, both two- and three- dimensional invariant tori of a discrete Bernoulli beam with nonlinear support spring were investigated in detail. The convergence of the results under increasing expansion order for the SSM has been verified. In the last two examples, we considered finite element models of von K\'arm\'an beam and plate structures with distributed nonlinearity. Numerical integration has been performed to validate the results obtained from SSM reduction.

\begin{sloppypar}
Both in Part I and in the present paper, we have assumed harmonic external forcing in our setup. A mechanical system, however, also displays quasi-periodic response under quasi-periodic external forcing. The existence of SSMs for such systems has also been discussed in details in~\cite{haller2016nonlinear}. Based on these results, one can derive reduced-order models for such systems using SSM theory, as already indicated by~\cite{SHOBHIT}. As an extension, one may, therefore, generalize present study to systems with internal resonance and quasi-periodic external forcing.
\end{sloppypar}

\section{Appendix}

\subsection{Proof of Theorem~\ref{theo-po}}
\label{sec:appendix-theo1}
\begin{sloppypar}
The statements (i-iii) are simply restatements of Theorems I and II of Part I. For this reason, we only discuss the proofs of statements (iv-v) here. Following the proofs in Part I, we first derive a slow-fast dynamical system for the SSM dynamics. We then use the method of averaging~\cite{guckenheimer2013nonlinear} to prove the two statements on local bifurcations.

Similarly to Part I, we focus on the reduced-order model in Cartesian coordinates (cf.~\eqref{eq:ode-reduced-slow-cartesian} and~\eqref{eq:ode-reduced-slow-cartesian-leading}) to complete the proof of (iv) and (v) because these two statements are not affected by the choice of coordinates.
Let $\mathbf{x}=(q_{1,\mathrm{s}}^{\mathrm{R}},q_{1,\mathrm{s}}^{\mathrm{I}},\cdots,q_{m,\mathrm{s}}^{\mathrm{R}},q_{m,\mathrm{s}}^{\mathrm{I}})$. Equation~\eqref{eq:ode-reduced-slow-cartesian} can then be rewritten as
\begin{equation}
\label{eq:ode-qc}
    \dot{\mathbf{x}}=\mathbf{A}\mathbf{x}+\mathbf{F}(\mathbf{x})+\epsilon \mathbf{F}^\mathrm{ext}+\mathcal{O}(\epsilon|\mathbf{x}|)\mathbf{G}(\phi),\quad \dot{\phi}=\Omega,
\end{equation}
where $\mathbf{A}$ is an invertible, block-diagonal matrix, $\mathbf{F}$ contains the nonlinear terms in the reduced-order model~\eqref{eq:ode-reduced-slow-cartesian-leading}, $\mathbf{F}^\mathrm{ext}$ is a constant vector and $\mathbf{G}(\phi)$ is a periodic function. Detailed expressions for these terms can be found in Appendix 8.3 of Part I. Introducing the transformation
\begin{equation}
\label{eq:xtoxhat}
    \mathbf{x}=\mu\hat{\mathbf{x}},\quad \mu=\epsilon^{1-q},\quad q=1-1/k,
\end{equation}
where $k$ is the lowest order of nonlinearity in $\mathbf{F}(\mathbf{x})$, we obtain the governing equation for $\hat{\mathbf{x}}$ from~\eqref{eq:ode-qc} in the form
\begin{equation}
    \dot{\hat{\mathbf{x}}} = \epsilon^q\mathcal{F}(\hat{\mathbf{x}},\mu,\Omega)+\mathcal{O}(\epsilon |\hat{\mathbf{x}}|)\mathbf{G}(\phi),\quad \dot{\phi}=\Omega.
\end{equation}
Here $\mathcal{F}$ can be found in eq. (117) of Part I (we include $\Omega$ explicitly as an argument here for bifurcation analysis). For $\nu=\epsilon^q$, the equations above can be rewritten as
\begin{equation}
\label{eq:full-flow}
    \dot{\hat{\mathbf{x}}} = \nu\mathcal{F}(\hat{\mathbf{x}},\mu,\Omega)+\nu\mu\mathcal{O}( |\hat{\mathbf{x}}|)\mathbf{G}(\phi),\quad \dot{\phi}=\Omega.
\end{equation}

To apply the method of averaging~\cite{guckenheimer2013nonlinear}, we also consider two other flows as follows
\begin{gather}
    \dot{\hat{\mathbf{x}}} = \nu \mathcal{F}(\hat{\mathbf{x}},\mu,\Omega),\quad\dot{\phi}=\Omega,\label{eq:leading-flow}\\
    \dot{\hat{\mathbf{x}}} = \nu \mathcal{F}(\hat{\mathbf{x}},0,\Omega),\quad \dot{\phi}=\Omega,\label{eq:zero-flow}
\end{gather}
where~\eqref{eq:leading-flow} is the truncation of~\eqref{eq:full-flow}, and both the flow~\eqref{eq:full-flow} and its truncation~\eqref{eq:leading-flow} are reduced to~\eqref{eq:zero-flow} in the limit of $\mu=0$.
Let $T=2\pi/(r_\mathrm{d}\Omega)$, where $r_\mathrm{d}$ is defined as the largest common divisor for the set of rational numbers $\{r_i\}_{i=1}^m$ (see~\eqref{eq:res-forcing}). We define the period-$T$ maps of the three flows~\eqref{eq:full-flow},~\eqref{eq:leading-flow} and~\eqref{eq:zero-flow} as $\mathbf{P}_1$, $\mathbf{P}_2$ and $\mathbf{P}_0$, respectively. Furthermore, we define $\mathbf{H}_i$ as the scaled zero functions associated with the fixed points of the Poincar\'e maps $\mathbf{P}_{i}$:
\begin{equation}
    \mathbf{H}_i(\hat{\mathbf{x}},\Omega)=\frac{1}{\nu}(\mathbf{P}_{i}\hat{\mathbf{x}}-\hat{\mathbf{x}}),\quad i=0,1,2.
\end{equation}
Suppose that the Poincar\'e map $\mathbf{P}_2$ of the truncated flow, viz, $\mathbf{P}_2$ has a local bifurcation. Our goal is to show that this bifurcation is persistent in the map $\mathbf{P}_1$. We will follow the line of thought $\mathbf{P}_2/\mathbf{H}_2\to\mathbf{P}_0/\mathbf{H}_0\to\mathbf{P}_1/\mathbf{H}_1$ to reach this goal.

The proof of statement (iv) here is adapted from~\cite{guckenheimer2013nonlinear}. Assume that the truncated reduced-order model~\eqref{eq:ode-reduced-slow-cartesian-leading} undergoes a saddle-node bifurcation at $\Omega=\Omega_0$. Then the averaged system~\eqref{eq:leading-flow} has a pair of fixed points $\hat{\mathbf{x}}_{+}$, $\hat{\mathbf{x}}_{-}$ which coalesce at $\Omega_0$ in a smooth arc in the $(\hat{\mathbf{x}},\Omega)$ space. This pair of fixed points are the zeros of $\mathbf{H}_2$. There exists a local change of coordinates near the points $(\hat{\mathbf{x}}_{\pm}(\Omega_0),\Omega_0)\in\mathbb{R}^{2m}\times\mathbb{R}$ under which points in these branches can be put into the form
$\hat{\mathbf{x}}_{\pm}(\Omega)=(\pm c\sqrt{\Omega_0-\Omega},\mathbf{0})\in\mathbb{R}\times\mathbb{R}^{2m-1}$ (see~\cite{guckenheimer2013nonlinear}; we have assumed that the fixed points disappear when $\Omega>\Omega_0$ without loss of generality). Since $\mathbf{H}_0$ is $\mu$-close to $\mathbf{H}_2$, $\mathbf{H}_0$ has a pair of branches of zeros $\hat{\mathbf{y}}_{\pm}(\Omega)=(\pm c\sqrt{\Omega_0-\Omega}+\mathcal{O}(\mu),\mathcal{O}(\mu))$ in the same coordinates. Likewise, since $\mathbf{H}_1$ is $\mu$-close to $\mathbf{H}_0$, $\mathbf{H}_1$ has a pair of branches of zeros $\hat{\mathbf{z}}_{\pm}(\Omega)=(\pm c\sqrt{\Omega_0-\Omega}+\mathcal{O}(\mu),\mathcal{O}(\mu))$. Since the Jacobian matrices $\partial_{\hat{\mathbf{x}}}\mathbf{P}_2(\hat{\mathbf{x}}_{\pm}(\Omega))$, $\partial_{\hat{\mathbf{x}}}\mathbf{P}_0(\hat{\mathbf{y}}_{\pm}(\Omega))$ and $\partial_{\hat{\mathbf{x}}}\mathbf{P}_1(\hat{\mathbf{z}}_{\pm}(\Omega))$ are also close, the linearized stability of the branches is the same for these systems.

We also adapt the proof from~\cite{guckenheimer2013nonlinear} to prove statement (v). Assume that the reduced-order model~\eqref{eq:ode-reduced-slow-cartesian-leading} undergoes a Hopf bifurcation at $\Omega=\Omega_0$. The map $\mathbf{P}_2$ then has a locally unique curve of fixed points $\hat{\mathbf{x}}(\Omega)$ and the spectrum of $\partial_{\hat{\mathbf{x}}}\mathbf{P}_2(\hat{\mathbf{x}}(\Omega))$ does not contain 1 for $\Omega\approx\Omega_0$. Since $\mathbf{H}_0$ is $\mu$-close to $\mathbf{H}_2$, a nearby curve of fixed points $\hat{\mathbf{y}}(\Omega)=\hat{\mathbf{x}}(\Omega)+\mathcal{O}(\mu)$ exists for $\mathbf{P}_0$, and a pair of complex eigenvalues must pass through the unit circle for $\mathbf{P}_0$ near $\Omega=\Omega_0$. Likewise, since $\mathbf{H}_1$ is $\mu$-close to $\mathbf{H}_0$, a nearby curve of fixed points $\hat{\mathbf{z}}(\Omega)=\hat{\mathbf{y}}(\Omega)+\mathcal{O}(\mu)$ exists for $\mathbf{P}_1$, and a pair of complex eigenvalues must pass through the unit circle for $\mathbf{P}_1$ near $\Omega=\Omega_0$. It remains only to check the persistence of the absence of strong resonance condition in Neimark-Sacker bifurcation, which is easily satisfied because the three Jacobian matrices of the Poincar\'{e} maps are close and $\nu\ll1$.
\end{sloppypar}

\subsection{Proof of Theorem~\ref{theo-3}}
\label{sec:appendix-theo3-proof}
\begin{sloppypar}
Let $\boldsymbol{v}(t)$ be a periodic orbit of the reduced-order model~\eqref{eq:ode-reduced-slow-polar-leading} or~\eqref{eq:ode-reduced-slow-polar-leading} with period $T_\mathrm{s}$. Here $\boldsymbol{v}=(\boldsymbol{\rho},\boldsymbol{\theta})$ or $\boldsymbol{v}=(\boldsymbol{q}_\mathrm{s}^{\mathrm{R}},\boldsymbol{q}_\mathrm{s}^{\mathrm{I}})$, depending on the choice of coordinates. Recall that $\boldsymbol{p}(t)$ is the corresponding solution to the truncated reduced dynamics~\eqref{eq:red-dyn-leading}. Consider a cross section
\begin{equation}
\Sigma = \{(\boldsymbol{p},t)\in\mathbb{C}^l\times\mathbb{R}:\mathrm{mod}(t,T_\mathrm{s})=0\},
\end{equation}
and the Poincar\'e map $P_{\Sigma}:\Sigma\to\Sigma$
\begin{equation}
 (\boldsymbol{p}(kT_\mathrm{s}),kT_\mathrm{s})\mapsto(\boldsymbol{p}((k+1)T_\mathrm{s}),(k+1)T_\mathrm{s}), \quad k\in\mathbb{N}.
\end{equation}

We first prove the statement (i). If $\varrho=\frac{T}{T_\mathrm{s}}\in\mathbb{Q}$, there exists some $m_\mathrm{p}\in\mathbb{N}$ such that
\begin{gather}
(\boldsymbol{p}(k_\mathrm{p}T),kT)\mapsto (\boldsymbol{p}(k_\mathrm{p}T),(k+m_\mathrm{p})T),\nonumber\\
\forall\, k\in\mathbb{N},\,\,k_\mathrm{p}=1,\cdots,m_\mathrm{p}.
\end{gather}
The equation above indicates that there are $m_\mathrm{p}$ fixed points of the period-$m_\mathrm{p}T$ map. It follows that $\boldsymbol{p}(t)$ is periodic with period $m_\mathrm{p}T$. Now we give an explicit expression for $m_\mathrm{p}$. Given $\varrho\in\mathbb{Q}$, there exists unique two co-prime integers $m_{1,2}$ such that $\varrho=m_1/m_2$. Then we have $m_2T=m_1T_\mathrm{s}$ and we can choose $m_\mathrm{p}=m_2$.

If $\varrho\not\in\mathbb{Q}$, we deduce from~\eqref{eq:polar-form} and~\eqref{eq:cartesian-form} that the periodic orbit $\boldsymbol{v}(t)$ in the phase space coincides with the invariant closed curve of the Poincar\'e map and the solution $\boldsymbol{p}(t)$ is a quasi-periodic orbit. The two frequencies of this torus are $r_\mathrm{d}\Omega$ and $\omega_\mathrm{s}$, where $\omega_\mathrm{s}=2\pi/T_\mathrm{s}$.

We now move to the proof of the statement (ii). If $\varrho\in\mathbb{Q}$, the $m_\mathrm{p}$ fixed points of the Poincar\'e map have the same stability type as the periodic orbit $\boldsymbol{v}(t)$, then the stability type of the periodic orbit $\boldsymbol{p}(t)$ is the same as that of $\boldsymbol{v}(t)$. If $\varrho\not\in\mathbb{Q}$, the closed invariant intersection of the Poincar\'e map has the same stability type as the periodic orbit $\boldsymbol{v}(t)$. Then the stability type of the quasi-periodic orbit $\boldsymbol{p}(t)$ is the same as that of the $\boldsymbol{v}(t)$. The analysis based on fixed point algorithm~\cite{kim1996quasi} yields the same results.

Statement (iii) follows directly from the definition of quasi-periodic SN/PD/HB bifurcations of tori~\cite{kim1996quasi,vitolo2011quasi}.
\end{sloppypar}

\subsection{Event functions for bifurcations of periodic orbits}
\label{sec:event-bif-po}
Let $\{\hat{\lambda}_i\}_{i=1}^n$ be the eigenvalues of the monodromy matrix of a periodic orbit $\boldsymbol{u}(t)$ in an $n$-dimensional system. These eigenvalues are generally referred to as Floquet multipliers. If the system is autonomous, there always exists a Floquet multiplier equal to one. This multiplier is constant and hence is unrelated to bifurcations. Without loss of generality, therefore, we set $\hat{\lambda}_n\equiv1$ if the system is autonomous. One can detect bifurcation periodic orbits based on the evolution of the Floquet multipliers. Specifically
\begin{enumerate}[label=(\roman*)]
\item \emph{Saddle-node} (SN) bifurcation: a single Floquet multiplier crossing the unit
circle in the complex plane through the real number 1,
\item \emph{period-doubling} (PD) bifurcation: a single Floquet multiplier crossing the unit
circle in the complex plane through the real number $-1$,
\item \emph{Torus} (TR) bifurcation: a pair of complex conjugate Floquet multiplier crossing the unit circle in the complex plane at points away from a strong resonance.
\end{enumerate}
Event functions for detecting these bifurcation periodic orbits can be defined as~\cite{dankowicz2013recipes}: $\psi_\mathrm{SN}:\boldsymbol{u}\mapsto \prod_{i=1}^N(\hat{\lambda}_i-1)$, $\psi_\mathrm{PD}:\boldsymbol{u}\mapsto \prod_{i=1}^N(\hat{\lambda}_i+1)$ and
\begin{equation}
\psi_\mathrm{TR}:\boldsymbol{u}\mapsto\left\lbrace \begin{array}{cl}
      1   & N=1 \\
      \prod_{i=1}^N\prod_{j=1}^{i-1}(\hat{\lambda}_i\hat{\lambda}_j-1)   & N>1
    \end{array} \right.,
\end{equation}
where $N=n-1$ for autonomous systems and $N=n$ for non-autonomous systems. These event functions are utilized in the \texttt{po}-toolbox in \textsc{coco}.

In high-dimensional systems, such as our discrete beam and plate examples, one has $\psi_\mathrm{SN}\sim0$ and $\psi_\mathrm{TR}\sim0$ on the whole solution branch, which results in the failure of detecting these two types of bifurcation periodic orbits with the \texttt{po}-toolbox. An intuitive explanation is that most Floquet multipliers are inside the unit circle of the complex plane, resulting in $|\hat{\lambda}_i-1|<1$ and $|\hat{\lambda}_i\hat{\lambda}_j-1|<1$ in most cases. A trick to tackle this issue is selecting a subset of Floquet multipliers instead of all of them (see the stability analysis of periodic orbits for systems with delay and hence infinite many Floqute multipliers~\cite{ddebiftoolmanual,Kunt}). Specifically, we can choose the $n_\mathrm{b}$ eigenvalues which are closest to the unit circle of the complex plane. We have used this trick to detect TR bifurcations of periodic orbits in the von K\'arm\'an beam example.

\subsection{From periodic orbits to two-dimensional tori}
\label{sec:app-po2tor}
\begin{sloppypar}
Let $\boldsymbol{w}(t)$ be a periodic orbit in the leading-order SSM-reduced model~\eqref{eq:ode-reduced-slow-polar-leading} or~\eqref{eq:ode-reduced-slow-cartesian-leading}. Specifically, we have $\boldsymbol{w}=(\boldsymbol{\rho},\boldsymbol{\theta})$ for polar coordinates and $\boldsymbol{w}=(\boldsymbol{q}_\mathrm{s}^{\mathrm{R}},\boldsymbol{q}_\mathrm{s}^{\mathrm{I}})$ for Cartesian coordinates. Let $\Gamma$ be the closed curve corresponding to $\boldsymbol{w}(t)$ in the phase space of the reduced-order model. Defining $\boldsymbol{w}_\sigma(t) = \boldsymbol{w}(t+\sigma)$, we obtain that $\boldsymbol{w}_\sigma(t)$ is also a periodic orbit solution to the reduced-order model with the same period for all $\sigma\in\mathbb{R}$. This claim holds because the vector field of the reduced-order model is \emph{autonomous}. In addition, the closed curve corresponding to $\boldsymbol{w}_\sigma(t)$ is exactly $\Gamma$, independently of $\sigma$. We can represent $\Gamma$ with $\boldsymbol{w}_\sigma(0)$ with $\sigma\in[0,T_\mathrm{s}]$, where $T_\mathrm{s}$ is the period of $\boldsymbol{w}(t)$.

Let $\boldsymbol{p}_\sigma(t)$ be the solution of the reduced dynamics~\eqref{eq:red-dyn-leading} corresponding to $\boldsymbol{w}_\sigma(t)$. To be consistent with~\eqref{eq:p-to-qs}, we have
\begin{equation}
\label{eq:psigma-to-qsigma}
\boldsymbol{p}_\sigma(t)=(q_{1,\sigma}(t),\bar{q}_{1,\sigma}(t),\cdots,q_{m,\sigma}(t),\bar{q}_{m,\sigma}(t)),
\end{equation}
where
\begin{align}
\label{eq:qsigma}
& q_{i,\sigma}(t)=\rho_{i}(t+\sigma)e^{\mathrm{i}(\theta_{i}(t+\sigma)+r_i\Omega t)} \quad\text{or}\nonumber\\ & q_{i,\sigma}(t)=(q_{i,\mathrm{s}}^\mathrm{R}(t+\sigma)+\mathrm{i}q_{i,\mathrm{s}}^\mathrm{I}(t+\sigma))e^{\mathrm{i}r_i\Omega t},
\end{align}
depending on the choice of coordinate representation (cf.~\eqref{eq:polar-form} and~\eqref{eq:cartesian-form}). We define $\mathcal{R}(\boldsymbol{p}_\sigma(t))$ as
\begin{gather}
\mathcal{R}(\boldsymbol{p}_\sigma(t))=\begin{pmatrix}|q_{1,\sigma}(t)|\\\vdots\\|q_{m,\sigma}(t)|\\\arg(q_{1,\sigma}(t))\\\vdots\\\arg(q_{m,\sigma}(t))\end{pmatrix}\,\, \text{or}\,\,\begin{pmatrix}\mathrm{Re}(q_{1,\sigma}(t))\\\vdots\\\mathrm{Re}(q_{m,\sigma}(t))\\\mathrm{Im}(q_{1,\sigma}(t))\\\vdots\\\mathrm{Im}(q_{m,\sigma}(t))\end{pmatrix},
\end{gather}
depending on the coordinate representation for $\boldsymbol{w}$ is \emph{polar} or \emph{Cartesian}. One can directly verify that $\mathcal{R}(\boldsymbol{p}_\sigma(kT))\in\Gamma$ for all $\sigma\in[0,T_\mathrm{s}]$ and $k\in\mathbb{Z}$, where $T$ has been defined in Theorem~\ref{theo-po}. This suggests that $\Gamma$ is the closed invariant intersection of the period-$T$ section with $\mathcal{R}(\boldsymbol{p}_\sigma(t))$, independently of $\sigma$. Therefore, we can approximate the two-dimensional invariant torus with a collection of trajectories $\{\boldsymbol{p}_{\sigma_i}(t)\}_{i=1}^{N}$, $t\in[0,T]$. Specifically, we can choose a set of shifts $\{\sigma_i\}_{i=1}^N$ with $\sigma_i\in[0,T_\mathrm{s}]$ for $i=1,\cdots N$, such that $\Gamma$ is well approximated by the set of points $\{\boldsymbol{w}_{\sigma_i}(0)\}_{i=1}^N$. Each $\boldsymbol{p}_{\sigma_i}(t)$ stays on the torus and the set of terminal points after one period, namely, $\{\mathcal{R}(\boldsymbol{p}_{\sigma_i}(T))\}_{i=1}^N$ or equivalently, $\{\boldsymbol{w}_{\sigma_i}(T)\}_{i=1}^N$, also provides a good approximation to $\Gamma$.

In practical computations, the \texttt{po}-toolbox of \textsc{coco} provides a discrete approximation of $\boldsymbol{w}(t)$: $\{\boldsymbol{w}(t_i)\}_{i=1}^N$ with $t_1=0$, $t_i\in(0,T_\mathrm{s})$ for $i=2,\cdots,N-1$ and $t_N=T_\mathrm{s}$. It follows that we can select $\sigma_i=t_i$ for $i=1,\cdots,N$. For each $\boldsymbol{p}_{\sigma_i}(t)$ with $t\in[0,T]$, it can be discretized at a set of time points $\{\tau_j\}_{j=1}^{N_\mathrm{t}}$ with $\tau_j=(j-1)T/N_\mathrm{t}$ for $j=1,\cdots,N_\mathrm{t}$. The remaining question is to determine $\boldsymbol{p}_{\sigma_i}(\tau_j)$. It follows from~\eqref{eq:psigma-to-qsigma} and~\eqref{eq:qsigma} that we need to know $\boldsymbol{w}_{\sigma_i}(\tau_j)$. Given $\boldsymbol{w}_{\sigma_i}(\tau_j)=\boldsymbol{w}(\tau_j+\sigma_i)=\boldsymbol{w}(\tau_{ij})$ where $\tau_{ij}=\mathrm{mod}(\tau_j+\sigma_i,T_\mathrm{s})$, one can approximate $\boldsymbol{w}_{\sigma_i}(\tau_j)$ based on the \emph{interpolation} of $\{\boldsymbol{w}(t_i)\}_{i=1}^N$ with query point $\tau_{ij}$.
\end{sloppypar}

Now we are ready to map the approximated invariant torus in normal coordinates $\boldsymbol{p}$ to the invariant torus in physical coordinates $\boldsymbol{z}$. Specifically, we apply the map~\eqref{eq:ssm-decomp-leading} to each $\boldsymbol{p}_{\sigma_i}(t)$ to yield the corresponding trajectory in physical coordinates, which is referred to as $\boldsymbol{z}_{\sigma_i}(t)$ here. Then the collection of trajectories $\{\boldsymbol{z}_{\sigma_i}(t)\}_{i=1}^{N}$ with $t\in[0,T]$ can be used to approximate the invariant torus in physical coordinates.

\subsection{A brief introduction to the~\texttt{Tor} toolbox}
\label{sec:app-Tor}
This toolbox supports~\cite{li2020tor}:
\begin{itemize}
\item The switch from the continuation of periodic orbits to the continuation of invariant tori at TR bifurcation periodic orbits.
\item The continuation of invariant tori with fixed or free rotation number.
\item The continuation of invariant tori along the secondary branch passing through a branch point.
\end{itemize}
The toolbox does not provide the stability and bifurcation analysis of invariant tori in the current version.

In the \texttt{Tor} toolbox, an invariant torus is solved for as the solution to a partial differential equation (PDE) that involves two independent variables, corresponding to the two frequency components of the torus, and suitable phase conditions. Fourier expansion is used to discretize the unknown solution in one variable. The PDE is then discretized as a system of ordinary differential equations (ODEs), and the torus is obtained by solving a multisegment boundary-value problem (MBVP) with the system of ODEs and \emph{all-to-all} coupling boundary conditions. Let the number of harmonics in the Fourier expansion be $n_\mathrm{h}$, the number of segments is $2n_\mathrm{h}+1$. The MBVP is then solved with the \texttt{coll}-toolbox in \textsc{coco}. In other words, the collocation method is used to perform the discretization in the other independent variable. In the current release of \texttt{Tor}, $n_\mathrm{h}$ is specified by users and then fixed during continuation, while the mesh for the collocation can be adaptively changed in continuation due to the support of adaptability of the \texttt{coll}-toolbox.
\begin{sloppypar}
Let $\omega_{1,2}$ be the two frequency components of the invariant torus, and the corresponding phase angles be $\phi_{1,2}$ respectively. \texttt{Tor} utilizes  $2n_\mathrm{h}+1$ trajectories for approximating the torus. Let the collection of such trajectories be $\{\boldsymbol{u}_i(t)\}_{i=1}^{2n_\mathrm{h}+1}$ with $t\in[0,T_2]$, where $T_2=2\pi/\omega_2$. The initial points of these trajectories, $\{\boldsymbol{u}_i(0)\}_{i=1}^{2n_\mathrm{h}+1}$, are on a closed curve $C$, which is the invariant intersection of the period-$T_2$ map of the quasi-periodic orbit. In addition, the phase angle of $\boldsymbol{u}_i(0)$ on the curve $C$ is given by $\phi_{i,1}=(i-1)2\pi/(2n_\mathrm{h}+1)$ for $i=1,\cdots,2n_\mathrm{h}+1$. It follows that the final points of the set of trajectories, namely, $\{\boldsymbol{u}_i(T_2)\}_{i=1}^{2n_\mathrm{h}+1}$, are also on the closed curve $C$. The phase angles of these final points are obtained by applying a rigid rotation to the initial phase angles
\begin{equation}
\phi_{i,1}\mapsto\phi_{i,1}+\omega_1T_2=\phi_{i,1}+2\pi\varrho,\,\, i=1,\cdots,2n_\mathrm{h}+1,
\end{equation}
where $\varrho=\omega_1/\omega_2$. Therefore, the collection of trajectories $\{\boldsymbol{u}_i(t)\}_{i=1}^{2n_\mathrm{h}+1}$ indeed provides a good approximation to the torus.
\end{sloppypar}

\subsection{From two-dimensional tori to three-dimensional tori}
\label{sec:app-tor22tor3}
\begin{sloppypar}
Consider a two-dimensional invariant torus in the leading-order SSM reduced model~\eqref{eq:ode-reduced-slow-polar-leading} or~\eqref{eq:ode-reduced-slow-cartesian-leading}. This torus is approximated by a set of trajectories $\{\boldsymbol{w}_i(t)\}_{i=1}^{2n_\mathrm{h}+1}$ obtained by the \texttt{Tor} toolbox. Here the two frequencies of the torus are $\omega_{1,\mathrm{s}}$ and $\omega_{2,\mathrm{s}}$. At a TR bifurcation of periodic orbits, the unique invariant torus bifurcating from the periodic orbit has exactly the same geometry as the periodic orbit in the phase space. So the set of trajectories, $\{\boldsymbol{w}_i(t)\}_{i=1}^{2n_\mathrm{h}+1}$, degenerates into a single trajectory $\boldsymbol{w}(t)$, which is periodic with period $T_{2,\mathrm{s}}=2\pi/\omega_{2,s}$. Here we provide an alternative way to map the periodic orbit of the reduced-order model to the corresponding two-dimensional invariant torus in the normal coordinates $\boldsymbol{p}$. Unlike the method in section~\ref{sec:app-po2tor}, this approach approximates the invariant torus using a \emph{single} trajectory on the torus, instead of a set of trajectories. As we will see, the new approach can be applied to $\{\boldsymbol{w}_i(t)\}_{i=1}^{2n_\mathrm{h}+1}$ to yield an approximation to the three-dimensional torus in the normal coordinates $\boldsymbol{p}$.
\end{sloppypar}

Let $\phi_{2,\mathrm{s}}(k)$ be the phase angle of the intersection point $\mathcal{R}(\boldsymbol{p}(kT)$) on the closed curve $\Gamma$ (see section~\ref{sec:app-po2tor} for more details about the definition of $\mathcal{R}$ and $\Gamma$), it follows that
\begin{equation}
\phi_{2,\mathrm{s}}(k)=\phi_{2,\mathrm{s}}(0)+k\omega_{2,\mathrm{s}}T=\phi_{2,\mathrm{s}}(0)+2k\pi\varrho_{2,\mathrm{s}}
\end{equation}
for $k\in\mathbb{N}$, where $\varrho_{2,\mathrm{s}}=\omega_{2,s}T/(2\pi)$. Note that $T_{2,\mathrm{s}}\gg T$ in general because $T_{2,\mathrm{s}}$ is the period in the slow-phase dynamics. Let $k_{2,\mathrm{s}}=\mathrm{round}(T_{2,\mathrm{s}}/T)$, we have
\begin{align}
2k_{2,\mathrm{s}}\pi\varrho_{2,\mathrm{s}}& =\mathrm{round}\left(\frac{T_{2,\mathrm{s}}}{T}\right)\omega_{2,s}T \nonumber\\
&= 2\pi \cdot \mathrm{round}\left(\frac{T_{2,\mathrm{s}}}{T}\right)\frac{T}{T_{\mathrm{s},2}}\approx2\pi,
\end{align}
and then $\phi_{2,\mathrm{s}}(k_{2,\mathrm{s}})\approx\phi_{2,\mathrm{s}}(0)$, which implies that the set of intersection points $\{\phi_{2,\mathrm{s}}(0),\cdots,\phi_{2,\mathrm{s}}(k_{2,\mathrm{s}})\}$ provides a good approximation to the closed invariant curve $\Gamma$. It follows that the trajectory $\boldsymbol{p}(t)$ with $t\in[0,T_{2,\mathrm{s}}]$ covers the two-dimensional torus effectively.

\begin{sloppypar}
Let $\boldsymbol{p}_i(t)$ be the trajectory corresponding to $\boldsymbol{w}_i(t)$. Then $\{\boldsymbol{p}_i(t)\}_{i=1}^{2n_\mathrm{h}+1}$ covers the three-dimensional invariant torus effectively. Indeed, when $\{\boldsymbol{w}_i(t)\}_{i=1}^{2n_\mathrm{h}+1}$ degenerates into a single trajectory (periodic orbit), $\{\boldsymbol{p}_i(t)\}_{i=1}^{2n_\mathrm{h}+1}$ is reduced to a single trajectory that covers a two-dimensional invariant torus, which is the degenerate limit of a three-dimensional invariant torus. In general, $\{\boldsymbol{w}_i(t)\}_{i=1}^{2n_\mathrm{h}+1}$ approximates a \emph{nondegenerate} two-dimensional invariant torus in the reduced-order model, while the set of trajectories $\{\boldsymbol{p}_i(t)\}_{i=1}^{2n_\mathrm{h}+1}$ represents a \emph{tube} whose center line covers effectively a two-dimensional invariant torus, yielding a good approximation to the three-dimensional invariant torus.
\end{sloppypar}

In practical computations, the \texttt{Tor} toolbox provides a discrete approximation of $\boldsymbol{w}_i(t)$, namely, $\{\boldsymbol{w}_i(t_j)\}_{j=1}^N$ with $t_1=0$, $t_j\in(0,T_{2,\mathrm{s}})$ for $j=2,\cdots,N-1$ and $t_N=T_{2,\mathrm{s}}$. The time scale of $T_{2,\mathrm{s}}$ is much larger than the time scale of the excitation period $T$. In order to obtain an approximate three-dimensional torus with good accuracy, we need to perform the interpolation of $\boldsymbol{w}_i(t)$ at finer meshes to match the smallest time scale of $\boldsymbol{p}_i(t)$, i.e., $T$. Specifically, we can divide the interval $[0,T_{2,\mathrm{s}}]$ uniformly with $\mathrm{ceil}(T_{2,\mathrm{s}}/T)\cdot n_\mathrm{T}$ points, where $n_\mathrm{T}$ controls the number of time points in each period $T$. These $\mathrm{ceil}(T_{2,\mathrm{s}}/T)\cdot n_\mathrm{T}$ points are query points in the interpolation to yield a high-fidelity three-dimensional torus

\begin{sloppypar}
Now we are ready to map the approximated three-dimensional invariant torus in the normal coordinates $\boldsymbol{p}$ to the torus in physical coordinates $\boldsymbol{z}$. Specifically, we apply the map~\eqref{eq:ssm-decomp-leading} to each $\boldsymbol{p}_i(t)$ to yield the corresponding trajectory in physical coordinates, which is referred to as $\boldsymbol{z}_{i}(t)$ here. It follows that the collection of trajectories $\{\boldsymbol{z}_{i}(t)\}_{i=1}^{2n_\mathrm{h}+1}$ with $t\in[0,T_{2,\mathrm{s}}]$ can be used to approximate the three-dimensional invariant torus in physical coordinates. Likewise, the set of trajectories $\{\boldsymbol{z}_i(t)\}_{i=1}^{2n_\mathrm{h}+1}$ represents a \emph{tube} whose center line covers effectively a two-dimensional invariant torus.
\end{sloppypar}

\subsection{Supplementary analysis of example~\ref{sec:example1}}
\label{sec:ap-example1}
\subsubsection{Continuation of limit cycles}
\label{sec:ap-exap-limitcycle}
The results of one-dimensional continuation of limit cycles under varying $\Omega$ are shown in Fig.~\ref{fig:FRC_oneTwo_tori_TandRho}. The upper panel of this figure presents the projection of the continuation path onto the $(T_\mathrm{s},\Omega)$ plane, where $T_\mathrm{s}$ denotes the period of the limit cycle in the leading-order reduced dynamics~\eqref{eq:ode-reduced-slow-cartesian-leading}. This path is of complicated geometry with several SN and period-doubling (PD) bifurcation limit cycles along it.

\begin{figure*}[!ht]
\centering
\includegraphics[width=0.8\textwidth]{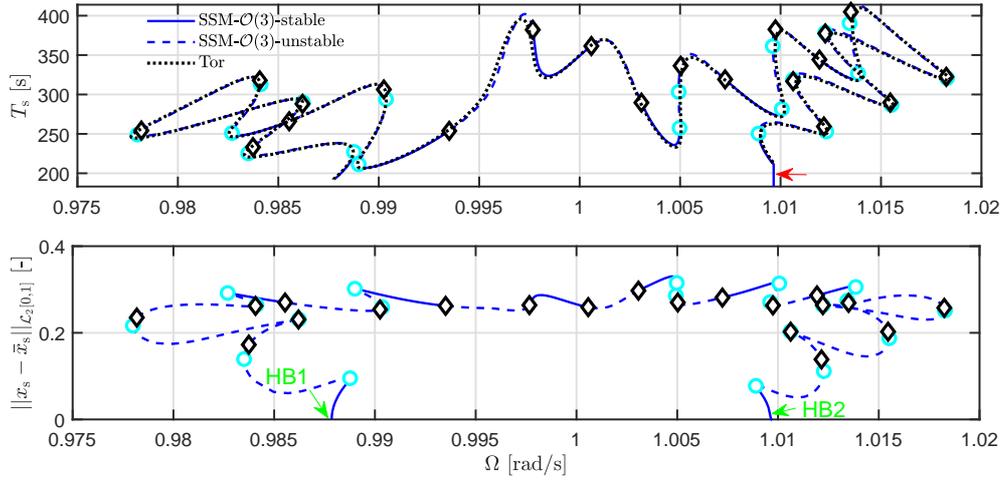}
\caption{Projections of the continuation path of limit cycles in the leading-order dynamics~\eqref{eq:ode-reduced-slow-cartesian-leading} of the coupled nonlinear oscillators~\eqref{eq:eom-two-os}. The upper and lower panels present the period and the size of the limit cycles as functions of $\Omega$, respectively. The solid lines here denote stable periodic orbits and dashed lines represent unstable ones. The circles and diamonds correspond to saddle-node (SN) and period-doubling (PD) bifurcation periodic orbits respectively. The two arrows in the lower panel indicate the two Hopf bifurcation (HB) points. The limit cycles here correspond to two-dimensional invariant tori in the original system~\eqref{eq:eom-two-os}. Such tori are also calculated directly using the \texttt{Tor} toolbox.}
\label{fig:FRC_oneTwo_tori_TandRho}
\end{figure*}

When a limit cycle is born out of a HB fixed point, the amplitude of such a limit cycle increases gradually from zero as $\Omega$ is varied from its critical value marking the HB point. To illustrate this evolution, we introduce a measure to quantify the size of a periodic orbit. This measure is essentially the deviation of the periodic orbit from its mean. Specifically, if $\boldsymbol{x}_\mathrm{s}(t)$ is a periodic solution with period $T_\mathrm{s}$, then we use the size measure
\begin{equation}
\label{eq:size-po}
||\boldsymbol{x}_\mathrm{s}-\bar{\boldsymbol{x}}_\mathrm{s}||_{\mathcal{L}_2[0,1]}^2:=\frac{1}{T_\mathrm{s}}\int_0^{T_\mathrm{s}} ||\boldsymbol{x}_\mathrm{s}(t)-\bar{\boldsymbol{x}}_\mathrm{s}||^2\mathrm{d}t,
\end{equation}
where the constant vector $\bar{\boldsymbol{x}}_\mathrm{s}$ is the time-average of the periodic signal, i.e.,
\begin{equation}
\bar{\boldsymbol{x}}_\mathrm{s}=\frac{1}{T_\mathrm{s}}\int_0^{T_\mathrm{s}} \boldsymbol{x}_\mathrm{s}(t)\mathrm{d} t.
\end{equation}
As seen in the lower panel of Fig.~\ref{fig:FRC_oneTwo_tori_TandRho}, the size of the limit cycles indeed increases from zero when $\Omega$ is increased from $\Omega=0.98785$ corresponding to the HB1 in Fig.~\ref{fig:oneTwo}. The orbit size evolves in a complicated way under varying $\Omega$ and returns to zero when $\Omega=1.0096$. At that point, the HB2 in Fig.~\ref{fig:oneTwo} is detected, with the limit cycle degenerating into a fixed point. As the continuation proceeds further from this critical periodic orbit, $\Omega$ is left constant while $T_\mathrm{s}$ decreases (cf. the red arrow in the upper panel of Fig.~\ref{fig:FRC_oneTwo_tori_TandRho}), and the periodic orbits are degenerate with zero size. The bifurcation curves in Fig.~\ref{fig:FRC_oneTwo_tori_TandRho} indicate the coexistence of several limit cycles for some values of $\Omega$. In addition, both stable and unstable limit cycles are observed.

\subsubsection{Two additional ways of validation via \texttt{Tor} toolbox}
\label{sec:ap-exap-twoval}
For a two-dimensional invariant torus, one frequency component is the excitation frequency $\Omega$, and the other one is \emph{a priori} unknown but can be solved for by the \texttt{Tor}-toolbox. If $\omega_\mathrm{s}$ is the unknown frequency component, which is a function of $\Omega$, then $T_\mathrm{s}={2\pi}/{\omega_\mathrm{s}}$ is exactly the period of the limit cycle in the leading-order dynamics~\eqref{eq:ode-reduced-slow-cartesian-leading}. Indeed, as can be seen in the upper panel of Fig.~\ref{fig:FRC_oneTwo_tori_TandRho}, the periods of limit cycles in the leading-order reduced dynamics obtained by SSM-analysis match well with the ones by \texttt{Tor} toolbox.

Our second additional validation method involves a Poinc\'are section. Consider the cross section 
\begin{equation}
\Sigma_{\boldsymbol{z}} = \{(\boldsymbol{z},t)\in\mathbb{R}^{4}\times\mathbb{R}:\mathrm{mod}(t,{2\pi}/{\Omega})=0\}.
\end{equation}
The intersection of a two-dimensional invariant torus with such a cross section is a closed curve, which is invariant under the period-${2\pi}/{\Omega}$ map generated by the flow. We can compare the closed curve obtained by SSM analysis and the \texttt{Tor} toolbox to demonstrate their agreement. We consider four sampled torus, namely, A-D in Fig.~\ref{fig:FRC_oneTwo_tori}, at which $\Omega$ is equal to 0.98, 0.99, 1.00, and 1.01, respectively. As seen in Fig.~\ref{fig:oneTwo-AD}, the intersection curves obtained by SSM analysis match well with the ones by \texttt{Tor} toolbox.

\begin{figure*}[!ht]
\centering
\includegraphics[width=0.45\textwidth]{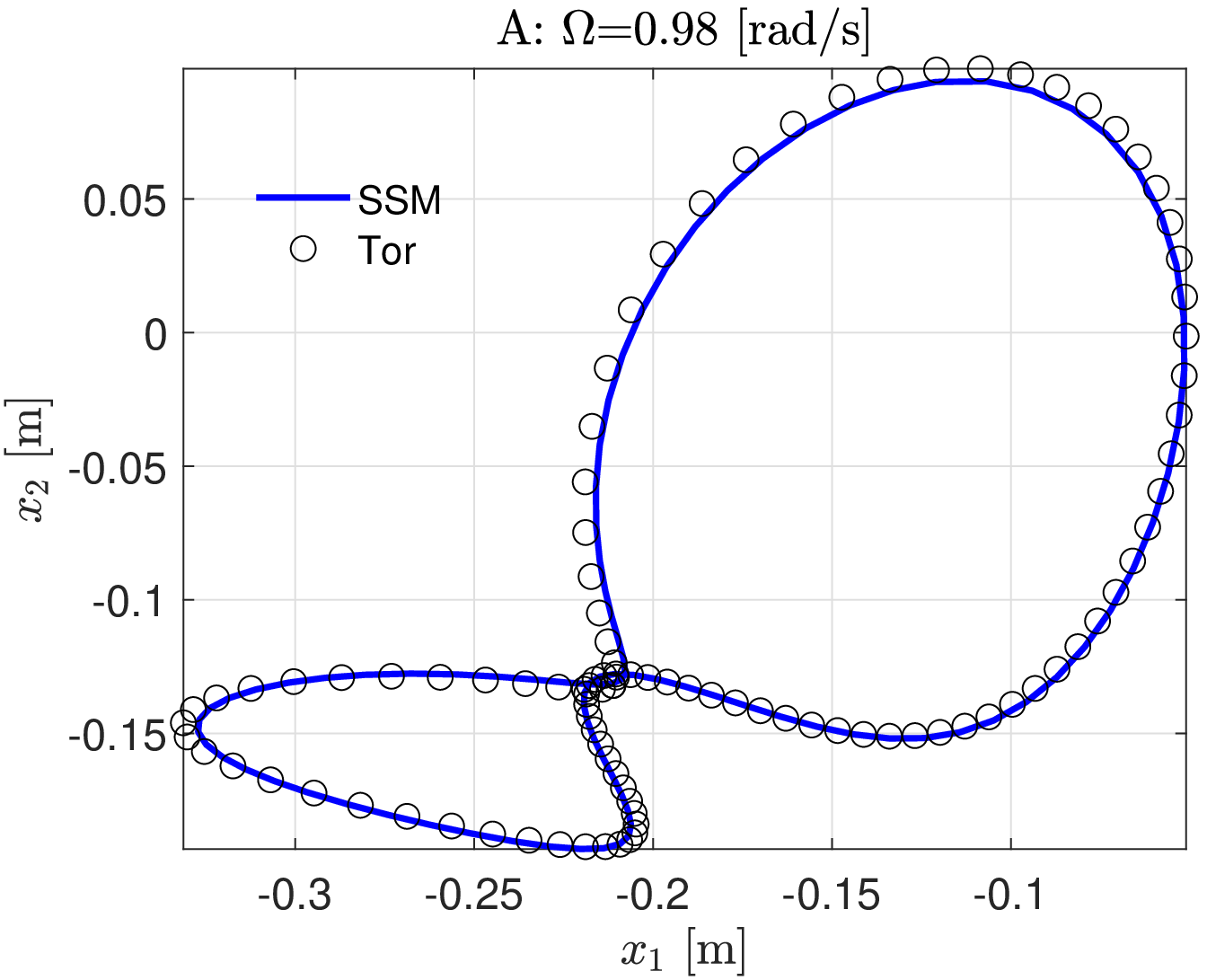}
\includegraphics[width=0.45\textwidth]{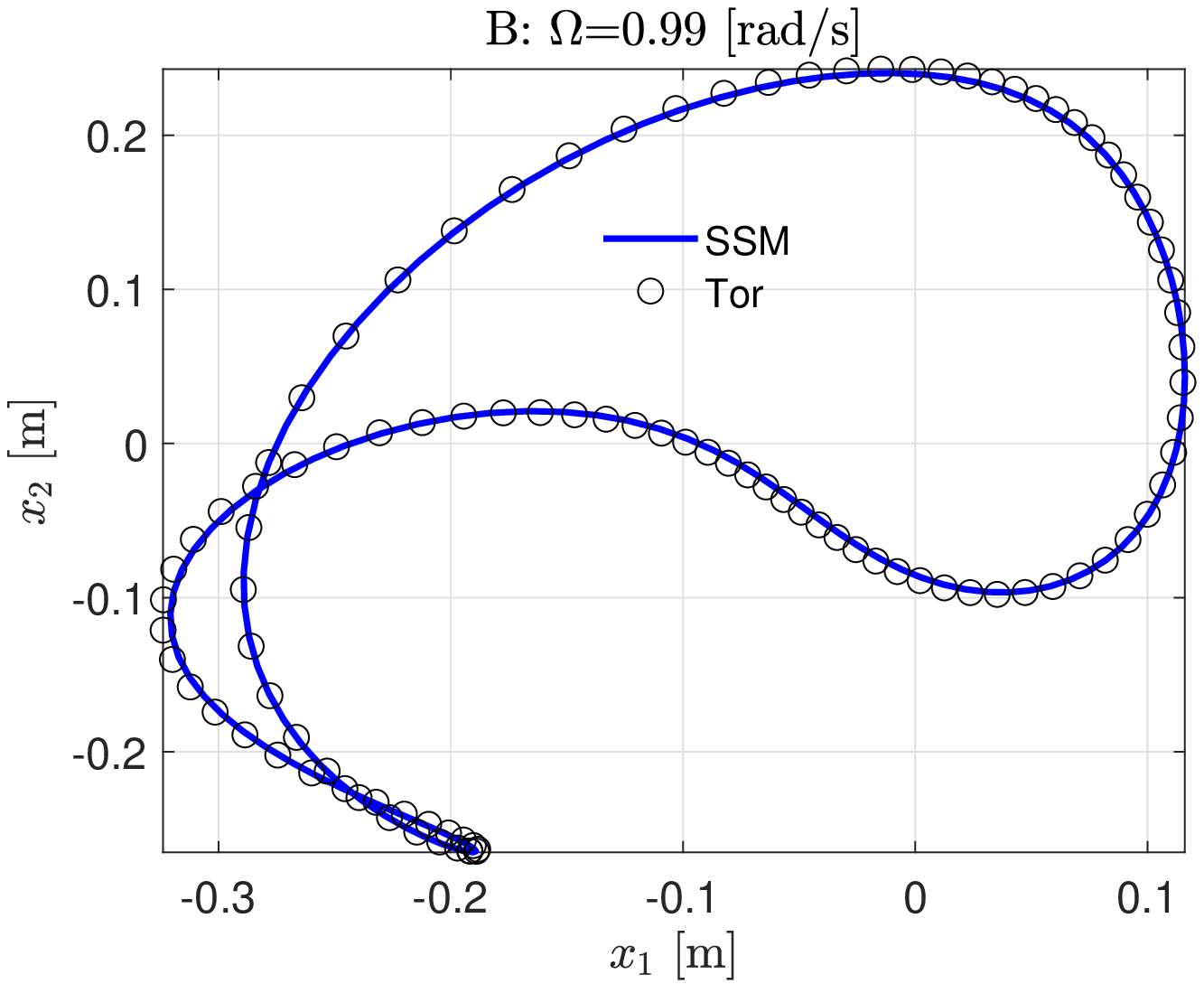}\\
\includegraphics[width=0.45\textwidth]{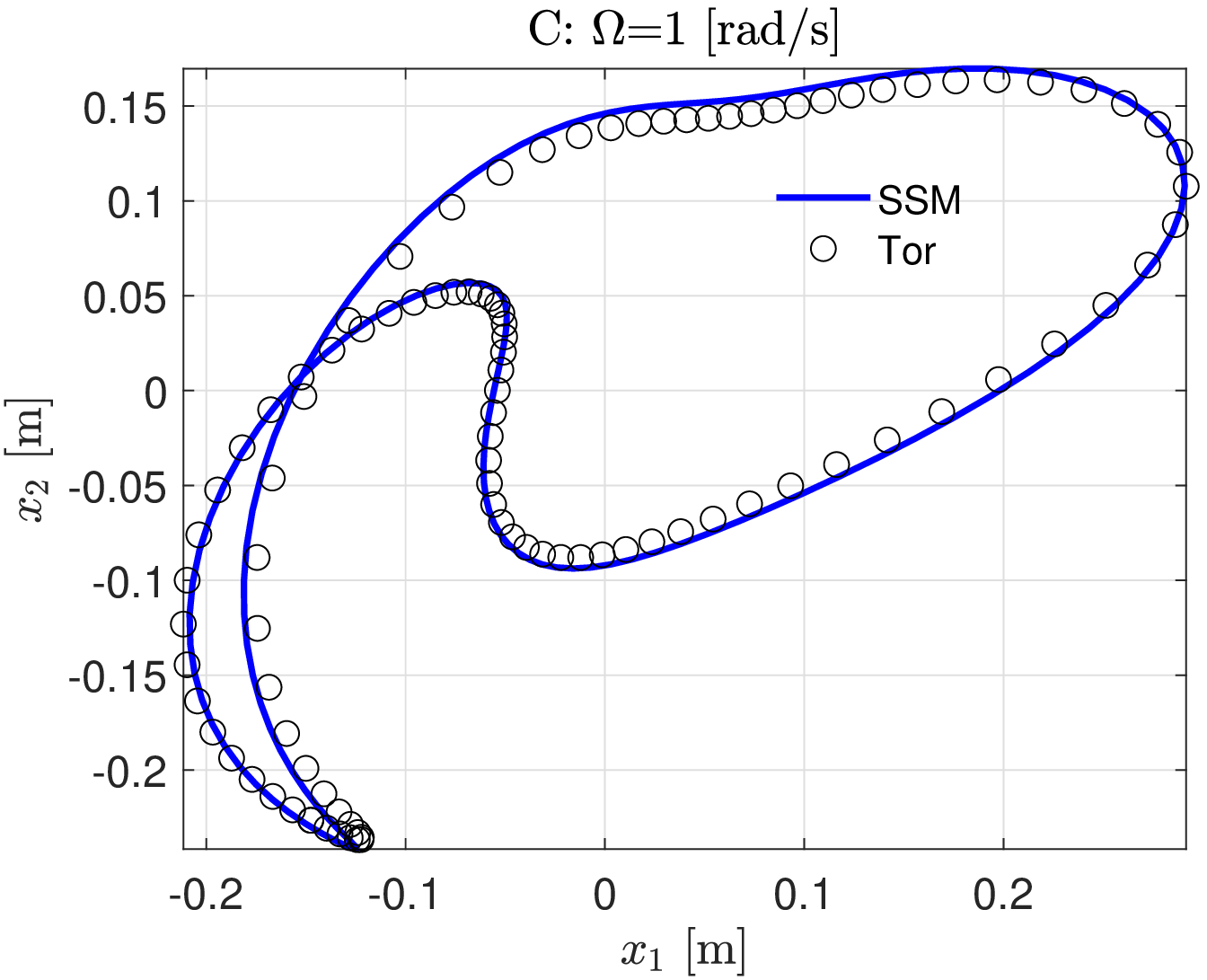}
\includegraphics[width=0.45\textwidth]{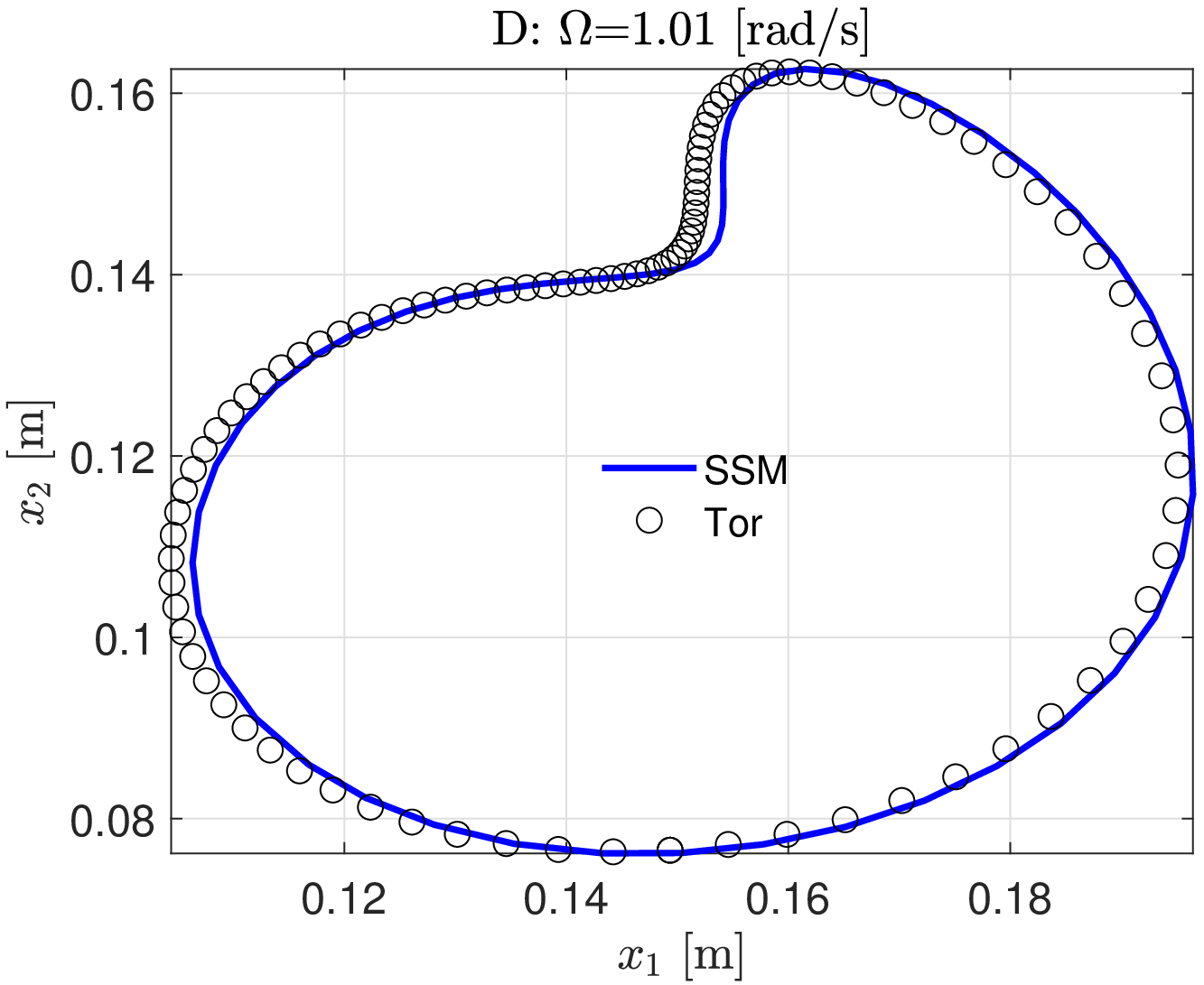}
\caption{Invariant intersections of the period-$2\pi/\Omega$ map of sampled tori of the coupled oscillators, approximated by limit cycles in the leading-order dynamics~\eqref{eq:ode-reduced-slow-cartesian-leading} from SSM analysis (solid lines) and direct computation using \texttt{Tor}-toolbox (cycles).}
\label{fig:oneTwo-AD}
\end{figure*}

\subsubsection{Validation of stability types}
\label{sec:ap-exap-stab}
To verify numerically the stability type of an invariant torus obtained from SSM analysis, we perform numerical integration of the original system with an initial point on the torus, and monitor whether the trajectory stays on the invariant torus or not. We take E and F in Fig.~\ref{fig:FRC_oneTwo_tori} as representative samples of stable and unstable invariant tori and then validate their stability types using numerical integration. Specifically, a point on the torus obtained from SSM analysis is selected as the initial condition of forward simulation applied to the original system. The simulation is performed using \texttt{ode45} in \textsc{matlab} over 300 excitation cycles (the final time of the simulation is $600\pi/\Omega$). As seen in Fig.~\ref{fig:oneTwo-EF3D}, the trajectory with the initial condition on the stable invariant torus E stays on the torus, while the trajectory with the initial condition on the unstable invariant torus F deviates from the toru, consistent with the prediction of SSM analysis. Indeed, the simulation time with 300 excitation cycles is long enough to infer the stability of the sampled invariant tori, as indicated by the maximal Lyapunov exponent (MLE) of the corresponding limit cycles in the leading-order SSM dynamics. Specifically, the MLE for the limit cycle corresponding to the stable invariant torus E is $\mu_\mathrm{E}=-0.0036$, yielding $e^{600\mu_\mathrm{E}/\Omega}=0.0012\ll1$. Likewise, the MLE for the limit cycle corresponding to the invariant torus F is $\mu_\mathrm{F}=0.0027$, yielding $e^{600\mu_\mathrm{F}/\Omega}=179.4491\gg1$.

\begin{figure*}[!ht]
\centering
\includegraphics[width=0.45\textwidth]{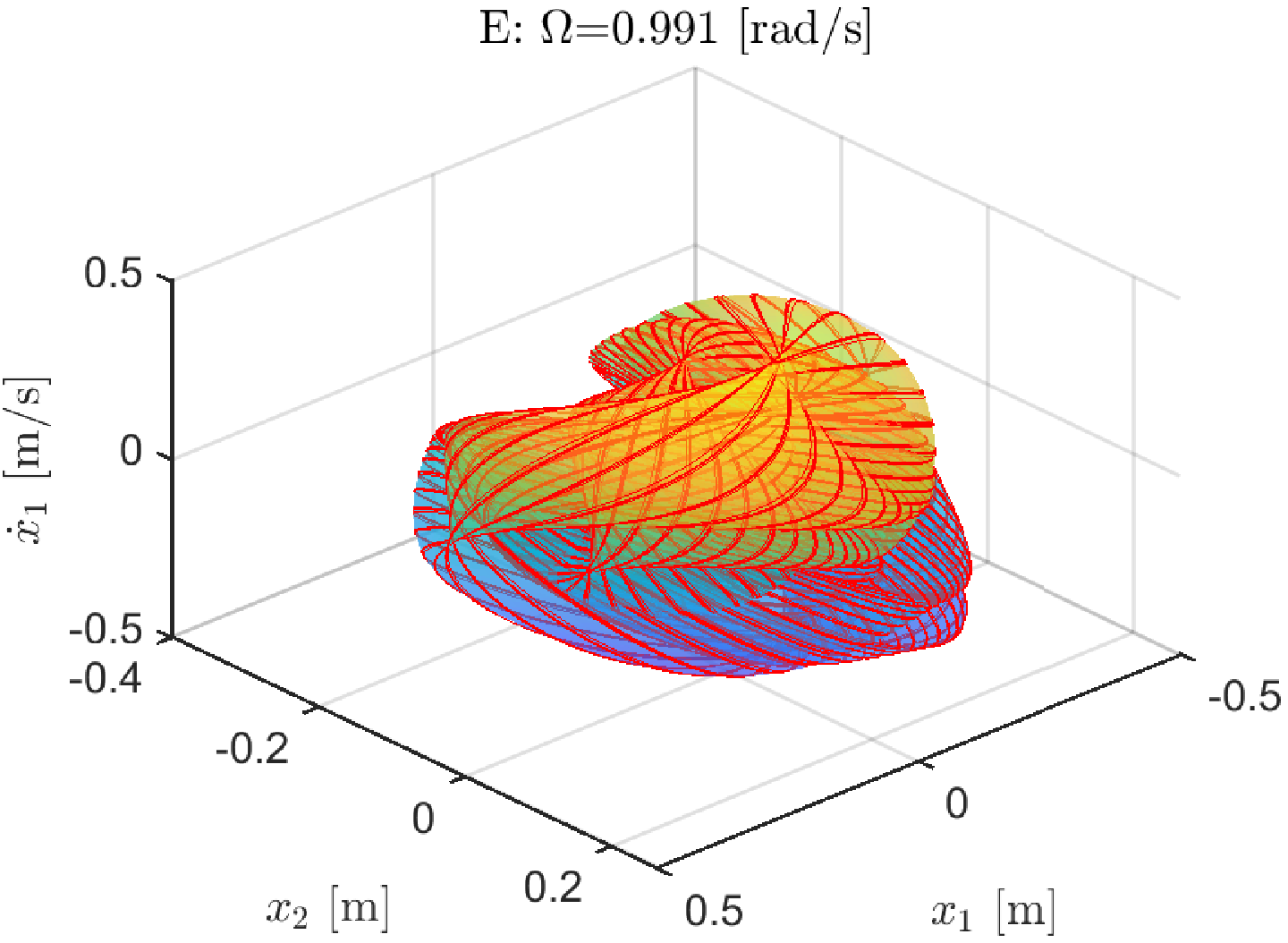}
\includegraphics[width=0.45\textwidth]{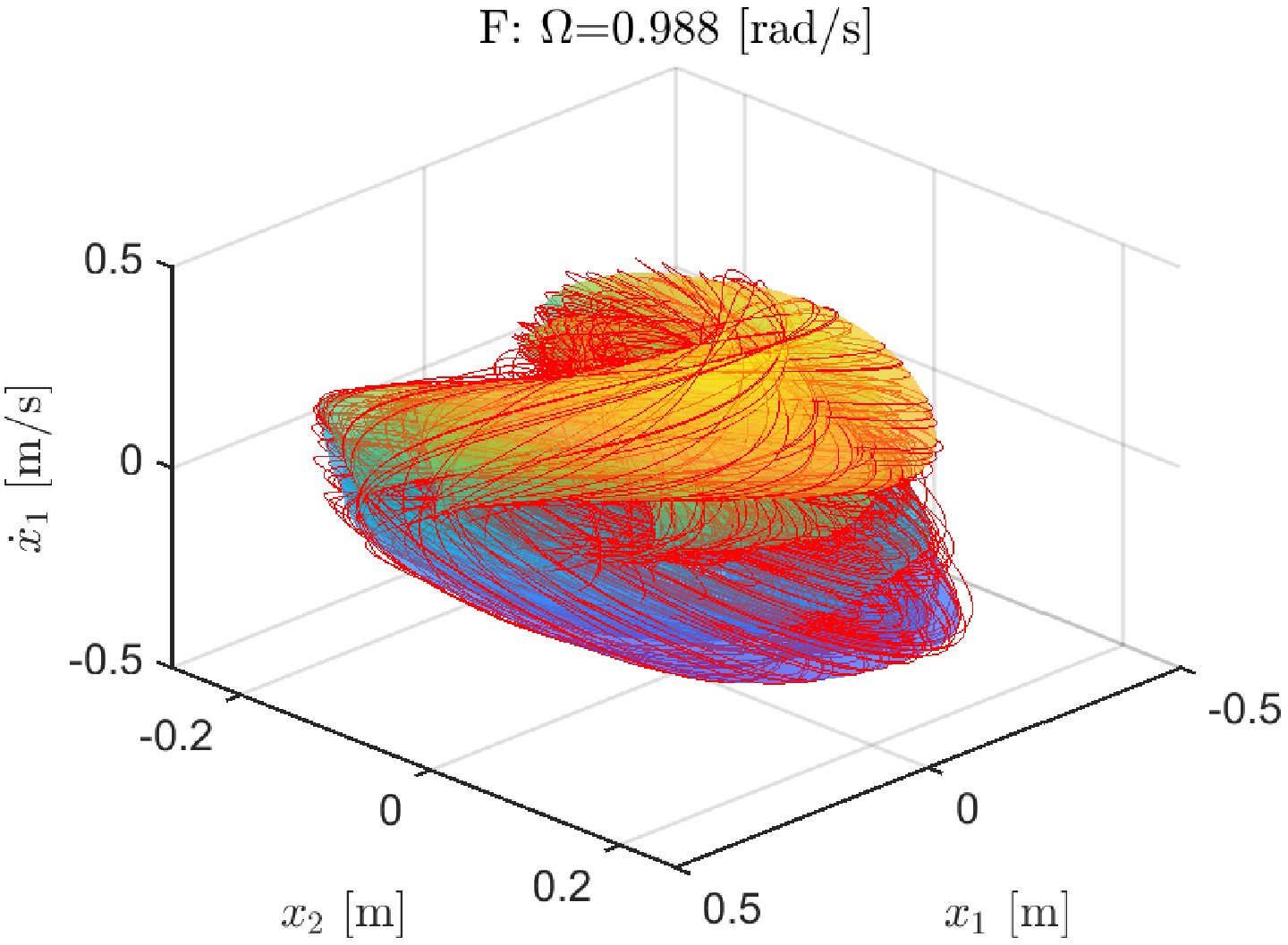}
\caption{Tori of coupled oscillators using SSM analysis (surface plot) and the trajectories by forward simulation of the original system~\eqref{eq:eom-two-os} with initial states on the computed tori (red lines).}
\label{fig:oneTwo-EF3D}
\end{figure*}

\subsection{Algorithm for selecting final time of direct numerical simulation}
\label{sec:app-tf-choice}
With a point on a torus obtained by the SSM reduction as the initial condition of the forward simulation of the full system, the trajectory obtained by the numerical integration should stay on the torus given that the torus is invariant and attracting. The remaining question is how long the forward simulation should be. We choose the final time $t_\mathrm{f}$ of the numerical integration such that i) $t_\mathrm{f}$ is large enough so that the trajectory covers the torus densely enough, and ii) $t_\mathrm{f}$ is short enough so that the computational cost remains reasonable. 

This method is based on information from the SSM analysis. Specifically, the SSM-reduced model reveals the two frequencies of such a torus: the external excitation frequency $\Omega$, and the {internal} frequency $\omega_\mathrm{s}\ll\Omega$ of the limit cycle in the reduced-order model. The rigid rotation of the phase angle $\phi_\mathrm{s}$ along the invariant intersection curve of the period-$2\pi/\Omega$ Poincar\'e section with the torus evolves as
\begin{equation}
\phi_\mathrm{s}\mapsto\phi_\mathrm{s}+\omega_s\frac{2\pi}{\Omega}=\phi_\mathrm{s}+2\pi\varrho_\mathrm{s},\quad \varrho_\mathrm{s}=\frac{\omega_\mathrm{s}}{\Omega}.
\end{equation}
Therefore, an approximate final time $t_{\mathrm{f}}$ for the forward simulation can be chosen as
\begin{equation}
\label{eq:final_time}
t_\mathrm{f} =M\cdot \mathrm{ceil}\left(\frac{1}{\varrho_\mathrm{s}}\right)\frac{2\pi}{\Omega}\approx M\frac{2\pi}{\omega_\mathrm{s}},\quad M\in\mathbb{N}_+.
\end{equation}
When $M=1$, the trajectory is already able to effectively cover the torus if it stays on the torus. With perturbation $\delta\boldsymbol{p}_0$ to the limit cycle, we have
\begin{equation}
    ||\delta\boldsymbol{p}(t_\mathrm{f})||\sim |\tilde{\lambda}_{\max}|^M ||\delta\boldsymbol{p}_0||, 
\end{equation}
where $|\tilde{\lambda}_{\max}|=\max\{\tilde{\lambda}_i\}_{i=1}^{2m-1}$. Here $\{\tilde{\lambda}_i\}_{i=1}^{2m-1}$ denotes the set of nontrivial Floquet multipliers of the limit cycle, with the trivial multiplier $\tilde{\lambda}_{2m}\equiv1$ excluded. Given that the limit cycle is stable, we have $|\tilde{\lambda}_{\max}|<1$. We can then choose a $\Delta\in(0,1)$ and determine $M$ as a function of $\Delta$ such that
\begin{align}
    M & =\underline{M}(\Delta):=\min\{k\in\mathbb{N}:|\tilde{\lambda}_{\max}|^k\leq\Delta\}\nonumber\\& =\mathrm{ceil}\left(\frac{\log\Delta}{\log|\tilde{\lambda}_{\max}|}\right).
\end{align}
Note that when $\Omega\to\Omega_\mathrm{HB}$, there exists a complex conjugate pair of multipliers approaching the unit circles of the complex plane ($\tilde{\lambda}\to e^{\pm\mathrm{i}\omega_\mathrm{s}}$). Then $|\tilde{\lambda}_{\max}|\to 1$ and $\underline{M}(\Delta)\to\infty$ for all $\Delta\in(0,1)$. Thus a more refined choice is
\begin{equation}
\label{eq:select-of-M}
    M = \min\{\underline{M}(\Delta),\bar{M}\}.
\end{equation}

From now on, we use the Newmark-beta method~\cite{geradin2014mechanical} to perform the numerical integration. The algorithm parameters are chosen as $\gamma=0.5+\alpha$ and $\beta=(1+\alpha)^2/4$ with $\alpha=0.005$. For each excitation cycle, 1000 integration steps are used with step size $2\pi/(1000\Omega)$. We can further calculate the amplitude of a quasi-periodic response using the trajectory obtained by the numerical integration. Specifically, we compute the amplitude of the trajectory in the horizon $[t_\mathrm{f}-t_\mathrm{f}/M,t_\mathrm{f}]$ to remove the artifacts of transient responses.

In example~\ref{sec:vonKarmanBeam}, we set $\Delta=0.001$ and $\bar{M}=200$ to ensure reliable verification. The $M(0.001)$ for the tori A and B shown in Fig.~\ref{fig:FRCs-vonBeam-physics} are 52 and 17, respectively, which are smaller than the threshold 200. In fact, $M=200$ occurs only for four sampled tori in the numerical simulations for each discrete model. In particular, these four tori are composed of the two tori whose $\Omega$ are closest to $\Omega_\mathrm{HB1}$, and the two tori whose $\Omega$ are closest to $\Omega_\mathrm{HB2}$. The calculation results for amplitudes of trajectories are plotted in Fig.~\ref{fig:FRCs-vonBeam-physics}.

In example~\ref{sec:vonKarmanPlate}, we have used the same algorithm to select the simulation time $t_\mathrm{f}$ for the tori appearing in Fig.~\ref{fig:vonKarmanPlate_tori} for the 1:1 resonant von K\'arm\'an plate. It follows from the upper panel of Fig.~\ref{fig:vonKarmanPlate_TandRho} that $1/\varrho_\mathrm{s}=\Omega/\omega_\mathrm{s}=\Omega T_\mathrm{s}/(2\pi)\sim200$. Equation~\eqref{eq:final_time} implies that the number of integration cycles is about 200 even for $M=1$.
With 1000 integration steps per cycle, the computational cost for each sampled $\Omega$ with $M=1$ is significant in this high-dimensional system. Indeed, each numerical integration here took more than 12 hours. So we simply set $M=1$ instead of using~\eqref{eq:select-of-M} to reduce computational load.

\section*{Declarations}
\section*{Conflict of interest}
The authors declare that they have no conflict of interest.

\section*{Data availability}
The data used to generate the numerical results included in this paper are available from the corresponding author on request.

\section*{Code availability}
\begin{sloppypar}
The code used to generate the numerical results included in this paper are available as part of the open-source \textsc{matlab} script SSMTool 2.2 at~\url{https://doi.org/10.5281/zenodo.6338831}.
\end{sloppypar}


\end{document}